\documentclass[A4]{amsart}

\usepackage[square,sort,comma,numbers]{natbib}
\usepackage{amsmath, amssymb, amstext, amsfonts, textcomp, amsxtra, amsbsy, amsgen, amsopn, amscd, mathrsfs, amsthm, latexsym, array}
\usepackage[textwidth=16cm,textheight=22cm,centering]{geometry}

\usepackage{color}

\usepackage[all]{xy}

\usepackage{bbm, dsfont}

\newtheorem{theorem}             {Theorem}  [section]
\newtheorem{definition} [theorem] {Definition}
\newtheorem{lemma}      [theorem]{Lemma}
\newtheorem{corollary}  [theorem]{Corollary}
\newtheorem{proposition}[theorem]{Proposition}
\newtheorem{remark} [theorem] {Remark}

\numberwithin{equation}{section} \everymath{\displaystyle}

\newcommand{\Cont}{{\rm C}}
\newcommand{\Aut}{\mathcal{A}}

\newcommand{\Sch}{\mathcal{S}}

\newcommand{\intL}{{\rm L}}
\newcommand{\Ht}{{\rm Ht}}

\newcommand{\Nr}{{\rm Nr}}
\newcommand{\Tr}{{\rm Tr}}

\newcommand{\gp}[1]{\mathbf{#1}}
\newcommand{\GL}{{\rm GL}}
\newcommand{\PGL}{{\rm PGL}}
\newcommand{\SL}{{\rm SL}}
\newcommand{\SO}{{\rm SO}}
\newcommand{\SU}{{\rm SU}}

\newcommand{\id}{\mathbbm{1}}

\newcommand{\ag}[1]{\mathbb{#1}}
\newcommand{\A}{\mathbb{A}}
\newcommand{\C}{\mathbb{C}}
\newcommand{\Q}{\mathbb{Q}}
\newcommand{\R}{\mathbb{R}}
\newcommand{\Z}{\mathbb{Z}}

\newcommand{\Mat}{{\rm M}}

\newcommand{\D}{\mathbf{D}}
\newcommand{\E}{\mathbf{E}}
\newcommand{\F}{\mathbf{F}}

\newcommand{\vo}{\mathfrak{o}}
\newcommand{\vO}{\mathcal{O}}
\newcommand{\vp}{\mathfrak{p}}

\newcommand{\Dis}{{\rm D}}

\newcommand{\norm}[1][\cdot]{\lvert #1 \rvert}
\newcommand{\extnorm}[1]{\left\lvert #1 \right\rvert}
\newcommand{\Norm}[1][\cdot]{\lVert #1 \rVert}

\newcommand{\Pairing}[2]{\langle #1, #2 \rangle}

\newcommand{\Four}[2][]{\mathfrak{F}_{#1} \left( #2 \right)}
\newcommand{\OFour}{\mathfrak{F}}

\newcommand{\rpL}{{\rm L}}
\newcommand{\rpR}{{\rm R}}
\newcommand{\Res}{{\rm Res}}
\newcommand{\Ind}{{\rm Ind}}
\newcommand{\Bas}{\mathcal{B}}

\newcommand{\Intw}{\mathcal{M}}

\newcommand{\Whi}{\mathcal{W}}

\newcommand{\Cond}{\mathbf{C}}

\newcommand{\fin}{{\rm fin}}
\newcommand{\eis}{{\rm E}}
\newcommand{\eisCst}{{\rm E}_{\gp{N}}}
\newcommand{\Reis}{\mathcal{E}}
\newcommand{\reg}{{\rm reg}}

\newcommand{\freg}{{\rm fr}}

\newcommand{\Ex}{\mathcal{E}{\rm x}}

\newcommand{\Zeta}{\mathrm{Z}}

\newcommand{\Shoulder}[3]{ #1 \left( #2 ; #3 \right) }
\newcommand{\Tree}[4]{ #1 \left( \frac{ #3 ; #4 }{ #2 } \right)}
\newcommand{\Sq}[5]{ #1 \left( \frac{ #4 ; #5 }{ #2 } \middle| #3 \right)}
\newcommand{\Vsum}{\sideset{}{^*} \sum}

\newcommand{\Vol}{{\rm Vol}}
\makeatletter

\newcommand{\Rmnum}[1]{\expandafter\@slowromancap\romannumeral #1@}
\makeatother

\title{On Motohashi's Formula}
\author{Han Wu}

\begin{document}

\begin{abstract}
	We complement and offer a new perspective of the proof of a Motohashi-type formula relating the fourth moment of $L$-functions for $\GL_1$ with the third moment of $L$-functions for $\GL_2$ over number fields, studied earlier by Michel-Venkatesh and Nelson. Our main tool is a new type of pre-trace formula with test functions on $\Mat_2(\A)$ instead of $\GL_2(\A)$, on whose spectral side the matrix coefficients are replaced by the standard Godement-Jacquet zeta integrals. This is also a generalization of Bruggeman-Motohashi's other proof of Motohashi's formula. We give a variation of our method in the case of division quaternion algebras instead of $\Mat_2$, yielding a new spectral reciprocity, for which we are not sure if it is within the period formalism given by Michel-Venkatesh. We also indicate a further possible generalization, which seems to be beyond what the period method can offer.
\end{abstract}

	\maketitle
	
	\tableofcontents

\section{Introduction}

	\subsection{History and New Perspective}

	In a series work (one collaborated with Iv\'ic) culminating in \cite{Mo93}, Motohashi established an explicit formula relating the fourth moment of Riemann zeta function with the cubic moment of $L$-functions related to modular forms (holomorphic and Maass forms and Eisenstein series) for the full modular group $\SL_2(\Z)$. In particular, the transform formula from the weight function on the fourth moment side to the weight function on the cubic moment side is described with explicit formula. A further extension to $\Q(i)$ can be found in \cite{BrM02}. In terms of automorphic representation theoretic language, this is a relation between the fourth moment of $L$-functions for $\GL_1$ and the third moment of $L$-functions for $\GL_2$ over $\Q$, hence has natural generalization in the level aspect. 
	
	Conrey-Iwaniec \cite{CI00} noticed that the continuous part of the cubic moment side becomes a sixth moment of $\GL_1$ $L$-functions. They studied Motohashi's formula in the level aspect for the inverse direction of the weight transform formula. They partially succeeded in doing this, and obtained a Weyl-type subconvexity in the level aspect for quadratic Dirichlet characters. This approach was recently extended by Petrow-Young \cite{PY19_CF} to cube-free level Dirchlet characters in the hybrid aspect, then to all Dirichlet characters in \cite{PY19_All}. Partial inversion in the archimedean aspect for Motohashi's original formula is also known for Maass forms by Ivi\'c \cite{Iv01}, and more recently for holomorphic forms by Frolenkov \cite{Fr20}, which also imply the relevant Weyl-type subconvexity. However, in these great achievements, an explicit transform formula of the weight functions in the inverse direction remains mysterious.
	
	In \cite[4.3.3]{MV06} and \cite[\S 4.5.3]{MV10} Michel-Venkatesh proposed a sketch of period approach to the Motohashi formula relating the fourth moment of the Riemann zeta function and some cubic moment of $L$-functions of automorphic forms for $\GL_2$, which exploits the Gan-Gross-Prasad conjecture (GGP) via two different paths of automorphic restriction of (a product of) Eisenstein series as follows
\begin{equation}
\begin{matrix}
	& & \GL_2 \times \GL_2 & & \\
	& \nearrow & & \nwarrow & \\
	\GL_1 \times \GL_1 & & & & \GL_2 \\
	& \nwarrow & & \nearrow & \\
	& & \GL_1 & &
\end{matrix}.
\label{MVGraph}
\end{equation}
	This \emph{Michel-Venkatesh sketch} is not a proof of a Motohashi-type formula because of serious convergence issue. Recently, Nelson \cite{Ne20} announced a solution to the convergence issue, a version of the transform formula of weight functions in the inverse direction which generalizes Conrey-Iwaniec's result over general number fields. In the current version, his inversion formula seems to work only for a special type of test functions, which is insufficient for our purpose in the subsequent work \cite{BFW21}.
	
	Inserting Godement's construction of Eisenstein series, we discover a new perspective of the Michel-Venkatesh and Nelson's (incomplete) treatment. Precisely, we realize Motohashi's formula as the equality of two different decompositions of a tempered distribution $\Theta(\lambda, \cdot)$ on $\Sch(\Mat_2(\A))$ (see Theorem \ref{Main}). This distribution satisfies a certain co-variance property under the action of $H(\A) = \GL_1(\A)^3$, offered by the following companion graph of (\ref{MVGraph})
\begin{equation}
\begin{matrix} 
	\GL_2 > \GL_1 & & \GL_1 < \GL_2 \\ 
	& {}^{\circlearrowleft} \Mat_2 ^{\circlearrowright} & \\
	& \circlearrowright & \\
	& \GL_1 & 
\end{matrix},
\label{CompGraph}
\end{equation}
	where the two actions of $\GL_2$ on $\Mat_2$ are the natural actions by multiplication and the bottom action of $\GL_1$ is the multiplication by the center of $\GL_2$. In such a way, we make appear the Godement-Jacquet zeta integrals on the cubic moment side to obtain the analytic continuation of all degenerate terms. The class of test functions is the whole $\Sch(\Mat_2(\A))$, and the degenerate terms are explicitly related to the main terms via concise formulas (see Proposition \ref{DegMainRel} below). We do not think that the theory of regularized integrals (as develped in \cite{Za81, MV10, Wu9}) alone can offer such results with the same depth as conveniently as we do here. The full power (Schwartz functions, invariance of distributions, formulas for degenerate terms) of the results obtained in this paper is important for the subsequent joint work of the author with Olga Balkanova and Dmitry Frolenkov \cite{BFW21}, where we generalize Petrow-Young's cube-free Weyl-type subconvex bound \cite{PY19_CF} over totally real number fields, via a proper inverse transform formula of weight functions at the real places. Incidentally, we point out that our method can be viewed as a generalization of Bruggeman-Motohashi's other proof \cite{BrM05} of Motohashi's formula.
	
	We note that a natural variation of our companion graph (\ref{CompGraph}) exists by replacing  $\GL_2$ with $\D^{\times}$, where $\D$ is a division quaternion algebra over $\F$, and $\GL_1 \times \GL_1$ with $\E^{\times}$, where $\E/\F$ is a quadratic field extension contained in $\D$. We thus replace $H$ by the $\F$-group $(\E^{\times} \times \E^{\times})/\F^{\times}$, which acts on $\D^{\times}$ via the following graph
\begin{equation}
\begin{matrix} 
	\D^{\times} > \E^{\times} & & \E^{\times} < \D^{\times} \\ 
	& {}^{\circlearrowleft} \D^{\times}{}^{\circlearrowright} & 
\end{matrix}.
\label{CompGraphVar}
\end{equation}
	Note that the previous case corresponds to taking $\E = \F \oplus \F$ as the split quadratic extension of $\F$. We study a special case where the characters of $H(\A)$ descend from $\Omega^{-1} \times \Omega$ of $\E^{\times}(\A) \times \E^{\times}(\A)$ for a Hecke character of $\E^{\times}(\F) \backslash \E^{\times}(\A) \simeq \E^{\times} \backslash \A_{\E}^{\times}$. Carrying out a similar proof of Theorem \ref{Main}, we obtain a new spectral reciprocity in Theorem \ref{MainVar} as the equality of two tempered distributions on $\Sch(\D(\A))$. We reserve the case for $\D=\Mat_2$ and for the same quadratic field extension $\E \subset \Mat_2$ for future study, since the relevant analysis is more fastidious (but not more complicated than the proof of Theorem \ref{Main}).
	
\begin{remark}
	Although we feel that the graph (\ref{CompGraphVar}) might be the companion graph of
\begin{equation}
\begin{matrix}
	& & \D^{\times} \times \D^{\times} & & \\
	& \nearrow & & \nwarrow & \\
	\E^{\times} \times \E^{\times} & & & & \D^{\times} \\
	& \nwarrow & & \nearrow & \\
	& & \E^{\times} & &
\end{matrix}.
\label{MVGraphVar}
\end{equation}
	as in the previous case, the precise period approach to our spectral reciprocity relation in Theorem \ref{MainVar} remains unclear (to us). The difficulty is to associate an automorphic form $\theta_{\Phi}(\Omega)$ on $\D^{\times}$ with a Hecke character $\Omega$ of $\E$ and a Schwartz function $\Phi \in \Sch(\A_{\E})$ in a canonical way. Some ideas in two special cases are as follows:
\begin{itemize} 
	\item[(1)] In the case $\Omega \mid_{\A^{\times}} = \id$, $\Omega$ can be viewed as an automorphic representation of $\F^{\times} \backslash \E^{\times} \simeq \E^1$, where $\E^1$ is the $\F$-subgroup of elements in $\E^{\times}$ with norm $1$. Then we have the theta lift $\Theta_1(\Omega)$ to the metaplectic group $\gp{Mp}$ via the Weil representation $r_1$ of $\E^1 \times \gp{Mp}$. We also have the theta lift $\Theta_2(\sigma)$ for any cuspidal automorphic representation $\sigma$ of $\gp{Mp}$ to any quaternionic group $\D^{\times}$ via the Weil representation $r_2$ of $\gp{PD}^{\times} \times \gp{Mp}$. One may expect that $\Theta_2(\Theta_1(\Omega))$ is well-defined, independent of the additive characters chosen in the above two Weil representations, and is the cuspidal representation containing $\theta_{\Phi}(\Omega)$.
	\item[(2)] In the case $\D = \Mat_2$, the automorphic representation containing $\theta_{\Phi}(\Omega)$ should be the automorphic induction of $\Omega$ to $\GL_2$.
\end{itemize}
	 However we do not know what role should $\Phi$ play in either case.
\end{remark}
		
	Since our proof relies on a new type of pre-trace formulas (see Theorem \ref{GJPreTrace} and Theorem \ref{GJPreTraceVar}), whose spectral side has the Godement-Jacquet zeta functions instead of the usual matrix coefficients, we point out that a further possible application of Theorem \ref{GJPreTrace} is to combine it with Jacquet-Zagier's approach of the trace formula for $\GL_2$ (see \cite{JZ87, Wu9}), to study the mixed moment of $L(1/2,\pi)L(1/2,\pi,\mathrm{Ad})$. This possibility does not seem to be within the scope of Michel-Venkatesh's period formalism. Such moment was recently studied in \cite{BBFR19, BBFR20}. It should be possible to generalize these results over general number fields in the light of this possible application of Theorem \ref{GJPreTrace}. We also reserve it for another paper.

	\subsection{Notations and Conventions}
	
	Throughout the paper, $\F$ is a (fixed) number field with ring of integers $\vo$. $V_{\F}$ denotes the set of places of $\F$. For any $v \in V_{\F}$ resp. $\vp \in V_{\F}, \vp < \infty$, $\F_v$ resp. $\vo_{\vp}$ is the completion of $\F$ resp. $\vo$ with respect to the absolute value $\norm_v$ corresponding to $v$ resp. $v=\vp$. $\ag{A} = \ag{A}_{\F}$ is the ring of adeles of $\F$, while $\ag{A}^{\times}$ denotes the group of ideles. We fix a section $s_{\F}$ of the adelic norm map $\norm_{\ag{A}}: \ag{A}^{\times} \to \ag{R}_+$, identifying $\R_+$ as a subgroup of $\A^{\times}$. Hence any character of $\R_+ \F^{\times} \backslash \A^{\times}$ is identified with a character of $\F^{\times} \backslash \A^{\times}$. We put the standard Tamagawa measure $dx = \sideset{}{_v} \prod dx_v$ on $\ag{A}$ resp. $d^{\times}x = \sideset{}{_v} \prod d^{\times}x_v$ on $\ag{A}^{\times}$. We recall their constructions. Let $\Tr = \Tr_{\ag{Q}}^{\F}$ be the trace map, extended to $\ag{A} \to \ag{A}_{\ag{Q}}$. Let $\psi_{\ag{Q}}$ be the additive character of $\ag{A}_{\ag{Q}}$ trivial on $\ag{Q}$, restricting to the infinite place as
	$$ \ag{Q}_{\infty} = \ag{R} \to \ag{C}^{(1)}, x \mapsto e^{2\pi i x}. $$
	We put $\psi = \psi_{\ag{Q}} \circ \Tr$, which decomposes as $\psi(x) = \sideset{}{_v} \prod \psi_v(x_v)$ for $x=(x_v)_v \in \ag{A}$. $dx_v$ is the additive Haar measure on $\F_v$, self-dual with respect to $\psi_v$. Precisely, if $\F_v = \ag{R}$, then $dx_v$ is the usual Lebesgue measure on $\ag{R}$; if $\F_v = \ag{C}$, then $dx_v$ is twice the usual Lebesgue measure on $\ag{C} \simeq \ag{R}^2$; if $v = \vp < \infty$ such that $\vo_{\vp}$ is the valuation ring of $\F_{\vp}$, then $dx_{\vp}$ gives $\vo_{\vp}$ the mass $\Dis_{\vp}^{-1/2}$, where $\Dis_{\vp} = \Dis(\F_{\vp})$ is the local component at $\vp$ of the discriminant $\Dis(\F)$ of $\F/\ag{Q}$ such that $\Dis(\F) = \sideset{}{_{\vp < \infty}} \prod \Dis_{\vp}$. Consequently, the quotient space $\F \backslash \ag{A}$ with the above measure quotient by the discrete measure on $\F$ admits the total mass $1$ \cite[Ch.\Rmnum{14} Prop.7]{Lan03}. Recall the local zeta-functions: if $\F_v = \ag{R}$, then $\zeta_v(s) = \Gamma_{\ag{R}}(s) = \pi^{-s/2} \Gamma(s/2)$; if $\F_v = \ag{C}$, then $\zeta_v(s) = \Gamma_{\ag{C}}(s) = (2\pi)^{1-s} \Gamma(s)$; if $v=\vp < \infty$ then $\zeta_{\vp}(s) = (1-q_{\vp}^{-s})^{-1}$, where $q_{\vp} := \Nr(\vp)$ is the cardinality of $\vo/\vp$. We then define
	$$ d^{\times} x_v := \zeta_v(1) \frac{dx_v}{\norm[x]_v}. $$
	In particular, $\Vol(\vo_{\vp}^{\times}, d^{\times}x_{\vp}) = \Vol(\vo_{\vp}, dx_{\vp})$ for $\vp < \infty$. Their product gives a measure of $\A^{\times}$. Equip $s_{\F}(\R_+)$ with the measure $dt/t$ on $\R_+$, where $dt$ is the (restriction of the) usual Lebesgue measure on $\R$, and $\F^{\times}$ with the counting measure. Then we have (see \cite[Ch.\Rmnum{14} Prop.13]{Lan03})
	$$ \Vol(\R_+ \F^{\times} \backslash \A^{\times}) = \zeta_{\F}^*, $$
	where $\zeta_{\F}^*$ is the residue at $1$ of the Dedekind zeta function $\zeta_{\F}(s)$.
	
	For $\vp < \infty$, let $\Sch(\F_{\vp}) = \Cont_c^{\infty}(\F_{\vp})$. We call $\Sch(\A) = \otimes_v' \Sch(\F_v)$ the space of Schwartz functions over $\A$. Let $\chi \in \widehat{\R_+ \F^{\times} \backslash \A^{\times}}$ viewed as a unitary Hecke character of $\F^{\times} \backslash \A^{\times}$, and let $f \in \Sch(\A)$. Tate's global zeta function is defined by
\begin{equation} 
	\Zeta(s,\chi,f) := \int_{\A^{\times}} f(x) \chi(x) \norm[x]_{\A}^s d^{\times}x. 
\label{TateInt}
\end{equation}
	These are integral representations of the complete $L$-function $\Lambda(s,\chi)$. They are absolutely convergent for $\Re s > 1$ with meromorphic continuation to $s \in \C$ and rapid decay for $\Cond(\chi \norm_{\A}^s) \to \infty$, satisfy the global functional equation
	$$ \Zeta(s,\chi,f) = \Zeta(1-s,\chi^{-1}, \OFour f) \quad \text{with} \quad \OFour f(x) := \int_{\A} f(y) \psi(-xy) dy, $$
	and admit possible simple poles at $s \in \{ 0,1 \}$ for $\chi = \mathbbm{1}$ with residues
\begin{equation} 
	\Res_{s=1} \Zeta(s,\mathbbm{1},f) = \zeta_{\F}^* \int_{\A} f(x) dx, \quad \Res_{s=0} \Zeta(s, \mathbbm{1}, f) = -\zeta_{\F}^* f(0). 
\label{TateIntRes}
\end{equation}
	
	For $R \in \left\{ \F_v \ \middle| \ v \in V_{\F} \right\} \cup \{ \A \}$, we define the following subgroups of $\GL_2(R)$
	$$ \gp{Z}(R) = \left\{ z(u) := \begin{pmatrix} u & 0 \\ 0 & u \end{pmatrix} \ \middle| \ u \in R^{\times} \right\}, \quad \gp{N}(R) = \left\{ n(x) := \begin{pmatrix} 1 & x \\ 0 & 1 \end{pmatrix} \ \middle| \ x \in R \right\}, $$
	$$ \gp{A}(R) = \left\{ a(y) := \begin{pmatrix} y & 0 \\ 0 & 1 \end{pmatrix} \ \middle| \ y \in R^{\times} \right\}, \quad \gp{A}(R)\gp{Z}(R) = \left\{ d(t_1,t_2) := \begin{pmatrix} t_1 & \\ & t_2 \end{pmatrix} \ \middle| \ t_1,t_2 \in R^{\times} \right\}, $$
and equip them with the Haar measures on $R^{\times}, R, R^{\times}, R^{\times} \times R^{\times}$ respectively. The long Weyl element $w$ in $\GL_2(R)$ is specified as 
\begin{equation}
	w := \begin{pmatrix} 0 & 1 \\ -1 & 0 \end{pmatrix}.
\label{LongWeyl}
\end{equation}
	The product $\gp{B} := \gp{Z} \gp{N} \gp{A}$ is a Borel subgroup of $\GL_2$. We pick the standard maximal compact subgroup $\gp{K} = \sideset{}{_v} \prod \gp{K}_v$ of $\GL_2(\ag{A})$ by defining
	$$ \gp{K}_v = \left\{ \begin{matrix} \SO_2(\ag{R}) & \text{if } \F_v = \ag{R} \\ \SU_2(\ag{C}) & \text{if } \F_v = \ag{C} \\ \GL_2(\vo_{\vp}) & \text{if } v = \vp < \infty \end{matrix} \right. , $$
and equip it with the Haar probability measure $d\kappa_v$. At every place $v$, we define a \emph{height function} 
	$$ \Ht_v: \gp{Z}(\F_v) \backslash \GL_2(\F_v) / \gp{K}_v \to \R_{>0}, \quad \begin{pmatrix} t_1 & x \\ 0 & t_2 \end{pmatrix} \kappa \mapsto \extnorm{\frac{t_1}{t_2}}_v. $$
	Their tensor product $\Ht(g) := \sideset{}{_v} \prod \Ht_v(g_v)$ for $g = (g_v)_v \in \GL_2(\A)$ is the height function on $\GL_2(\A)$.
	
	$\Mat_2(\A)$ admits an action of $\GL_2(\A) \times \GL_2(\A)$ which induces an action on $\Sch(\Mat_2(\A))$ by
	$$ \rpL_{g_1} \rpR_{g_2} \Psi(x) := \Psi \left( g_1^{-1} x g_2 \right), \quad \forall g_1,g_2 \in \GL_2(\A), \Psi \in \Sch(\Mat_2(\A)). $$
	This is a smooth Fr\'echet-representation.

	\subsection{Main Result}
	\label{MR}
	
	Let $R \in \left\{ \F_v \ \middle| \ v \in V_{\F} \right\} \cup \{ \A \}$. Consider the (right) action of $H(R) = R^{\times} \times R^{\times} \times R^{\times}$ on $\Mat_2(R)$ given for $x \in \Mat_2(R)$ and $t_j,z \in R^{\times}$ by
	$$ x^{(t_1,t_2,z)} := a(t_1)^{-1} x d(t_2z,z). $$

\begin{definition}
	A \emph{Motohashi distribution} is a tempered distribution $\Theta$ on $\Mat_2(R)$ satisfying the co-invariance property under the above action of $H(R) = (R^{\times})^3$
	$$ \Theta(h.\Psi) = \lambda(h) \Theta(\Psi), \quad \forall \Psi \in \Sch(\Mat_2(R)) \ \& \ h.\Psi(x) := \Psi(x^h), $$
	where $\lambda$ is a quasi-character of $H(R)$, called the \emph{parameter} of the distribution $\Theta$. According as $R = \A$ or $\F_v$, we say $\Theta$ is \emph{global} or \emph{local}. For a global Motohashi distribution, we require its parameter $\lambda$ to be \emph{automorphic}, i.e., to be a quasi-character of $H(\F) \backslash H(\A) \simeq (\F^{\times} \backslash \A^{\times})^3$.
\label{MotoDis}
\end{definition}

	The set of parameters $\lambda$ for Motohashi distributions (global or local) has the structure of a (trivial) complex vector bundle over the discrete set of a compact abelian group, with bundles isomorphic to $\C^3$. In the global case, we shall \emph{fix} three unitary characters $\chi_1, \chi_2, \omega$ of $\R_+ \F^{\times} \backslash \A^{\times}$, the relevant bundle at which is identified with $\C^3$ via a shift
	$$ \lambda = (s_1,s_2,s_0) \in \C^3 \quad \leftrightarrow \quad (t_1, t_2, z) \in (\A^{\times})^3 \mapsto \chi_1(t_1) \chi_2(t_2) \omega(z) \cdot \norm[t_1]_{\A}^{s_1-1} \norm[t_2]_{\A}^{s_2+1} \norm[z]_{\A}^{s_0+2}. $$


	To any $\Psi \in \Sch(\Mat_2(\A))$, we associate a kernel function (regarded as a tempered distribution)
	$$ \Shoulder{KK}{x}{y} := \sideset{}{_{0 \neq \xi \in \Mat_2(\F)}} \sum \Psi(x^{-1} \xi y), \qquad x,y \in \GL_2(\A). $$
	It is a smooth function on $\GL_2(\F) \backslash \GL_2(\A)$. Hence we can define
\begin{equation} 
\begin{matrix}
	\Shoulder{NK}{x}{y} := \int_{\F \backslash \A} \Shoulder{KK}{\begin{pmatrix} 1 & u \\ 0 & 1 \end{pmatrix}x}{y} du, \\
	\Shoulder{KN}{x}{y} := \int_{\F \backslash \A} \Shoulder{KK}{x}{\begin{pmatrix} 1 & u \\ 0 & 1 \end{pmatrix}y} du, \\
	\Shoulder{NN}{x}{y} := \int_{\F \backslash \A} \Shoulder{KN}{\begin{pmatrix} 1 & u \\ 0 & 1 \end{pmatrix}x}{y} du.
\end{matrix}
\label{PartialCstT}
\end{equation}
	From these functions, we form a smooth function on $\gp{B}(\F) \backslash \GL_2(\A) \times \gp{B}(\F) \backslash \GL_2(\A)$
\begin{equation} 
	\Shoulder{\Delta \Delta}{x}{y} := \Shoulder{KK}{x}{y} - \Shoulder{NK}{x}{y} - \Shoulder{KN}{x}{y} + \Shoulder{NN}{x}{y}. 
\label{KerToBeInt}
\end{equation}
	Note that we have chosen our notation so that ``$K$'' resp. ``$N$'' resp. ``$\Delta$'' stands for ``kernel'' resp. ``constant term'' resp. ``difference (non constant term)'', and have omitted the dependence on $\Psi$ for simplicity for this set of notations.
	
	The global Motohashi distribution we are interested in is
\begin{equation} 
	\Theta(\lambda, \Psi) := \int_{ ( \F^{\times} \backslash \A^{\times} )^3} \Shoulder{\Delta \Delta}{a(t_1)  }{d(t_2 z,z)} \omega(z) \norm[z]_{\A}^{s_0+2} \chi_1(t_1) \norm[t_1]_{\A}^{s_1-1} \chi_2(t_2) \norm[t_2]_{\A}^{s_2+1} d^{\times}z d^{\times}t_1 d^{\times}t_2,
\label{MainMD}
\end{equation}
	whose fundamental properties are summarized in the following proposition.
\begin{proposition}
	(1) The integral defining $\Theta(\lambda, \Psi)$ is absolutely convergent in
\begin{equation} 
	\lambda \in D := \left\{ (s_1, s_2, s_0) \in \C^3 \ \middle| \ \Re s_0 > 2, \  \Re s_1 - 1 - \frac{\Re s_0+1}{2} \ \& \ \Re s_2 + 1  - \frac{\Re s_0+1}{2} > 1 \right\}.
\label{InitialD}
\end{equation}
	(2) It has a meromorphic continuation to $\C^3$.
\label{GoalMotoDis}
\end{proposition}
\begin{remark}
	Throughout the paper, we hide the dependence on $\chi_1,\chi_2,\omega$ in the notation for simplicity.
\end{remark}

	We will establish Proposition \ref{GoalMotoDis} by two different methods. One method (Proposition \ref{SpecSideAC}) uses the spectral theory for $\GL_2$ while the other (Proposition \ref{GeomSideAC}) uses (partial) Poisson summation formula over $\Mat_2(\A)$. They allow us to write $\Theta(\lambda, \Psi)$ for $\lambda$ near $\vec{0}$ as
	$$ M_3(\lambda, \Psi) + \sideset{}{_{j=0}^3} \sum DS_j(\lambda, \Psi), \quad \text{resp.} \quad M_4(\lambda, \Psi) + \sideset{}{_{j=1}^8} \sum DG_j(\lambda, \Psi) - \sideset{}{_{j=4}^5} \sum DS_j(\lambda, \Psi), $$
where $M_3(\vec{0}, \cdot)$ resp. $M_4(\vec{0}, \cdot)$ is a Motohashi distribution of parameter $\vec{0}$ representing some cubic moment of $\GL_2$ $L$-functions resp. fourth moment of $\GL_1$ $L$-functions (in particular $M_3(\lambda, \Psi)$ resp. $M_4(\lambda, \Psi)$ is regular at $\vec{0}$), while $DS_j(\lambda, \cdot)$ resp. $DG_j(\lambda, \cdot)$ are \emph{degenerate distributions} in the sense that
\begin{itemize}
	\item[(1)] Each of them is supported, up to partial Fourier transforms, in a single $H(\A) \simeq (\A^{\times})^3$-orbit, a \emph{singularity} (in some sense) in the variety of $H(\A)$-orbits on $\Mat_2(\A)$;
	\item[(2)] Each is meromorphic at $\vec{0}$ but not necessarily regular there, and $DS := \sideset{}{_j} \sum DS_j$ resp. $DG := \sideset{}{_j} \sum DG_j$ is regular at $\vec{0}$;
	\item[(3)] $DS_j$ resp. $DG_j$ is expressible in terms of $M_3(\lambda, \Psi \mid \chi, s)$, the projection onto the principal series representation $\pi(\chi \norm_{\A}^s, \omega\chi^{-1}\norm_{\A}^{-s})$ of $M_3(\lambda, \Psi)$ resp. $M_4(\lambda, \Psi \mid \chi, s)$, the projection onto the Hecke character $\chi$ of $M_4(\lambda, \Psi)$.
\end{itemize} 
	Our main result is thus the equality at $\vec{0}$ of the two expressions of meromorphic continuation.

\begin{theorem}
	We get a Motohashi-type formula as the following equality of tempered distributions
	$$ M_3(\vec{0}, \Psi) + DS(\vec{0}, \Psi) = M_4(\vec{0}, \Psi) + DG(\vec{0}, \Psi), \quad \forall \Psi \in \Sch(\Mat_2(\A)). $$
\label{Main}
\end{theorem}

	We will not give full detail for the analysis of the degenerate terms in all cases. However, motivated by the application in \cite{BFW21}, we will give the answer in the case $\omega=\chi_1=\chi_2=\mathbbm{1}$ and establish
\begin{proposition}
	In the case $\omega = \chi_1 = \chi_2 = \mathbbm{1}$, we have
	$$ DG(\vec{0}, \Psi) = \frac{1}{\zeta_{\F}^*} \left\{ \Res_{s=1} M_4(\vec{0}, \Psi \mid \mathbbm{1},s) - \Res_{s=0} M_4(\vec{0}, \Psi \mid \mathbbm{1},s) \right\}, $$
	$$ DS(\vec{0}, \Psi) = \frac{1}{\zeta_{\F}^*} \Res_{s=\frac{1}{2}} M_3(\vec{0}, \Psi \mid \id, s). $$
\label{DegMainRel}
\end{proposition}

\begin{remark}
	Our method works for all other cases, and is simpler than the original work of Motohashi.
\end{remark}

\begin{remark}
	The formula of $DS(\vec{0}, \Psi)$ is new. In fact, in Motohashi's original work \cite[Theorem 4.2]{Mo97}, a precise explicit formula for this term (the first component of $\mathcal{L}_{2,r}(g)$) is incomplete, and has not so far been completed in the literature.
\end{remark}

	\subsection{Compact Variation}

	We re-interpret our graph of invariance (\ref{CompGraph}) by viewing the split torus $\GL_1 \times \GL_1$ as $\E^{\times}$ for the split algebra $\E \simeq \F \oplus \F$. We regard $\E$ resp. $\E^{\times}$ as $\F$-algebra resp. $\F$-group embedded in the split quaternion algebra $\Mat_2$. Then the group $H$ is canonically isomorphic to $(\E^{\times} \times \E^{\times})/\F^{\times}$, where $\F^{\times}$ is the group of centers of $\GL_2$ diagonally embedded in $\E^{\times} \times \E^{\times}$. 
	
	We consider the case where $\E$ is non-split, and $\Mat_2$ is replaced by a division quaternion algebra $\D$ over $\F$ containing $\E$. The group $H = (\E^{\times} \times \E^{\times})/\F^{\times}$ acts on $\D$ (from right) by
	$$ H \times \D \to \D, \quad (t_1,t_2) \times x \mapsto t_1^{-1}xt_2. $$
	Let $\Omega$ be a (quasi-)character of $\E^{\times} \backslash \A_{\E}^{\times}$, whose restriction to $\F^{\times} \backslash \A^{\times}$ is $\omega$. It defines a (quasi-)character of $H(\F) \backslash H(\A)$ by
\begin{equation} 
	H(\A) \to \C^{\times}, (t_1,t_2) \mapsto \Omega(t_1)^{-1} \Omega(t_2). 
\label{CharH}
\end{equation}
	Form the kernel function
	$$ KK_{\D}(x,y) := \sum_{\xi \in \D^{\times}(\F)} \rpL_x \rpR_y \Psi(\xi), \quad \Psi \in \Sch(\D(\A)). $$
	The tempered distributions we are interested in are
\begin{equation}
	\Theta_{\D}(\Omega, \Psi) := \int_{(\E^{\times} \A^{\times} \backslash \A_{\E}^{\times})^2} \left( \int_{\F^{\times} \backslash \A^{\times}} KK_{\D}(t_1,t_2 z) \omega(z) d^{\times}z  \right) \Omega(t_1^{-1}t_2) d^{\times}t_1 d^{\times}t_2.
\label{MainMDVar}
\end{equation}
	From now on, we fix $\Omega$ a unitary character of $\E^{\times} \R_+ \backslash \A_{\E}^{\times}$ and write
	$$ \Theta_{\D}(s_0, \Psi) := \Theta_{\D}(\Omega \norm_{\A_{\E}}^{s_0+1}, \Psi). $$
	For simplicity, we assume $\Omega \neq \id$ is not trivial.
\begin{proposition}
	The integral defining $\Theta_{\D}(s_0, \Psi)$ is absolutely convergent for $\Re s_0 > 1$, and admits a meromorphic continuation to $s_0 \in \C$.
\label{GoalMotoDisVar}
\end{proposition}

	Similarly to Proposition \ref{GoalMotoDis}, we will establish Proposition \ref{GoalMotoDisVar} with two different methods. They allow us to write $\Theta_{\D}(s_0,\Psi)$ for $\norm[s_0] < 1/4$ as
	$$ M(s_0,\Psi) + DS(s_0,\Psi) \quad \text{resp.} \quad M_2(s_0,\Psi) + \sum_{j=1}^4 DG_j(s_0,\Psi). $$
	Every term is regular at $s_0 = 0$. $M(0,\Psi)$ represents some mixed moment of $L(1/2,\pi) L(1/2, \pi_{\E} \otimes \Omega^{-1})$ as $\pi$ traverses cuspidal representations of $\D^{\times}$ with central character $\omega$. $M_2(0,\Psi)$ represents the mean value of $L(1/2-i\tau,\Omega\Xi^{-1})L(1/2+i\tau,\Xi)$ as $\Xi$ traverses the unitary characters of $[\E^1 \backslash \E^{\times}]$ and $\tau \in \R$.

\begin{theorem}
	We get an equality of tempered distributions
	$$ M(0, \Psi) + DS(0, \Psi) = M_2(0, \Psi) + \sum_{j=1}^4 DG_j(0, \Psi), \quad \forall \Psi \in \Sch(\D(\A)). $$
\label{MainVar}
\end{theorem}

	\subsection{List of Tempered Distributions}
	\label{DegTerms}

		\subsubsection{Geometric Main Term}
		
	We can identify $\Mat_2(\A)$ with $\A^4$ via
	$$ \Mat_2(\A) \simeq \A^4, \quad \begin{pmatrix} x_1 & x_2 \\ x_3 & x_4 \end{pmatrix} \mapsto (x_1,x_2,x_3,x_4). $$
	In particular, we transport $\OFour_j$, the partial Fourier transform on $\A^4$ with respect to the $j$-th variable, to $\Sch(\Mat_2(\A))$. Tate's integral (see (\ref{TateInt})) for $\Sch(\A)$ admits obvious multi-dimensional generalization. In particular, it generalizes to $\Sch(\Mat_2(\A))$. Define (recall we have identified $\lambda$ with $(s_1,s_2,s_0) \in \C^3$)
	$$ \eta_1 = \omega^{-1}\chi_1^{-1}\chi_2, \quad \eta_2 = \omega^{-1}\chi_2, \quad \eta_3 = \omega^{-1}\chi_1^{-1}, \quad \eta_4 = \mathbbm{1}, $$
	$$ s_1' = s_2-s_0-s_1, \quad s_2' = s_2-s_0, \quad s_3' = -s_0-s_1, \quad s_4' = 0. $$
	We form a tempered distribution for $\lambda$ near $\vec{0}$ using four dimensional Tate's integrals (see (\ref{TateInt}))
\begin{equation}
	M_4(\lambda, \Psi) = \frac{1}{\zeta_{\F}^*} \sideset{}{_{\chi}} \sum \int_{\Re s = 1/2} \int_{(\A^{\times})^4} \OFour_2 \OFour_3 \Psi \begin{pmatrix} x_1 & x_2 \\ x_3 & x_4 \end{pmatrix} \left( \prod_{i=1}^4 \eta_i \chi(x_i) \norm[x_i]_{\A}^{s+s_i'} d^{\times}x_i \right) \frac{ds}{2\pi i},
\label{II_10AC}
\end{equation}
	which is absolutely convergent by the rapid decay recalled after (\ref{TateInt}). Note that the inner integral
\begin{equation} 
	M_4(\lambda, \Psi \mid \chi,s)  = \int_{(\A^{\times})^4} \OFour_2 \OFour_3 \Psi \begin{pmatrix} x_1 & x_2 \\ x_3 & x_4 \end{pmatrix} \left( \prod_{i=1}^4 \eta_i \chi(x_i) \norm[x_i]_{\A}^{s+s_i'} d^{\times}x_i \right)
\label{M4Proj}
\end{equation}
	naturally defines a meromorphic function in $(\lambda,s) \in \C^4$.
	

		\subsubsection{Spectral Main Term}
	
	Let $\pi$ be a cuspidal representation in the (right) regular representation $\rpR_{\omega}$ of $\GL_2(\A)$ on $V_{\omega} = \intL^2(\GL_2, \omega)$ given by
\begin{equation} 
	\intL^2(\GL_2, \omega) := \left\{ \varphi: \GL_2(\F) \backslash \GL_2(\A) \to \C \ \middle| \ \begin{matrix} \varphi(g z(u)) = \omega(u) \varphi(g), \quad \forall g \in \GL_2(\A), u \in \A^{\times} \\ \int_{[\PGL_2]} \norm[\varphi(g)]^2 dg < \infty \end{matrix} \right\}, 
\label{AutoSwC}
\end{equation}
	where we have written
	$$ [\PGL_2] := \gp{Z}(\A) \GL_2(\F) \backslash \GL_2(\A). $$
	If $V_{\pi}$ denotes the underlying Hilbert space of $\pi$ with subspace $V_{\pi}^{\infty}$ of smooth vectors realized as smooth functions on $\GL_2(\A)$, then we have the Hecke-Jacquet-Langlands zeta integrals (for all $s \in \C$)
	$$ \Zeta(s, \varphi, \chi) := \int_{\F^{\times} \backslash \A^{\times}} \varphi(a(t)) \chi(t) \norm[t]_{\A}^{s-\frac{1}{2}} d^{\times}t, \quad \forall \varphi \in V_{\pi}^{\infty}, \chi \in \widehat{\R_+ \F^{\times} \backslash \A^{\times}}. $$
	
	Similarly, let $\chi$ be a unitary character of $\R_+ \F^{\times} \backslash \A^{\times}$, we can associate the principal series representation $\pi_s := \pi(\chi \norm_{\A}^s, \omega \chi^{-1} \norm_{\A}^{-s})$, whose underlying Hilbert space is
	$$ V_{\chi,\omega\chi^{-1}} := \left\{ f: \gp{K} \to \C \ \middle| \ \begin{matrix} f \left( \begin{pmatrix} t_1 & u \\ 0 & t_2 \end{pmatrix} g \right) = \chi(t_1)\omega\chi^{-1}(t_2) f(g) \\ \int_{\gp{K}} \norm[f(\kappa)]^2 d\kappa < \infty \end{matrix} \right\}. $$
	Note that this space is the common one for all $s \in \C$, and the subspaces of smooth vectors are identical for any $s$, although the actions $\pi_s$ differ as $s$ varies. In particular, $\pi_s$ is unitary for $s \in i\R$. To any smooth vector $e \in V_{\chi,\omega\chi^{-1}}^{\infty}$ is associated a flat section $e(s)$ in $\pi_s$, from which we construct an Eisenstein series
	$$ \eis(s,e)(g) = \eis(e(s))(g) := \sideset{}{_{\gamma \in \gp{B}(\F) \backslash \GL_2(\F)}} \sum e(s)(\gamma g), $$
	convergent for $\Re s \gg 1$ and admitting meromorphic continuation to $s \in \C$. Defining the constant term
	$$ \eisCst(s,e)(g) := \int_{\F \backslash \A} \eis(s,e)(n(x)g) dx, $$
	we have the extended Hecke-Jacquet-Langlands zeta integrals
	$$ \Zeta\left( s', \eis(s,e), \eta \right) := \int_{\F^{\times} \backslash \A^{\times}} \left( \eis - \eisCst \right)(s,e)(a(t)) \eta(t) \norm[t]_{\A}^{s'-\frac{1}{2}} d^{\times}t, \quad \forall e \in V_{\chi,\omega\chi^{-1}}^{\infty}, \eta \in \widehat{\R_+ \F^{\times} \backslash \A^{\times}}, $$
	which is convergent for $\Re s' \gg 1$ and admits meromorphic continuation to $s' \in \C$.
	
	For $V = V_{\pi}$ or $V_{\chi,\omega\chi^{-1}}$, the underlying inner product $\Pairing{\cdot}{\cdot}$ identifies $V$ with its dual space $V^{\vee}$. Define 
	$$ \varphi^{\vee} := \overline{\varphi} / \Norm[\varphi]^2 \in V^{\vee}, \quad \forall 0 \neq \varphi \in V. $$
	The \emph{matrix coefficient} $\beta(e_2,e_1^{\vee})$ or $\beta_s(e_2, e_1^{\vee})$ associated with a pair of nonzero vectors $e_1,e_2 \in V$ is defined to be
	$$ \beta(e_2,e_1^{\vee})(g) := \frac{\Pairing{\pi(g)e_2}{e_1}}{\Norm[e_1]^2} \quad \text{or} \quad \beta_s(e_2, e_1^{\vee})(g) := \frac{\Pairing{\pi_s(g).e_2}{e_1}}{\Norm[e_1]^2}, \quad g \in \GL_2(\A). $$
	We have the Godement-Jacquet zeta integral for $\beta \in \{ \beta(e_2,e_1^{\vee}), \beta_s(e_2, e_1^{\vee}) \}$
	$$ \Zeta(s', \Psi, \beta) := \int_{\GL_2(\A)} \Psi(g) \beta(g) \norm[\det g]_{\A}^{s'+\frac{1}{2}} dg, $$
which is convergent for $\Re s' \gg 1$ and admits a meromorphic continuation to $s' \in \C$.

Let $\Bas = \Bas(\pi)$ resp. $\Bas(\chi,\omega\chi^{-1})$ be an orthogonal basis of $V_{\pi}$ resp. $V_{\chi, \omega\chi^{-1}}$. Then $\Bas^{\vee} := \left\{ \varphi^{\vee} \ \middle| \ \varphi \in \Bas \right\}$ is the dual basis of $\Bas$ in $V^{\vee}$. Define the \emph{(cuspidal resp. continuous) spectral Motohashi distributions}
\begin{align} 
	M_3(\lambda, \Psi \mid \pi) &:= \sum_{\varphi_1, \varphi_2 \in \Bas(\pi)} \Zeta \left( \frac{s_0+1}{2}, \Psi, \beta(\varphi_2, \varphi_1^{\vee}) \right) \cdot \nonumber \\ 
	&\quad \Zeta\left( s_1 + \frac{s_0+1}{2}, \varphi_1, \chi_1 \right) \cdot \Zeta\left( s_2 - \frac{s_0-1}{2}, \varphi_2^{\vee}, \chi_2 \right), \label{SpecMDCusp}
\end{align}
\begin{align}
	M_3(\lambda, \Psi \mid \chi ,s) &:= \sum_{e_1,e_2 \in \Bas(\chi, \omega \chi^{-1})} \Zeta \left( \frac{s_0+1}{2}, \Psi, \beta_s(e_2, e_1^{\vee}) \right) \cdot \nonumber \\
	&\quad \Zeta\left( s_1 + \frac{s_0+1}{2}, \eis(s,e_1), \chi_1 \right) \cdot \Zeta\left( s_2 - \frac{s_0-1}{2}, \eis(-s, e_2^{\vee}), \chi_2 \right), \label{SpecMDCont}
\end{align}
	which are Motohashi distributions of parameter $\lambda = (s_1,s_2,s_0)$ and are meromorphic in $(\lambda,s) \in \C^4$. $M_3(\vec{0}, \Psi)$ is defined for $\lambda$ near $\vec{0}$ via
\begin{equation}
	M_3(\lambda, \Psi) := \sum_{\substack{\pi \text{ cuspidal} \\ \omega_{\pi} = \omega }} M_3(\lambda, \Psi \mid \pi) + \sum_{\chi \in \widehat{\R_+ \F^{\times} \backslash \A^{\times}}} \int_{-\infty}^{\infty} M_3(\lambda, \Psi \mid \chi,i\tau) \frac{d\tau}{4\pi}.
\label{3rdM}
\end{equation}

		\subsubsection{Geometric Degenerate Terms}
		
	Using Tate's integral, we define eight Motohashi distributions of parameter $\lambda = (s_1,s_2,s_0)$ near $\vec{0}$ as
\begin{align}
	&DG_1(\lambda, \Psi) := \frac{1}{\zeta_{\F}^*} \Res_{\substack{\chi = \eta_1^{-1} \\ s=1-s_1'}} M_4(\lambda, \Psi \mid \chi, s) = \int_{(\A^{\times})^3} \OFour_1 \OFour_2 \OFour_3 \Psi \begin{pmatrix} 0 & x_2 \\ x_3 & x_4 \end{pmatrix} \nonumber \\
	&\quad \omega\chi_1^2\chi_2^{-2}(x_2) \norm[x_2]_{\A}^{s_1+1} \chi_2^{-1}(x_3) \norm[x_3]_{\A}^{1-s_2} \omega\chi_1\chi_2^{-1}(x_4) \norm[x_4]_{\A}^{1+s_0+s_1-s_2} \sideset{}{_{i\neq 1}} \prod d^{\times}x_i,
\label{II_11AC}
\end{align}
\begin{align}
	&DG_2(\lambda, \Psi) = \frac{1}{\zeta_{\F}^*} \Res_{\substack{\chi = \eta_2^{-1} \\ s=1-s_2'}} M_4(\lambda, \Psi \mid \chi, s) = \int_{(\A^{\times})^3} \OFour_3 \Psi \begin{pmatrix} x_1 & 0 \\ x_3 & x_4 \end{pmatrix} \nonumber \\
	&\quad \omega^{-1}\chi_1^{-2}\chi_2^2(x_1) \norm[x_1]_{\A}^{1-s_1} \omega^{-1}\chi_1^{-2}\chi_2(x_3) \norm[x_3]_{\A}^{1-s_1-s_2} \chi_1^{-1}\chi_2(x_4) \norm[x_4]_{\A}^{s_0-s_2+1} \sideset{}{_{i\neq 2}} \prod d^{\times}x_i,
\label{II_12AC}
\end{align}
\begin{align}
	&DG_3(\lambda, \Psi) = \frac{1}{\zeta_{\F}^*} \Res_{\substack{\chi = \eta_3^{-1} \\ s=1-s_3'}} M_4(\lambda, \Psi \mid \chi, s) = \int_{(\A^{\times})^3} \OFour_2 \Psi \begin{pmatrix} x_1 & x_2 \\ 0 & x_4 \end{pmatrix} \nonumber \\
	&\quad \chi_2(x_1) \norm[x_1]_{\A}^{s_2+1} \omega\chi_1^2\chi_2^{-1}(x_2) \norm[x_2]_{\A}^{1+s_1+s_2} \omega\chi_1(x_4) \norm[x_4]_{\A}^{s_0+s_1+1} \sideset{}{_{i\neq 3}} \prod d^{\times}x_i,
\label{II_13AC}
\end{align}
\begin{align}
	&DG_4(\lambda, \Psi) = \frac{1}{\zeta_{\F}^*} \Res_{\substack{\chi = \eta_4^{-1} \\ s=1-s_4'}} M_4(\lambda, \Psi \mid \chi, s) = \int_{(\A^{\times})^3} \OFour_2 \OFour_3 \OFour_4 \Psi \begin{pmatrix} x_1 & x_2 \\ x_3 & 0 \end{pmatrix} \nonumber \\
	&\quad \omega^{-1}\chi_1^{-1}\chi_2(x_1) \norm[x_1]_{\A}^{1+s_2-s_0-s_1} \chi_1\chi_2^{-1}(x_2) \norm[x_2]_{\A}^{1+s_2-s_0} \omega^{-1}\chi_1^{-1}(x_3) \norm[x_3]_{\A}^{1-s_0-s_1} \sideset{}{_{i\neq 4}} \prod d^{\times}x_i,
\label{II_14AC}
\end{align}
\begin{align}
	&DG_5(\lambda, \Psi) = - \frac{1}{\zeta_{\F}^*} \Res_{\substack{\chi = \eta_1^{-1} \\ s=-s_1'}} M_4(\lambda, \Psi \mid \chi, s) = \int_{(\A^{\times})^3} \OFour_2 \OFour_4 \Psi \begin{pmatrix} 0 & x_2 \\ x_3 & x_4 \end{pmatrix} \nonumber \\
	&\quad \chi_1(x_2)\norm[x_2]_{\A}^{s_1} \chi_2(x_3)\norm[x_3]_{\A}^{s_2+1} \omega^{-1}\chi_1^{-1}\chi_2(x_4)\norm[x_4]_{\A}^{s_2-s_0-s_1+1} \sideset{}{_{i \neq 1}} \prod d^{\times}x_i , \label{II_21CE}
\end{align}
\begin{align}
	&DG_6(\lambda, \Psi) = -\frac{1}{\zeta_{\F}^*} \Res_{\substack{\chi = \eta_2^{-1} \\ s=-s_2'}} M_4(\lambda, \Psi \mid \chi, s) = \int_{(\A^{\times})^3} \OFour_1 \OFour_2 \OFour_4 \Psi \begin{pmatrix} x_1 & 0 \\ x_3 & x_4 \end{pmatrix} \nonumber \\
	&\quad \chi_1(x_1)\norm[x_1]_{\A}^{s_1+1} \chi_1\chi_2(x_3)\norm[x_3]_{\A}^{s_2+s_1+1} \omega^{-1}\chi_2(x_4)\norm[x_4]_{\A}^{s_2-s_0+1} \sideset{}{_{i \neq 2}} \prod d^{\times}x_i , \label{II_22CE}
\end{align}
\begin{align}
	&DG_7(\lambda, \Psi) = -\frac{1}{\zeta_{\F}^*} \Res_{\substack{\chi = \eta_3^{-1} \\ s=-s_3'}} M_4(\lambda, \Psi \mid \chi, s) = \int_{(\A^{\times})^3} \OFour_2 \OFour_3 \Psi \begin{pmatrix} x_1 & x_2 \\ 0 & x_4 \end{pmatrix} \nonumber \\
	&\quad \chi_2(x_1)\norm[x_1]_{\A}^{s_2} \chi_1\chi_2(x_2)\norm[x_2]_{\A}^{s_2+s_1} \omega\chi_1(x_4)\norm[x_4]_{\A}^{s_0+s_1} \sideset{}{_{i \neq 3}} \prod d^{\times}x_i , \label{II_23CE}
\end{align}
\begin{align}
	&DG_8(\lambda, \Psi) = -\frac{1}{\zeta_{\F}^*} \Res_{\substack{\chi = \eta_4^{-1} \\ s=-s_4'}} M_4(\lambda, \Psi \mid \chi, s) = \int_{(\A^{\times})^3} \OFour_2 \Psi \begin{pmatrix} x_1 & x_2 \\ x_3 & 0 \end{pmatrix} \nonumber \\
	&\quad \omega^{-1}\chi_1^{-1}\chi_2(x_1)\norm[x_1]_{\A}^{s_2-s_0-s_1} \omega^{-1}\chi_2(x_2)\norm[x_2]_{\A}^{s_2-s_0} \omega\chi_1(x_3)\norm[x_3]_{\A}^{s_0+s_1+1} \sideset{}{_{i \neq 4}} \prod d^{\times}x_i . \label{II_24CE}
\end{align}
	They are possibly singular at $\lambda = \vec{0}$.

		\subsubsection{Spectral Degenerate Terms}
		
	For $\chi_1,\chi_2,\omega$ at generic position, the continuous spectral Motohashi distribution $M_3(\lambda, \Psi \mid \chi, s)$ (see (\ref{SpecMDCont})) has a simple pole for some special $s$, the residue at which is a Motohashi distribution of parameter $\lambda=(s_1,s_2,s_0)$, meromorphic in $\lambda$ and singular at $\lambda = \vec{0}$. The $DS_j(\lambda, \Psi)$ are given by
	$$ DS_0(\lambda,\Psi) := \frac{1}{\zeta_{\F}^*} \Res_{s=\frac{1+s_0}{2}} M_3(\lambda, \Psi \mid \mathbbm{1},s), \quad DS_1(\lambda, \Psi) := \frac{1}{\zeta_{\F}^*} \Res_{s=\frac{1-s_0}{2}} M_3(\lambda, \Psi \mid \mathbbm{1},s), $$
	$$ DS_2(\lambda, \Psi) := \frac{1}{\zeta_{\F}^*} \Res_{s=-\left( s_1 + \frac{s_0-1}{2} \right)} M_3(\lambda, \Psi \mid \chi_1^{-1},s), \quad DS_3(\lambda, \Psi) := \frac{1}{\zeta_{\F}^*} \Res_{s=-\left( s_2 - \frac{s_0+1}{2} \right)} M_3(\lambda, \Psi \mid \omega\chi_2^{-1},s), $$
	$$ DS_4(\lambda, \Psi) := \frac{1}{\zeta_{\F}^*} \Res_{s=s_2-\frac{s_0-1}{2}} M_3(\lambda, \Psi \mid \chi_2,s), \quad DS_5(\lambda, \Psi) := \frac{1}{\zeta_{\F}^*} \Res_{s=s_1+\frac{s_0+1}{2}} M_3(\lambda, \Psi \mid \omega\chi_1,s). $$

	\subsection{List of Tempered Distributions: Compact Variation}
	
		\subsubsection{Spectral Terms}
		\label{SpTVar}
	
	Let $\nu_{\D}$ be the reduced norm map of $\D$, whose restriction to $\E$ is the usual norm map $\nu_{\E}$ from $\E$ to $\F$. Write the (compact) automorphic quotient spaces as
	$$ [\D^{\times}] := \D^{\times}(\F) \A^{\times} \backslash \D^{\times}(\A), \quad [\E^{\times}] := \E^{\times}(\F) \A^{\times} \backslash \E^{\times}(\A). $$
	Let $\pi$ be any cuspidal representation of $\D^{\times}$ with central character $\omega$. Let $\eta$ be any Hecke character of $\F^{\times} \R_+ \backslash \A^{\times}$. We introduce
	$$ M(s_0,\Psi \mid \pi) := \sideset{}{_{\varphi_1, \varphi_2 \in \Bas(\pi)}} \sum \Zeta \left( s_0+\frac{1}{2}, \Psi, \beta(\varphi_2, \varphi_1^{\vee}) \right) \int_{[\E^{\times}]}\varphi_1(t_1) \Omega(t_1)^{-1} d^{\times}t_1 \int_{[\E^{\times}]} \varphi_2^{\vee}(t_2) \Omega(t_2) d^{\times}t_2, $$
	$$ DS(s_0, \Psi \mid \eta) := \left( \int_{\D^{\times}(\A)} \Psi(g) \eta(\nu_{\D}(g)) \norm[\nu_{\D}(g)]_{\A}^{s_0+1} dg \right) \cdot \mathbbm{1}_{\Omega = \eta \circ \nu_{\E}}. $$
	The integrals in $DS(s_0,\Psi \mid \eta)$ resp. $M(s_0,\Psi \mid \pi)$ with parameter $s_0$ have meromorphic continuation to $s_0 \in \C$, and are integral representations for $L(s_0+1/2,\eta)$ resp. $L(s_0+1/2,\pi)$ (see \cite[Proposition 6.9 \& Theorem 13.8]{GoJa72}). By Waldspurger's famous formula \cite{Wa85}, the product of integrals over $[\E^{\times}]$ is an integral representation for $L(1/2, \pi_{\E} \otimes \Omega^{-1})$, where $\pi_{\E}$ is the base change lift of $\pi$ to $(\D \otimes_{\F} \E)^{\times}(\A) \simeq \GL_2(\A_{\E})$. We define the tempered distributions
	$$ M(s_0, \Psi) := \sideset{}{_{\substack{\pi \text{\ cuspidal} \\ \omega_{\pi} = \omega}}} \sum M(s_0, \Psi \mid \pi), \quad DS(s_0, \Psi) := \sum_{\eta} DS(s_0,\Psi \mid \eta). $$
	Their value at $s_0=0$ represent some mixed moment of $L(1/2,\pi) L(1/2, \pi_{\E} \otimes \Omega^{-1})$ resp. first moment of $L(1/2,\eta)$ with $\eta \circ \nu_{\D}$ distinguished by $(\E^{\times}, \Omega^{-1})$.

		\subsubsection{Geometric Terms}
		
	By the structure theory (see \cite[Definition \Rmnum{1}.1]{Vi80}) there exists $j \in \D(\F)$ so that $\D = \E \oplus \E j$, $j^2 \in \F^{\times}$ and $j x j^{-1} = \bar{x}$ for any $x \in \E$, where $x \mapsto \bar{x}$ is the action of the non trivial element in the Galois group of $\E/\F$. Consequently, we have $\D(\A) \simeq \E(\A)^2$, hence $\Sch(\D(\A)) \simeq \Sch(\A_{\E}^2)$. Under this identification, we can view $\Psi \in \Sch(\A_{\E}^2)$ by writing $\Psi(x_1+x_2j)$ for $x_j \in \A_{\E}$. We denote by $\OFour_j$ the partial Fourier transform with respect to the variable $x_j$ for the additive character $\psi \circ \Tr_{\E/\F}$. Let the $\F$-group $\E^1$ be the subgroup of elements $t$ in $\E^{\times}$ such that $t \bar{t} = 1$. Define a compact group
	$$ [\E^1 \backslash \E^{\times}] := \E^{\times} \R_+ \E^1(\A) \backslash \A_{\E}^{\times}, $$
	where $\R_+$ is viewed as the image of a fixed section map $s_{\E}$ of the adelic norm map $\E^{\times} \backslash \A_{\E}^{\times} \to \R_+$.
	
	For $\norm[s_0] < 1/4$, we define for any $\Xi \in \widehat{[\E^1 \backslash \E^{\times}]}$
	$$ M_2(s_0,\Psi \mid \Xi) := \frac{1}{\Vol([\E^1 \backslash \E^{\times}])} \int_{(\frac{1}{2})} \int_{\A_{\E}^{\times} \times \A_{\E}^{\times}} \Psi(t_1+t_2j) \Omega \Xi^{-1}(t_1)\norm[t_1]_{\A_{\E}}^{s_0+1-s} \Xi(t_2) \norm[t_2]_{\A}^s d^{\times}t_1 d^{\times}t_2 \frac{ds}{2\pi i}, $$
	$$ DG_1(s_0,\Psi) = \Vol([\E^{\times}]) \cdot \int_{\A_{\E}^{\times}} \Psi(t) \Omega(t) \norm[t]_{\A_{\E}}^{s_0+1} d^{\times}t_1, $$
	$$ DG_2(s_0,\Psi) = \Vol([\E^{\times}]) \id_{\Omega \mid_{\E^1(\A)}=\id} \cdot \int_{\A_{\E}^{\times}} \Psi(t) \Omega(t) \norm[t]_{\A_{\E}}^{s_0+1} d^{\times}t_1, $$
	$$ DG_3(s_0,\Psi) = \frac{\zeta_{\E}^*}{\Vol([\E^1 \backslash \E^{\times}])} \int_{\A_{\E}^{\times}} \left( \int_{\A_{\E}} \Psi(t_1+x_2j) dx_2 \right) \Omega(t_1) \norm[t_1]_{\A_{\E}}^{s_0} d^{\times}t_1, $$
	$$ DG_4(s_0,\Psi) = \frac{\zeta_{\E}^*}{\Vol([\E^1 \backslash \E^{\times}])} \id_{\Omega \mid_{\E^1(\A)} = \id} \cdot \int_{\A_{\E}^{\times}} \left( \int_{\A_{\E}} \Psi(x_1+t_2j) dx_1 \right) \Omega(t_2) \norm[t_2]_{\A_{\E}}^{s_0} d^{\times}t_2. $$
	The geometric main term is defined to be
	$$ M_2(s_0,\Psi) := \sum_{\Xi \in \widehat{[\E^1 \backslash \E^{\times}]}} M_2(s_0,\Psi \mid \Xi), $$
	whose value at $s_0=0$ represents the mean value of $L(1/2-i\tau,\Omega\Xi^{-1})L(1/2+i\tau,\Xi)$. The degenerate terms $DG_j(s_0,\Psi)$ have obvious meromorphic continuation to $s_0 \in \C$ and are regular at $s_0$ under the assumption $\Omega \neq \id$.

\section{Third Moment (Spectral) Side}
\label{3rdMSec}

	\subsection{Some Spectral Theory}
	
	Recall the representation $(\rpR_{\omega}, V_{\omega})$ of $\GL_2(\A)$ given by (\ref{AutoSwC}).
\begin{definition}
	Let $\varphi \in V_{\omega}^{\infty}$ be a smooth vector represented by a smooth function on $\GL_2(\A)$, so that for any $X$ in the universal enveloping algebra of the Lie algebra of $\GL_2(\A_{\infty})$, we have
	$$ \extnorm{ \rpR_{\omega}(X).\varphi(g) } \ll_{\epsilon} \Ht(g)^{1/2-\epsilon} $$
	for any $\epsilon > 0$ sufficiently small, uniformly in $g$ lying in a/any \emph{Siegel domain}, then we call $\varphi$ ``nice''.
\label{NiceFDef}
\end{definition}
\begin{remark}
	$V_{\omega}^{\infty}$ is the Sobolev space of order $\infty$ for the $\intL^2$-norm.
\end{remark}
\begin{theorem}{(\emph{automorphic Fourier inversion formula})}
	For ``nice'' $\varphi \in V_{\omega}^{\infty}$, we have a decomposition 
\begin{align*}
	\varphi(g) &= \sideset{}{_{\substack{ \pi \text{ cuspidal} \\ \omega_{\pi} = \omega }}} \sum \sideset{}{_{e \in \Bas(\pi)}} \sum \Pairing{\varphi}{e} e(g) \\
	&+ \sum_{\chi \in \widehat{\R_+ \F^{\times} \backslash \A^{\times}}} \sum_{f \in \Bas(\chi,\omega\chi^{-1})} \int_{-\infty}^{\infty} \Pairing{\varphi}{\eis(i\tau,f)} \eis(i\tau,f)(g) \frac{d\tau}{4\pi} \\
	&+ \frac{1}{\Vol([\PGL_2])} \sideset{}{_{\substack{ \chi \in \widehat{\F^{\times} \backslash \A^{\times}} \\ \chi^2 = \omega }}} \sum \int_{[\PGL_2]} \varphi(x) \overline{\chi(\det x)} dx \cdot \chi(\det g)
\end{align*}
	with normal convergence in $ [\PGL_2] $. Here, $\Bas(\pi)$ resp. $\Bas(\chi,\omega\chi^{-1})$ is an orthonormal basis of $\gp{K}_{\infty}$-isotypic and $\gp{K}_{\fin}$-finite vectors of $V_{\pi}$ resp. $V_{\chi,\omega\chi^{-1}}$ and the \emph{automorphic Fourier coefficients}, namely the above terms written via inner product such as $\Pairing{\varphi}{\eis(i\tau,f)}$, are given by the usual convergent integrals.
\label{AutoFourInv}
\end{theorem}
\begin{proof}
	For general $\varphi \in V_{\omega}^{\infty}$, the above formula is established in \cite[Theorem 2.16]{Wu14} with $\Pairing{\varphi}{\eis(i\tau,f)}$ replaced by the existence of some $a(i\tau,f) \in \C$ (see the clarification in \cite{Wu5}). The justification of 
	$$ a(i\tau,f) = \Pairing{\varphi}{\eis(i\tau,e)} \quad \text{a.e. } \tau \in \R $$
literally follows the proof of \cite[Lemma 3.24]{Wu5}, replacing $h \in \Cont_c^{\infty}(\GL_2, \omega)$ there by ``nice'' $\varphi$ here. 
\end{proof}
\begin{remark}
	For a clarification of the terminology ``automorphic Fourier inversion'' as well as its relation with ``spectral decomposition'', please see \cite[\S 1.2 \& 1.3]{Wu5} and \cite[\S 2.1]{LPW21}.
\end{remark}
\begin{theorem}
	Let notations be as in the previous theorem. Write
	$$ \varphi_{\gp{N}}(g) := \int_{\F \backslash \A} \varphi(n(u)g) du, $$
	$$ W_e(g) := \int_{\F \backslash \A} e(n(u)g) \psi(-u) du \quad \text{resp.} \quad W(i\tau,f)(g) := \int_{\F \backslash \A} \eis(i\tau,f)(n(u)g) \psi(-u) du $$
	for the Whittaker function of $e$ resp. $\eis(i\tau,f)$. Then we have an equality as functions on $\gp{B}(\F) \backslash \GL_2(\A)$
\begin{align*}
	\varphi(g) - \varphi_{\gp{N}}(g) &= \sideset{}{_{\substack{ \pi \text{ cuspidal} \\ \omega_{\pi} = \omega }}} \sum \sideset{}{_{e \in \Bas(\pi)}} \sum \Pairing{\varphi}{e} \sideset{}{_{\alpha \in \F^{\times}}} \sum W_e(a(\alpha)g) \\
	&+ \sum_{\chi \in \widehat{\R_+ \F^{\times} \backslash \A^{\times}}} \sum_{f \in \Bas(\chi,\omega\chi^{-1})} \int_{-\infty}^{\infty} \Pairing{\varphi}{\eis(i\tau,f)} \sideset{}{_{\alpha \in \F^{\times}}} \sum W(i\tau,f)(a(\alpha)g) \frac{d\tau}{4\pi}
\end{align*}
	with absolute and uniform convergence in any Siegel domain. Moreover, the dominating sum by replacing each coefficient \& summand \& integrand with its absolute value is of rapid decay with respect to $\Ht(g)$.
\label{AutoFourInvV}
\end{theorem}
\begin{proof}
	This is \cite[Theorem 2.18]{Wu14}. Note that the ``moreover'' part is implicit in the proof.
\end{proof}

\begin{definition}
	Let $\omega$ be a unitary character of $\F^{\times} \backslash \ag{A}^{\times}$. Let $\varphi$ be a continuous function on $\GL_2(\F) \backslash \GL_2(\ag{A})$ with central character $\omega$. We call $\varphi$ \emph{finitely regularizable} if there exist characters $\chi_i: \F^{\times} \R_+ \backslash \ag{A}^{\times} \to \ag{C}^{(1)}$, $\alpha_i \in \ag{C}, n_i \in \ag{N}$ and continuous functions $f_i \in \Ind_{\gp{B}(\ag{A}) \cap \gp{K}}^{\gp{K}} (\chi_i, \omega \chi_i^{-1})$, such that for any $M \gg 1$
\begin{equation} 
	\varphi(n(x)a(y)k) = \varphi_{\gp{N}}^*(n(x)a(y)k) + O(\norm[y]_{\ag{A}}^{-M}), \text{ as } \norm[y]_{\ag{A}} \to \infty, 
\label{DiffRD}
\end{equation}
	where we have written the \emph{essential constant term}
	$$ \varphi_{\gp{N}}^*(n(x)a(y)k)=\varphi_{\gp{N}}^*(a(y)k)=\sideset{}{_{i=1}^l} \sum \chi_i(y) \norm[y]_{\ag{A}}^{\frac{1}{2}+\alpha_i} \log^{n_i} \norm[y]_{\ag{A}} f_i(k). $$
	We call $\Ex(\varphi)=\{ \chi_i \norm^{\frac{1}{2}+\alpha_i}: 1 \leq i \leq l \}$ the \emph{exponent set} of $\varphi$, and define
	$$ \Ex^+(\varphi) = \{ \chi_i \norm^{\frac{1}{2}+\alpha_i} \in \Ex(\varphi): \Re \alpha_i \geq 0 \}; \quad \Ex^-(\varphi) = \{ \chi_i \norm^{\frac{1}{2}+\alpha_i} \in \Ex(\varphi): \Re \alpha_i < 0 \}. $$
	The space of finitely regularizable functions with central character $\omega$ is denoted by $\Aut^{\freg}(\GL_2,\omega)$.
\label{FinRegFuncDef}
\end{definition}
\begin{proposition}
	Let $\xi_1,\xi_2, \omega$ be Hecke characters with $\xi_1 \xi_2 \omega =1$. Let $f \in \pi(\xi_1,\xi_2)$. Suppose $\varphi \in \Aut^{\freg}(\GL_2, \omega)$ as in Definition \ref{FinRegFuncDef}. For $\Re s \gg 1$ sufficiently large,
	$$ R(s,\varphi; f) := \int_{\F^{\times} \backslash \ag{A}^{\times}} \int_{\gp{K}} (\varphi_{\gp{N}} - \varphi_{\gp{N}}^*)(a(y)\kappa) f(\kappa) \xi_1(y) \norm[y]_{\ag{A}}^{s-\frac{1}{2}} d\kappa d^{\times}y $$
is absolutely convergent. It has a meromorphic continuation to $s \in \ag{C}$. If in addition
	$$ \Theta := \max_j \{ \Re \alpha_j \} < 0, $$
then we have, with the right hand side absolutely converging
	$$ R(s,\varphi; f) = \int_{[\PGL_2]} \varphi \cdot \eis(s,f), \quad \Theta < \Re s < -\Theta.  $$
\label{VRI}
\end{proposition}
\begin{proof}
	This is (part of) \cite[Proposition 2.5]{Wu2}.
\end{proof}

	\subsection{Godement-Jacquet Pre-trace Formula}
	
	Above all, we mention the following elementary estimation without proof (see \cite[Lemma 5.37]{Wu9} or \cite[Lemma 11.4]{GoJa72}).
\begin{proposition}
	Let $\Phi \in \Sch(\A)$. Then we have for any $N > 1$
	$$ \sideset{}{_{\alpha \in \F^{\times}}} \sum \norm[\Phi(\alpha t)] \ll_N \min(\norm[t]_{\A}^{-1}, \norm[t]_{\A}^{-N}), $$
	where the implied constant depends only on $\F$ and some Schwartz norm of $\Phi$ with order depending on $N$.
\label{PoissonSEst}
\end{proposition}
	For fixed $x$, we propose to study the function in $y$
	$$ \Tree{KK}{\omega, s_0}{x}{y} := \int_{\F^{\times} \backslash \A^{\times}} \Shoulder{KK}{x}{y z(u)} \omega(u) \norm[u]_{\A}^{s_0+2} d^{\times}u, $$
	$$ \text{or equivalently} \quad \Tree{\widetilde{KK}}{\omega, s_0}{x}{y} := \Tree{KK}{\omega, s_0}{x}{y} \norm[\det x^{-1}y]_{\A}^{\frac{s_0}{2}+1}. $$
\begin{remark}
	The partition of $\Mat_2(\F) - \{ 0 \}$ into $\GL_2(\F) \times \GL_2(\F)$-orbits
	$$ \Mat_2(\F) - \{ 0 \} = \sideset{}{_{i=1}^2} \bigsqcup \Mat_2^{(i)}(\F) \quad \text{with} \quad \Mat_2^{(i)}(\F) := \left\{ \xi \in \Mat_2(\F) \ \middle| \ \mathrm{rank}(\xi) = i \right\}. $$
	implies a decomposition
\begin{equation} 
	\Shoulder{KK}{x}{y} = \sideset{}{_{i=1}^2} \sum \Shoulder{KK^{(i)}}{x}{y}, \quad \Shoulder{KK^{(i)}}{x}{y} := \sideset{}{_{\xi \in \Mat_2^{(i)}(\F)}} \sum \Psi(x^{-1} \xi y). 
\label{KKDecomp}
\end{equation}
	They give the corresponding functions $\Tree{KK^{(i)}}{\omega, s_0}{x}{y}$, $\Tree{\widetilde{KK}^{(i)}}{\omega, s_0}{x}{y}$ and $\Tree{\widetilde{\Delta \Delta}^{(i)}}{\omega, s_0}{x}{y}$, etc.
\end{remark}
	
	We would like to show that $y \mapsto \Tree{KK^{(i)}}{\omega, s_0}{x}{y}$ lie in $\Aut^{\freg}(\GL_2,\omega^{-1})$. This is obvious for $i=1$. In fact, if we define
	$$ \Shoulder{RR^{(1)}}{x}{y} := \sideset{}{_{\alpha \in \F^{\times}}} \sum \rpL_x \rpR_y \Psi \begin{pmatrix} 0 & \alpha \\ 0 & 0 \end{pmatrix}, $$
\begin{align} 
	\Tree{\widetilde{RR}^{(1)}}{\omega,s_0}{x}{y} &:= \int_{\F^{\times} \backslash \A^{\times}} \Shoulder{RR^{(1)}}{x}{y d(z,z)} \omega(z) \norm[z]_{\A}^{s_0+2} d^{\times}z \cdot \norm[\det x^{-1}y]_{\A}^{\frac{s_0}{2}+1} \nonumber \\
	&= \int_{\A^{\times}} \Psi \left( x^{-1} \begin{pmatrix} 0 & z \\ 0 & 0 \end{pmatrix} y \right) \omega(z) \norm[z]_{\A}^{s_0+2} d^{\times}z \cdot \norm[\det x^{-1}y]_{\A}^{\frac{s_0}{2}+1}, \label{R1BiSec}
\end{align}
	then it is easy to verify that
	$$ x \mapsto \Tree{\widetilde{RR}^{(1)}}{\omega,s_0}{x}{y} \quad \text{resp.} \quad y \mapsto \Tree{\widetilde{RR}^{(1)}}{\omega,s_0}{x}{y} $$
	lies in the induced model of the principal series representation $\pi(\omega \norm_{\A}^{(s_0+1)/2}, \norm_{\A}^{-(1+s_0)/2})$ resp. \\ $\pi(\norm_{\A}^{(s_0+1)/2}, \omega^{-1}\norm_{\A}^{-(1+s_0)/2})$. Note that $\Mat_2^{(1)}(\F)$ is a single orbit under the action of $\GL_2(\F) \times \GL_2(\F)$, and the stabilizer group of the line consisting of the elements
	$$ \begin{pmatrix} 0 & \alpha \\ 0 & 0 \end{pmatrix}, \quad \alpha \in \F^{\times} $$
	is $\gp{B}(\F) \times \gp{B}(\F)$. Therefore we get (with absolute convergence for $\Re s_0 > 1$)
\begin{equation} 
	\Tree{\widetilde{KK}^{(1)}}{\omega,s_0}{x}{y} = \sum_{[\gamma_1], [\gamma_2] \in \gp{B}(\F) \backslash \GL_2(\F)} \Tree{\widetilde{RR}^{(1)}}{\omega,s_0}{\gamma_1 x}{\gamma_2 y} 
\label{R1Eis}
\end{equation}
	is an Eisenstein series in either variable, hence lies in $\Aut^{\freg}(\GL_2,\omega^{-1})$. 
	
	It remains to show that $y \mapsto \Tree{KK}{\omega, s_0}{x}{y}$ lie in $\Aut^{\freg}(\GL_2,\omega^{-1})$. We assume
	$$ y = n(u) a(t) \kappa, \qquad u \in \F \backslash \A, t \in \F^{\times} \backslash \A^{\times}, \kappa \in \gp{K} $$
with $\norm[t]_{\A} \geq 1$. Fourier inversion for the unipotent group gives
	$$ \Shoulder{KK}{x}{y} = \Shoulder{KN}{x}{y} + \sideset{}{_{\delta \in \F^{\times}}} \sum \Shoulder{KW}{x}{a(\delta)y}, $$
	where $\Shoulder{KN}{x}{y}$ is given in (\ref{PartialCstT}) and
\begin{equation}
	\Shoulder{KW}{x}{y} := \int_{\F \backslash \A} \Shoulder{KK}{x}{n(u)y} \psi(-u) du.
\label{PartialWhiR}
\end{equation}

	We first deal with the sum over $\delta \in \F^{\times}$. The orbital decomposition of $\Mat_2^*(\F) := \Mat_2^{(1)}(\F) \cup \Mat_2^{(2)}(\F)$ by right multiplication by $\gp{N}(\F)$ (or $\gp{B}(\F)$)
\begin{align*}
	\Mat_2^*(\F) &= \sideset{}{_{\substack{ \alpha, \beta \in \F^{\times} \\ \gamma \in \F }}} \bigsqcup \begin{pmatrix} \alpha & \gamma \\ \beta & 0 \end{pmatrix} \gp{N}(\F) \sqcup \sideset{}{_{\substack{ \beta \in \F^{\times} \\ \gamma \in \F }}} \bigsqcup \begin{pmatrix} 0 & \gamma \\ \beta & 0 \end{pmatrix} \gp{N}(\F) \sqcup \sideset{}{_{\substack{ \alpha \in \F^{\times} \\ \gamma \in \F }}} \bigsqcup \begin{pmatrix} \alpha & 0 \\ 0 & \gamma \end{pmatrix} \gp{N}(\F) \\
	&\quad \sqcup \sideset{}{_{(\alpha,\beta) \in \F^2- \{ \vec{0} \}}} \bigsqcup \begin{pmatrix} 0 & \alpha \\ 0 & \beta \end{pmatrix} \\
	&=: \vO_1 \sqcup \vO_2 \sqcup \vO_3 \sqcup \vO_4
\end{align*}
	implies the following decomposition
	$$ \Shoulder{KW}{x}{y} = \sideset{}{_{i=1}^5} \sum \Shoulder{KW_i}{x}{y}, \qquad \Shoulder{KW_i}{x}{y} := \int_{\F \backslash \A} \sum_{\xi \in \vO_i} \Psi(x^{-1} \xi n(u) y) \psi(-u) du. $$
	For the first term, we have
	$$ \Shoulder{KW_1}{x}{n(u)a(t)\kappa} = \psi(u) \int_{\A} \sideset{}{_{\substack{ \alpha, \beta \in \F^{\times} \\ \gamma \in \F }}} \sum \rpL_{d(\alpha,\beta)^{-1}} \rpL_x \rpR_{\kappa} \Psi \begin{pmatrix} t & \gamma + v \\ t & v \end{pmatrix} \psi(-v) dv. $$
	For any $\Psi \in \Sch(\Mat_2(\A))$, Poisson summation formula implies
	$$ \int_{\A} \sideset{}{_{\gamma \in \F}} \sum \Psi \begin{pmatrix} t & \gamma + v \\ t & v \end{pmatrix} \psi(-v) dv = \sideset{}{_{\gamma \in \F}} \sum \OFour_4 \OFour_2 \Psi\begin{pmatrix} t & \gamma \\ t & 1-\gamma \end{pmatrix}. $$
	Hence we can rewrite
	$$ \Shoulder{KW_1}{x}{n(u)a(t)\kappa d(z,z)} = \frac{\psi(u)}{\norm[z]_{\A}^2} \sideset{}{_{\substack{ \alpha, \beta \in \F^{\times} \\ \gamma \in \F }}} \sum \OFour_4 \OFour_2 \rpL_x \rpR_{\kappa} \Psi\begin{pmatrix} \alpha t z & \gamma \alpha^{-1} z^{-1} \\ \beta t z & (1-\gamma) \beta^{-1} z^{-1} \end{pmatrix}, $$
	from which we deduce
\begin{align*}
	\sideset{}{_{\delta \in \F^{\times}}} \sum \extnorm{ \Shoulder{KW_1}{x}{a(\delta)n(u)a(t)\kappa d(z,z)} } &\leq \norm[z]_{\A}^{-2} \sideset{}{_{\substack{ \alpha, \beta, \delta \in \F^{\times} \\ \gamma \in \F }}} \sum \extnorm{ \OFour_4 \OFour_2 \rpL_x \rpR_{\kappa} \Psi\begin{pmatrix} \alpha t z & \gamma \delta^{-1} \alpha^{-1} z^{-1} \\ \beta t z & (1-\gamma) \delta^{-1} \beta^{-1} z^{-1} \end{pmatrix} } \\
	&= \norm[z]_{\A}^{-2} \sideset{}{_{\alpha, \beta \in \F^{\times}}} \sum \sideset{}{_{\substack{ \gamma, \delta \in \F \\ \gamma + \delta \neq 0 }}} \sum \extnorm{ \OFour_4 \OFour_2 \rpL_x \rpR_{\kappa} \Psi\begin{pmatrix} \alpha t z & \gamma \alpha^{-1} z^{-1} \\ \beta t z & \delta \beta^{-1} z^{-1} \end{pmatrix} }.
\end{align*}
	We split the sum over $\gamma, \delta$ as
	$$ \sideset{}{_{\substack{ \gamma, \delta \in \F \\ \gamma + \delta \neq 0 }}} \sum = \sideset{}{_{\substack{ \gamma + \delta \neq 0 \\ \gamma \delta = 0 }}} \sum + \sideset{}{_{\substack{ \gamma, \delta \in \F^{\times} \\ \gamma + \delta \neq 0 }}} \sum. $$
	By Proposition \ref{PoissonSEst}, the first part is bounded as
\begin{align*}
	&\quad \norm[z]_{\A}^{-2} \sideset{}{_{\alpha, \beta, \gamma \in \F^{\times}}} \sum \extnorm{ \OFour_4 \OFour_2 \rpL_x \rpR_{\kappa} \Psi\begin{pmatrix} \alpha t z & \gamma \alpha^{-1} z^{-1} \\ \beta t z & 0 \end{pmatrix} } + \norm[z]_{\A}^{-2} \sideset{}{_{\alpha, \beta, \delta \in \F^{\times}}} \sum \extnorm{ \OFour_4 \OFour_2 \rpL_x \rpR_{\kappa} \Psi\begin{pmatrix} \alpha t z & 0 \\ \beta t z & \delta \beta^{-1} z^{-1} \end{pmatrix} } \\
	&\ll_{N_1,N_2} \norm[z]_{\A}^{-2} \min(\norm[tz]_{\A}^{-1}, \norm[tz]_{\A}^{-N_1}) \min(\norm[z]_{\A}, \norm[z]_{\A}^{N_2});
\end{align*}
	similarly the second part is dominated and bounded as
\begin{align*}
	&\quad \norm[z]_{\A}^{-2} \sideset{}{_{\alpha, \beta \in \F^{\times}}} \sum \sideset{}{_{\substack{ \gamma, \delta \in \F^{\times} \\ \gamma + \delta \neq 0 }}} \sum \extnorm{ \OFour_4 \OFour_2 \rpL_x \rpR_{\kappa} \Psi\begin{pmatrix} \alpha t z & \gamma \alpha^{-1} z^{-1} \\ \beta t z & \delta \beta^{-1} z^{-1} \end{pmatrix} } \\
	&\leq \norm[z]_{\A}^{-2} \sideset{}{_{\alpha, \beta, \gamma, \delta \in \F^{\times}}} \sum \extnorm{ \OFour_4 \OFour_2 \rpL_x \rpR_{\kappa} \Psi\begin{pmatrix} \alpha t z & \gamma z^{-1} \\ \beta t z & \delta z^{-1} \end{pmatrix} } \ll_{N_1,N_2} \norm[z]_{\A}^{-2} \min(\norm[tz]_{\A}^{-1}, \norm[tz]_{\A}^{-N_1}) \min(\norm[z]_{\A}, \norm[z]_{\A}^{N_2}).
\end{align*}
	We deduce that for any $N_1, N_2 > 1$
	$$ \sideset{}{_{\delta \in \F^{\times}}} \sum \extnorm{ \Shoulder{KW_1}{x}{a(\delta)n(u)a(t)\kappa d(z,z)} } \ll_{N_1,N_2} \norm[z]_{\A}^{-2} \min(\norm[tz]_{\A}^{-1}, \norm[tz]_{\A}^{-N_1}) \min(\norm[z]_{\A}, \norm[z]_{\A}^{N_2}), $$
	where the implied constant depends only on $\F$ and some Schwartz norm of $\Psi$ with order depending only on $x$, $N_1$ and $N_2$ (not on $u$ nor $\kappa$). 
	
\noindent The treatment of the second and third terms being similar, we only deal with the second one. We have
\begin{align*}
	\Shoulder{KW_2}{x}{n(u)a(t)\kappa} &= \psi(u) \int_{\A} \sideset{}{_{\substack{\beta \in \F^{\times} \\ \gamma \in \F}}} \sum \Psi \left( x^{-1} \begin{pmatrix} 0 & \gamma \\ \beta & 0 \end{pmatrix} \begin{pmatrix} 1 & v \\ 0 & 1 \end{pmatrix} \begin{pmatrix} t & 0 \\ 0 & 1 \end{pmatrix} \kappa \right) \psi(-v) dv \\
	&= \psi(u) \sideset{}{_{\substack{\beta \in \F^{\times} \\ \gamma \in \F}}} \sum \OFour_4 \rpL_x \rpR_{\kappa} \Psi \begin{pmatrix} 0 & \gamma \\ \beta t & \beta^{-1} \end{pmatrix}.
\end{align*}
	It follows that for any $N_1, N_2 > 1$
\begin{align*}
	\sideset{}{_{\delta \in \F^{\times}}} \sum & \extnorm{ \Shoulder{KW_2}{x}{a(\delta)n(u)a(t)\kappa d(z,z)} } \leq \norm[z]_{\A}^{-1} \sideset{}{_{\substack{\beta, \delta \in \F^{\times} \\ \gamma \in \F}}} \sum \extnorm{ \OFour_4 \rpL_x \rpR_{\kappa} \Psi \begin{pmatrix} 0 & \gamma z \\ \beta \delta t z & \beta^{-1} z^{-1} \end{pmatrix} } \\
	&\ll_{N_1,N_2} \norm[z]_{\A}^{-1} \min(\norm[tz]_{\A}^{-1}, \norm[tz]_{\A}^{-N_1}) \min(\norm[z]_{\A}, \norm[z]_{\A}^{N_2}) \left\{ 1 + \norm[z]_{\A}^{-1} \right\}.
\end{align*}

\noindent We leave it as a simple exercise to check $\Shoulder{KW_4}{x}{y} = 0$.

\noindent We summarize the above estimation as: for any $N_1, N_2 > 1$
\begin{equation}
	\sideset{}{_{\delta \in \F^{\times}}} \sum \extnorm{\Shoulder{KW}{x}{a(\delta)n(u)a(t)\kappa d(z,z)}} \ll_{N_1,N_2} \min(\norm[tz]_{\A}^{-1}, \norm[tz]_{\A}^{-N_1}) \min(\norm[z]_{\A}, \norm[z]_{\A}^{N_2}).
\label{ParWhiBd}
\end{equation}

	For $\Shoulder{KN}{x}{y}$, the same classification of orbits implies the decomposition
	$$ \Shoulder{KN}{x}{y} = \sideset{}{_{i=1}^4} \sum \Shoulder{KN_i}{x}{y}, \qquad \Shoulder{KN_i}{x}{y} := \int_{\F \backslash \A} \sum_{\xi \in \vO_i} \Psi(x^{-1} \xi n(u) y) du. $$
	For the first term, we have
\begin{align*}
	\Shoulder{KN_1}{x}{y} &= \sideset{}{_{\substack{ \alpha, \beta \in \F^{\times} \\ \gamma \in \F }}} \sum \int_{\A} \rpL_x \rpR_y \Psi \left( \begin{pmatrix} \alpha & \gamma \\ \beta & 0 \end{pmatrix} \begin{pmatrix} 1 & u \\ 0 & 1 \end{pmatrix} \right) du \\
	&= \sideset{}{_{\substack{ \alpha, \beta \in \F^{\times} \\ \gamma \in \F }}} \sum \int_{\A} \rpL_{d(\alpha,\beta)^{-1}} \rpL_x \rpR_y \Psi \begin{pmatrix} 1 & u + \gamma \\ 1 & u \end{pmatrix} du \\
	&= \sideset{}{_{\alpha, \beta \in \F^{\times}}} \sum \int_{\A} \sideset{}{_{\gamma \in \F}} \sum \OFour_2 \rpL_{d(\alpha,\beta)^{-1}} \rpL_x \rpR_y \Psi \begin{pmatrix} 1 & \gamma \\ 1 & u \end{pmatrix} \psi(\gamma u) du \\
	&= \sideset{}{_{\substack{ \alpha, \beta \in \F^{\times} \\ \gamma \in \F }}} \sum \OFour_4 \OFour_2 \rpL_{d(\alpha,\beta)^{-1}} \rpL_x \rpR_y \Psi \begin{pmatrix} 1 & \gamma \\ 1 & -\gamma \end{pmatrix}.
\end{align*}
	We distinguish the terms for which $\gamma \neq 0$ from $\gamma = 0$, splitting the above sum as
	$$ \sideset{}{_{\substack{ \alpha, \beta \in \F^{\times} \\ \gamma \in \F }}} \sum = \sideset{}{_{\substack{ \alpha, \beta \in \F^{\times} \\ \gamma \in \F^{\times} }}} \sum + \sideset{}{_{\substack{ \alpha, \beta \in \F^{\times} \\ \gamma = 0 }}} \sum. $$
	Denote the first resp. second part by $\Shoulder{KN_{1,1}}{x}{y}$ resp. $\Shoulder{KN_{1,2}}{x}{y}$. For the first part, we have the bound for any $N_1, N_2 > 1$
\begin{align}
	\extnorm{ \Shoulder{KN_{1,1}}{x}{n(u)a(t)\kappa d(z,z)} } &= \norm[z]_{\A}^{-2} \sideset{}{_{\substack{ \alpha, \beta \in \F^{\times} \\ \gamma \in \F^{\times} }}} \sum \extnorm{ \OFour_4 \OFour_2 \rpL_x \rpR_{\kappa} \Psi \begin{pmatrix} \alpha t z & \gamma (\alpha z)^{-1} \\ \beta t z & -\gamma (\beta z)^{-1} \end{pmatrix} } \nonumber \\
	&\leq \norm[z]_{\A}^{-2} \sideset{}{_{\substack{ \alpha, \beta \in \F^{\times} \\ \gamma, \delta \in \F^{\times} }}} \sum \extnorm{ \OFour_4 \OFour_2 \rpL_x \rpR_{\kappa} \Psi \begin{pmatrix} \alpha t z & \gamma (\alpha z)^{-1} \\ \beta t z & \delta (\beta z)^{-1} \end{pmatrix} } \nonumber \\
	&\ll_{N_1,N_2} \norm[z]_{\A}^{-2} \min(\norm[tz]_{\A}^{-1}, \norm[tz]_{\A}^{-N_1}) \min(\norm[z]_{\A}, \norm[z]_{\A}^{N_2}). \label{ParCstBd11}
\end{align}
	For the second part, we have for all $\Re s_0 > 2$
\begin{align}
	\Tree{KN_{1,2}}{\omega, s_0}{x}{n(u)a(t)\kappa} &:= \int_{\F^{\times} \backslash \A^{\times}} \Shoulder{KN_{1,2}}{x}{n(u)a(t)\kappa d(z,z)} \omega(z) \norm[z]_{\A}^{s_0+2} d^{\times}z \nonumber \\
	&= \omega(t)^{-1} \norm[t]_{\A}^{-s_0} \int_{\F^{\times} \backslash \A^{\times}} \left( \sideset{}{_{ \alpha, \beta \in \F^{\times} }} \sum \OFour_4 \OFour_2 \rpL_x \rpR_{\kappa} \Psi \begin{pmatrix} \alpha z & 0 \\ \beta z & 0 \end{pmatrix} \right) \omega(z) \norm[z]_{\A}^{s_0} d^{\times}z. \label{ParCstBd12}
\end{align}
	
\noindent For the second term, we have by Poisson summation
\begin{align*}
	\Shoulder{KN_2}{x}{y} &= \sideset{}{_{\substack{\beta \in \F^{\times} \\ \gamma \in \F }}} \sum \int_{\A} \Psi \left( x^{-1} \begin{pmatrix} 0 & \gamma \\ \beta & u \end{pmatrix} y \right) du = \sideset{}{_{\substack{\beta \in \F^{\times} \\ \gamma \in \F }}} \sum \OFour_2 \OFour_4 \rpL_x \rpR_y \Psi \begin{pmatrix} 0 & \gamma \\ \beta & 0 \end{pmatrix} \\
	&= \Shoulder{KN_{2,1}}{x}{y} + \Shoulder{KN_{2,2}}{x}{y}, \quad \text{where}
\end{align*}
\begin{align*} 
	\Shoulder{KN_{2,1}}{x}{y} &:= \sideset{}{_{\beta, \gamma \in \F^{\times}}} \sum \OFour_2 \OFour_4 \rpL_x \rpR_y \Psi \begin{pmatrix} 0 & \gamma \\ \beta & 0 \end{pmatrix} , 
\end{align*}
	$$ \Shoulder{KN_{2,2}}{x}{y} := \sideset{}{_{\beta \in \F^{\times}}} \sum \OFour_2 \OFour_4 \rpL_x \rpR_y \Psi \begin{pmatrix} 0 & 0 \\ \beta & 0 \end{pmatrix}. $$
	It follows easily that
\begin{align}
	\Shoulder{KN_{2,1}}{x}{n(u)a(t)\kappa d(z,z)} &= \norm[z]_{\A}^{-2} \sideset{}{_{\beta, \gamma \in \F^{\times}}} \sum \OFour_2 \OFour_4 \rpL_x \rpR_{\kappa} \Psi \begin{pmatrix} 0 & \gamma z^{-1} \\ \beta tz & 0 \end{pmatrix} \nonumber \\
	&\ll_{N_1,N_2} \norm[z]_{\A}^{-2} \min(\norm[tz]_{\A}^{-1}, \norm[tz]_{\A}^{-N_1}) \min(\norm[z]_{\A}, \norm[z]_{\A}^{N_2}), \label{ParCstBd21}
\end{align}
\begin{align}
	\Tree{KN_{2,2}}{\omega, s_0}{x}{n(u)a(t)\kappa} &:= \int_{\F^{\times} \backslash \A^{\times}} \Shoulder{KN_{2,2}}{x}{n(u)a(t)\kappa d(z,z)} \omega(z) \norm[z]_{\A}^{s_0+2} d^{\times}z \nonumber \\
	&= \omega(t)^{-1} \norm[t]_{\A}^{-s_0} \int_{\A^{\times}} \OFour_2 \OFour_4 \rpL_x \rpR_{\kappa} \Psi \begin{pmatrix} 0 & 0 \\ z & 0 \end{pmatrix} \omega(z) \norm[z]_{\A}^{s_0} d^{\times}z, \qquad \Re s_0 > 1. \label{ParCstBd22}
\end{align}

\noindent For the third term, we have similarly
\begin{align*}
	\Shoulder{KN_3}{x}{y} &= \sideset{}{_{\substack{\alpha \in \F^{\times} \\ \gamma \in \F }}} \sum \int_{\A} \Psi \left( x^{-1} \begin{pmatrix} \alpha & u \\ 0 & \gamma \end{pmatrix} y \right) du = \sideset{}{_{\substack{\alpha \in \F^{\times} \\ \gamma \in \F }}} \sum \OFour_2 \OFour_4 \rpL_x \rpR_y \Psi \begin{pmatrix} \alpha & 0 \\ 0 & \gamma \end{pmatrix} \\
	&= \Shoulder{KN_{3,1}}{x}{y} + \Shoulder{KN_{3,2}}{x}{y}, \quad \text{where}
\end{align*}
\begin{align*} 
	\Shoulder{KN_{3,1}}{x}{y} &:= \sideset{}{_{\alpha, \gamma \in \F^{\times}}} \sum \OFour_2 \OFour_4 \rpL_x \rpR_y \Psi \begin{pmatrix} \alpha & 0 \\ 0 & \gamma \end{pmatrix} , 
\end{align*}
	$$ \Shoulder{KN_{3,2}}{x}{y} := \sideset{}{_{\alpha \in \F^{\times}}} \sum \OFour_2 \OFour_4 \rpL_x \rpR_y \Psi \begin{pmatrix} \alpha & 0 \\ 0 & 0 \end{pmatrix}. $$
	It follows easily that
\begin{align}
	\Shoulder{KN_{3,1}}{x}{n(u)a(t)\kappa d(z,z)} &= \norm[z]_{\A}^{-2} \sideset{}{_{\alpha, \gamma \in \F^{\times}}} \sum \OFour_2 \OFour_4 \rpL_x \rpR_{\kappa} \Psi \begin{pmatrix} \alpha t z & 0 \\ 0 & \gamma z^{-1} \end{pmatrix} \nonumber \\
	&\ll_{N_1,N_2} \norm[z]_{\A}^{-2} \min(\norm[tz]_{\A}^{-1}, \norm[tz]_{\A}^{-N_1}) \min(\norm[z]_{\A}, \norm[z]_{\A}^{N_2}), \label{ParCstBd31}
\end{align}
\begin{align}
	\Tree{KN_{3,2}}{\omega, s_0}{x}{n(u)a(t)\kappa} &:= \int_{\F^{\times} \backslash \A^{\times}} \Shoulder{KN_{3,2}}{x}{n(u)a(t)\kappa d(z,z)} \omega(z) \norm[z]_{\A}^{s_0+2} d^{\times}z \nonumber \\
	&= \omega(t)^{-1} \norm[t]_{\A}^{-s_0} \int_{\A^{\times}} \OFour_2 \OFour_4 \rpL_x \rpR_{\kappa} \Psi \begin{pmatrix} z & 0 \\ 0 & 0 \end{pmatrix} \omega(z) \norm[z]_{\A}^{s_0} d^{\times}z, \qquad \Re s_0 > 1. \label{ParCstBd32}
\end{align}

\noindent For the fourth term, we leave the reader to check the simple formula
\begin{equation}
	\Tree{KN_4}{\omega,s_0}{x}{n(u)a(t)\kappa} = \int_{\F^{\times} \backslash \A^{\times}} \left( \sum_{(\alpha,\beta) \in \F^2-\{ \vec{0} \}} \rpL_x \rpR_{\kappa} \Psi \begin{pmatrix} 0 & \alpha z \\ 0 & \beta z \end{pmatrix} \right) \omega(z) \norm[z]_{\A}^{s_0+2} d^{\times}z, \quad \Re s_0 > 0.
\label{ParCstBd4}
\end{equation}

\begin{theorem}{(Godement-Jacquet pre-trace formula)}
	Assume $\Re s_0 > 2$.
	
\noindent (1) For fixed $x$, the function (see (\ref{KKDecomp})) $y \mapsto \Tree{\widetilde{KK}^{(2)}}{\omega, s_0}{x}{y}$ lies in $\Aut^{\freg}(\GL_2, \omega^{-1})$ and is bounded by $\ll \Ht(y)^{1 - \frac{\Re s_0}{2}}$ in any Siegel domain.
	
\noindent (2) Its Fourier inversion with respect to $y$ in $\intL^2(\GL_2(\F) \backslash \GL_2(\A), \omega^{-1})$ converges normally for $(x,y) \in (\GL_2(\F) \backslash \GL_2(\A))^2$, and takes the form
\begin{align*}
	\Tree{\widetilde{KK}^{(2)}}{\omega, s_0}{x}{y} &= \sideset{}{_{\substack{\pi \text{\ cuspidal} \\ \omega_{\pi} = \omega}}} \sum \Sq{\widetilde{KK}^{(2)}}{\omega, s_0}{\pi}{x}{y} + \sideset{}{_{\chi \in \widehat{\R_+ \F^{\times}\backslash \A^{\times}}}} \sum  \int_{-\infty}^{\infty} \Sq{\widetilde{KK}^{(2)}}{\omega, s_0}{\chi, i\tau}{x}{y} \frac{d\tau}{4\pi} \\
	&+ \frac{1}{\Vol([\PGL_2])} \sideset{}{_{\substack{ \eta \in \widehat{\F^{\times} \backslash \A^{\times}} \\ \eta^2 = \omega }}} \sum \left( \int_{\GL_2(\A)} \Psi(g) \eta(\det g) \norm[\det g]_{\A}^{\frac{s_0}{2}+1} dg \right) \eta(\det x) \overline{\eta(\det y)},
\end{align*}
	$$ \text{with} \quad \Sq{\widetilde{KK}^{(2)}}{\omega, s_0}{\pi}{x}{y} := \sideset{}{_{\varphi_1, \varphi_2 \in \Bas(\pi)}} \sum \Zeta \left( \frac{s_0+1}{2}, \Psi, \beta(\varphi_2, \varphi_1^{\vee}) \right) \varphi_1(x) \varphi_2^{\vee}(y), $$
	$$ \Sq{\widetilde{KK}^{(2)}}{\omega, s_0}{\chi, s}{x}{y} := \sum_{e_1,e_2 \in \Bas(\chi, \omega \chi^{-1})} \Zeta \left( \frac{s_0+1}{2}, \Psi, \beta_s(e_2, e_1^{\vee}) \right) \eis(s,e_1)(x) \eis(-s,e_2^{\vee})(y), $$
	where $\Bas(\pi)$ resp. $\Bas(\chi, \omega\chi^{-1})$ is an orthogonal basis (of $\gp{K}$-isotypic vectors) in $\pi$ resp. $V_{\chi, \omega\chi^{-1}}$.
\label{GJPreTrace}
\end{theorem}
\begin{proof}
	(1) The bounds (\ref{ParWhiBd}) \& (\ref{ParCstBd11}) \& (\ref{ParCstBd21}) \& (\ref{ParCstBd31}) and the expressions (\ref{ParCstBd12}) \& (\ref{ParCstBd22}) \& (\ref{ParCstBd32}) \& (\ref{ParCstBd4}) readily imply $y \mapsto \Tree{\widetilde{KK}}{\omega, s_0}{x}{y} \in \Aut^{\freg}(\GL_2, \omega^{-1})$. Since $y \mapsto \Tree{\widetilde{KK}^{(1)}}{\omega, s_0}{x}{y}$ is an Eisenstein series by (\ref{R1Eis}), we deduce that $y \mapsto \Tree{\widetilde{KK}^{(2)}}{\omega, s_0}{x}{y} \in \Aut^{\freg}(\GL_2, \omega^{-1})$. Moreover, still by (\ref{R1Eis}) we have
	$$ \Tree{\widetilde{KN}^{(1)}}{\omega,s_0}{x}{y} = \Tree{\widetilde{KN}_1^{(1)}}{\omega,s_0}{x}{y} + \Tree{\widetilde{KN}_2^{(1)}}{\omega,s_0}{x}{y}, $$
	where we introduce and compute (regarded as functions of $y$)
\begin{align}
	\Tree{\widetilde{KN}_1^{(1)}}{\omega,s_0}{x}{y} &:= \sum_{[\gamma_1] \in \gp{B}(\F) \backslash \GL_2(\F)} \int_{\A^{\times}} \rpL_{\gamma_1 x} \rpR_y \Psi \begin{pmatrix} 0 & z \\ 0 & 0 \end{pmatrix} \omega(z) \norm[z]_{\A}^{s_0+2} d^{\times}z \cdot \norm[\det x^{-1}y]_{\A}^{\frac{s_0}{2}+1} \nonumber \\
	&= \int_{\F^{\times} \backslash \A^{\times}} \left( \sum_{(\alpha,\beta) \in \F^2-\{ \vec{0} \}} \rpL_{x} \rpR_y \Psi \begin{pmatrix} 0 & \alpha z \\ 0 & \beta z \end{pmatrix} \right) \omega(z) \norm[z]_{\A}^{s_0+2} d^{\times}z \cdot \norm[\det x^{-1}y]_{\A}^{\frac{s_0}{2}+1} \nonumber \\
	&= \Tree{\widetilde{KN}_4}{\omega,s_0}{x}{y}; \label{R1Cst}
\end{align}
\begin{equation}
	\Tree{\widetilde{KN}_2^{(1)}}{\omega,s_0}{x}{y} := \int_{\A} \left( \sum_{[\gamma_1] \in \gp{B}(\F) \backslash \GL_2(\F)} \Tree{\widetilde{RR}^{(1)}}{\omega,s_0}{\gamma_1 x}{wn(u) y} \right) du \in \pi(\omega^{-1}\norm_{\A}^{-(1+s_0)/2},\norm_{\A}^{(s_0+1)/2}).
\label{R1CstIntw}
\end{equation}
	The desired bound of $\Tree{\widetilde{KK}^{(2)}}{\omega, s_0}{x}{y}$ follows readily by (\ref{ParWhiBd})-(\ref{R1CstIntw}). In particular its essential constant term (see Definition \ref{FinRegFuncDef}) is given by
	$$ \Tree{\widetilde{KN}^{(2),*}}{\omega, s_0}{x}{y} = \sum_{i=1}^3 \Tree{\widetilde{KN_{i,2}}}{\omega, s_0}{x}{y} - \Tree{\widetilde{KN}_2^{(1)}}{\omega,s_0}{x}{y}, $$
	where we recall
	$$ \Tree{\widetilde{KN_{i,j}}}{\omega, s_0}{x}{y} := \int_{\F^{\times} \backslash \A^{\times}} \Shoulder{KN_{i,j}}{x}{y d(z,z)} \omega(z) \norm[z]_{\A}^{s_0+2} d^{\times}z \cdot \norm[\det x^{-1}y]_{\A}^{\frac{s_0}{2}+1}. $$
	
\noindent (2) For any $X$ in the universal enveloping algebra of $\mathfrak{sl}_2(\A_{\infty})$, $\rpR_X \Psi \in \Sch(\Mat_2(\A))$. Hence the assertions in (1) remain valid if we replace $\Psi$ with $\rpR_X \Psi$. In other words, the function $y \mapsto \Tree{\widetilde{KK}^{(2)}}{\omega, s_0}{x}{y}$ is ``nice'' in the sense of Definition \ref{NiceFDef}, to which Theorem \ref{AutoFourInv} applies. We need to compute the automorphic Fourier coefficients.

\noindent If $\varphi_2 \in \Bas(\pi)$ for a cuspidal $\pi$ of central character $\omega$, then
\begin{align*}
	\int_{[\PGL_2]} \Tree{\widetilde{KK}^{(2)}}{\omega, s_0}{x}{y} \varphi_2(y) dy &= \int_{\GL_2(\F) \backslash \GL_2(\A)} \Shoulder{KK^{(2)}}{x}{y} \norm[\det x^{-1}y]_{\A}^{\frac{s_0}{2}+1} \varphi_2(y) dy \\
	&= \int_{\GL_2(\A)} \Psi(g) \norm[\det g]_{\A}^{\frac{s_0}{2}+1} \varphi_2(xg) dg \\
	&= \sideset{}{_{\varphi_1 \in \Bas(\pi)}} \sum \Zeta \left( \frac{s_0+1}{2}, \Psi, \beta(\varphi_2, \varphi_1^{\vee}) \right) \varphi_1(x),
\end{align*}
	justifying the automorphic Fourier coefficients for the cuspidal spectrum.
	
\noindent Take $e_2 \in \Bas(\chi, \omega\chi^{-1})$. Proposition \ref{VRI} identifies the automorphic Fourier coefficient
	$$ \int_{[\PGL_2]} \Tree{\widetilde{KK}^{(2)}}{\omega, s_0}{x}{y} \eis(i\tau, e_2)(y) dy $$
	with the analytically continued value at $s = i\tau$ of
\begin{align*} 
	&\qquad R\left(s, \Tree{\widetilde{KK}^{(2)}}{\omega, s_0}{x}{\cdot}, e_2 \right) \\
	&:= \int_{\F^{\times} \backslash \A^{\times} \times \gp{K}} \left( \Tree{\widetilde{KN}^{(2)}}{\omega, s_0}{x}{a(t)\kappa} - \Tree{\widetilde{KN}^{(2),*}}{\omega, s_0}{x}{a(t)\kappa} \right) e_2(\kappa) \chi(t) \norm[t]_{\A}^{s-\frac{1}{2}} d\kappa d^{\times}t \\
	&= \sum_{i=1}^3 \int_{\F^{\times} \backslash \A^{\times} \times \gp{K}} \Tree{\widetilde{KN_{i,1}}}{\omega, s_0}{x}{a(t)\kappa} e_2(\kappa) \chi(t) \norm[t]_{\A}^{s-\frac{1}{2}} d\kappa d^{\times}t, \qquad \Re s \gg 1.
\end{align*}
	Inserting (\ref{ParCstBd31}), we get
\begin{align*} 
	f_s(x) &:= \int_{\F^{\times} \backslash \A^{\times} \times \gp{K}} \Tree{\widetilde{KN_{3,1}}}{\omega, s_0}{x}{a(t)\kappa} e_2(\kappa) \chi(t) \norm[t]_{\A}^{s-\frac{1}{2}} d\kappa d^{\times}t \\
	&= \int_{(\A^{\times})^2 \times \gp{K}} \OFour_2 \OFour_4 \rpL_x \rpR_{\kappa} \Psi \begin{pmatrix} tz & \\ & z^{-1} \end{pmatrix} e_2(\kappa) \omega(z) \norm[z]_{\A}^{s_0} \chi(t) \norm[t]_{\A}^{s+\frac{s_0+1}{2}} d\kappa d^{\times}z d^{\times}t \cdot \norm[\det x]_{\A}^{-\frac{s_0}{2}-1}.
\end{align*}
	The easy to check property
	$$ \OFour_2 \OFour_4 \rpL_{n(u)d(t_1z_0,z_0)} \rpL_x \rpR_{\kappa} \Psi \begin{pmatrix} tz & \\ & z^{-1} \end{pmatrix} = \OFour_2 \OFour_4 \rpL_x \rpR_{\kappa} \Psi \begin{pmatrix} t_1^{-1}t z_0^{-1}z & \\ & z_0 z^{-1} \end{pmatrix} \cdot \norm[z_0]_{\A}^2 \cdot \norm[t_1]_{\A}, $$
	readily implies
	$$ f_s \left( \begin{pmatrix} tz & u \\ & z \end{pmatrix} x \right) = \omega(z) \chi(t) \norm[t]_{\A}^{s+\frac{1}{2}} f_s(x). $$
	Hence $f_s \in \pi(\chi \norm_{\A}^s, \omega \chi^{-1} \norm_{\A}^{-s})$ (not necessarily a flat section). Note that for any $\alpha \in \F$
\begin{align*}
	\OFour_2 \OFour_4 \rpL_{wn(\alpha)} \rpL_x \rpR_{\kappa} \Psi \begin{pmatrix} tz & \\ & z^{-1} \end{pmatrix} &= \int_{\A^2} \rpL_x \rpR_{\kappa} \Psi \begin{pmatrix} -\alpha tz & x_4 - \alpha x_2 \\ tz & x_2 \end{pmatrix} \psi(-x_4 z^{-1}) dx_2 dx_4 \\
	&= \OFour_2 \OFour_4 \rpL_x \rpR_{\kappa} \Psi \begin{pmatrix} - \alpha tz & z^{-1} \\ tz & \alpha z^{-1} \end{pmatrix},
\end{align*}
	From which we easily verify
	$$ f_s(wx) = \int_{\F^{\times} \backslash \A^{\times} \times \gp{K}} \Tree{\widetilde{KN_{2,1}}}{\omega, s_0}{x}{a(t)\kappa} e_2(\kappa) \chi(t) \norm[t]_{\A}^{s-\frac{1}{2}} d\kappa d^{\times}t \cdot \norm[\det x]_{\A}^{-\frac{s_0}{2}-1}, $$
	$$ \sum_{\alpha \in \F^{\times}} f_s(wn(\alpha)x) = \int_{\F^{\times} \backslash \A^{\times} \times \gp{K}} \Tree{\widetilde{KN_{1,1}}}{\omega, s_0}{x}{a(t)\kappa} e_2(\kappa) \chi(t) \norm[t]_{\A}^{s-\frac{1}{2}} d\kappa d^{\times}t \cdot \norm[\det x]_{\A}^{-\frac{s_0}{2}-1}. $$
	We obtain, with absolute convergence for $\Re s \gg 1$
	$$ R\left(s, \Tree{\widetilde{KK}^{(2)}}{\omega, s_0}{x}{\cdot}, e_2 \right) = \sum_{\xi \in \gp{B}(\F) \backslash \GL_2(\F)} f_s(\gamma x) = \sum_{e_1 \in \Bas(\chi, \omega \chi^{-1})} \int_{\gp{K}} f_s(\kappa) e_1^{\vee}(\kappa) d\kappa \cdot \eis(s, e_1)(x). $$
	Note that (the global version of) \cite[(11.9.4)]{GoJ11a} gives
\begin{equation} 
	\Zeta(s', \Psi, \beta_s(e_2, e_1^{\vee})) = \int_{(\A^{\times})^2} \Four[2]{ {}_{e^{\vee}_1} \! \Psi_{e_2} } \begin{pmatrix} t_1 & 0 \\ 0 & t_2 \end{pmatrix} \chi(t_1) \norm[t_1]_{\A}^{s'+s} \omega\chi^{-1}(t_2) \norm[t_2]_{\A}^{s'-s} d^{\times}t_1 d^{\times}t_2, 
\label{GJZetaPS}
\end{equation}
	$$ \text{where} \quad {}_{e_1^{\vee}} \! \Psi_{e_2}(x) := \int_{\gp{K} \times \gp{K}} e_1^{\vee}(\kappa_1) \Psi(\kappa_1^{-1} x \kappa_2) e_2(\kappa_2) d\kappa_1 d\kappa_2. $$
	By a change of variables and the global functional equation of Tate's integral, $f_s(x)$ is equal to
	$$ \int_{(\A^{\times})^2 \times \gp{K}} \OFour_2 \rpL_x \rpR_{\kappa} \Psi \begin{pmatrix} x_1 & \\ & x_4 \end{pmatrix} e_2(\kappa) \omega\chi^{-1}(x_4) \norm[x_4]_{\A}^{\frac{s_0+1}{2}-s} \chi(x_1) \norm[x_1]_{\A}^{s+\frac{s_0+1}{2}} d\kappa d^{\times}x_1 d^{\times}x_4 \cdot \norm[\det x]_{\A}^{-\frac{s_0}{2}-1}. $$
	We deduce that
	$$ R\left(s, \Tree{\widetilde{KK}^{(2)}}{\omega, s_0}{x}{\cdot}, e_2 \right) = \sum_{e_1 \in \Bas(\chi, \omega \chi^{-1})} \Zeta \left( \frac{s_0+1}{2}, \Psi, \beta_s(e_2, e_1^{\vee}) \right) \eis(s,e_1)(x), $$
	justifying the automorphic Fourier coefficients for the continuous spectrum.
	
\noindent We leave the justification for the residue spectrum as an exercise (unimportant for this paper).
\end{proof}
	
	We introduce (temporarily) the absolutely convergent integral for $\lambda \in D$
\begin{equation}
	DS_0(\lambda; \Psi) := \int_{(\F^{\times} \backslash \A^{\times})^2} \Tree{\widetilde{\Delta \Delta}^{(1)}}{\omega,s_0}{a(t_1)}{a(t_2)} \chi_1(t_1) \norm[t_1]_{\A}^{s_1-1} \chi_2(t_2) \norm[t_2]_{\A}^{s_2+1} d^{\times}t_1 d^{\times}t_2.
\label{DS0Bis}
\end{equation}
	An obvious variant of Theorem \ref{GJPreTrace} (just as Theorem \ref{AutoFourInvV} v.s. Theorem \ref{AutoFourInv}) implies Proposition \ref{GoalMotoDis} (1) via (recall (\ref{SpecMDCusp}) \& (\ref{SpecMDCont})) absolutely convergent sum \& integral for $\lambda \in D$ of
\begin{equation}
	\Theta(\lambda, \Psi) - DS_0(\lambda,\Psi) = \sum_{\substack{\pi \text{ cuspidal} \\ \omega_{\pi} = \omega }} M_3(\lambda,\Psi \mid \pi) + \sum_{\chi \in \widehat{\R_+ \F^{\times} \backslash \A^{\times}}} \int_{-\infty}^{\infty} M_3(\lambda, \Psi \mid \chi, i\tau) \frac{d\tau}{4\pi}. 
\label{SpecSide}
\end{equation}

\section{Fourth Moment (Geometric) Side}
\label{4thMSec}

	We would like to write in another way the following distribution
	$$ \Shoulder{\Delta \Delta}{\id}{\id} = \Shoulder{K K}{\id}{\id} - \Shoulder{K N}{\id}{\id} - \Shoulder{N \Delta}{\id}{\id}. $$
\begin{remark}
	For simplicity of notation, 
\begin{itemize}
	\item the summation symbol ``$\sum$'' below means, by default, summing over $\xi_i \in \F$ for those variables $\xi_i$ appearing in the summands. Only extra conditions, such as ``$\xi_1 \in \F^{\times}$'', will be explicitly written;
	\item the summation symbol ``$\Vsum$'' below means, by default, summing over $\xi_i \in \F^{\times}$ for those variables $\xi_i$ appearing in the summands. Only extra conditions, such as ``$\xi_1 \in \F$'', will be explicitly written.
\end{itemize}
\label{SumConv}
\end{remark}

\begin{lemma}
	We have
	$$ \Shoulder{K N}{\id}{\id} = \sum_{\xi_1\xi_2 + \xi_3 \xi_4 = 0} \OFour_2 \OFour_4 \Psi \begin{pmatrix} \xi_1 & \xi_2 \\ \xi_3 & \xi_4 \end{pmatrix} - \Psi(0), \quad \Shoulder{N K}{\id}{\id} = \sum_{\xi_1\xi_3 + \xi_2 \xi_4 = 0} \OFour_1 \OFour_2 \Psi \begin{pmatrix} \xi_1 & \xi_2 \\ \xi_3 & \xi_4 \end{pmatrix} - \Psi(0). $$
\label{KNF}
\end{lemma}
\begin{proof}
	Applying the Poisson summation formula with respect to the variables $\xi_2, \xi_4$, we get
	$$ \Shoulder{KK}{\id}{\begin{pmatrix} 1 & u \\ 0 & 1 \end{pmatrix}} + \Psi(0) = \sum \Psi \begin{pmatrix} \xi_1 & \xi_2 + u \xi_1 \\ \xi_3 & \xi_4 + u \xi_3 \end{pmatrix} = \sum \OFour_2 \OFour_4 \Psi \begin{pmatrix} \xi_1 & \xi_2 \\ \xi_3 & \xi_4 \end{pmatrix} \psi(u(\xi_1 \xi_2 + \xi_3 \xi_4)). $$
	Integrating against $u \in \F \backslash \A$ gives the first equation. The second one is proved similarly.
\end{proof}

\begin{lemma}
	We have
\begin{align*}
	\Shoulder{N \Delta}{\id}{\id} &= \Vsum \OFour_1 \OFour_2 \OFour_4 \Psi \begin{pmatrix} 0 & 0 \\ \xi_3 & \xi_4 \end{pmatrix} + \Vsum \OFour_2 \Psi \begin{pmatrix} \xi_1 & \xi_2 \\ 0 & 0 \end{pmatrix} + \\
	&\quad \Vsum \int_{\A} \OFour_1 \OFour_2 \Psi \begin{pmatrix} -\xi_2 v & \xi_2 \\ \xi_3 & \xi_3 v \end{pmatrix} \psi(-v \xi) dv.
\end{align*}
\label{NDF}
\end{lemma}
\begin{proof}
	We combine the following elementary relation
\begin{align*} 
	\OFour_1 \OFour_2 \rpR(n(u)) \Psi \begin{pmatrix} x_1 & x_2 \\ x_3 & x_4 \end{pmatrix} &= \int_{\A^2} \Psi \begin{pmatrix} y_1 & y_2 + uy_1 \\ x_3 & x_4 + ux_3 \end{pmatrix} \psi(-y_1x_1-y_2x_2) dy_1 dy_2 \\
	&= \int_{\A^2} \Psi \begin{pmatrix} y_1 & y_2 \\ x_3 & x_4 + ux_3 \end{pmatrix} \psi(-y_1(x_1-ux_2)-y_2x_2) dy_1 dy_2 \\
	&= \OFour_1 \OFour_2 \Psi \begin{pmatrix} x_1 - ux_2 & x_2 \\ x_3 & x_4 + ux_3 \end{pmatrix}
\end{align*}
	with the second equation in Lemma \ref{KNF} to get
\begin{align*}
	\Shoulder{N K}{\id}{n(u)} + \Psi(0) &= \sum_{\xi_1\xi_3 + \xi_2 \xi_4 = 0} \OFour_1 \OFour_2 \Psi \begin{pmatrix} \xi_1 - u\xi_2 & \xi_2 \\ \xi_3 & \xi_4 + u\xi_3 \end{pmatrix} \\
	&= \sum \OFour_1 \OFour_2 \Psi \begin{pmatrix} \xi_1 & 0 \\ 0 & \xi_4 \end{pmatrix} + \sum_{\xi_3 \neq 0} \OFour_1 \OFour_2 \OFour_4 \Psi \begin{pmatrix} 0 & 0 \\ \xi_3 & \xi_4 \end{pmatrix} \psi(u \xi_3 \xi_4) + \\
	&\quad \sum_{\xi_2 \neq 0} \OFour_2 \Psi \begin{pmatrix} \xi_1 & \xi_2 \\ 0 & 0 \end{pmatrix} \psi(u \xi_1 \xi_2) + \sum_{\xi_2 \xi_3 \neq 0} \OFour_1 \OFour_2 \Psi \begin{pmatrix} -\xi_2(u+\xi) & \xi_2 \\ \xi_3 & \xi_3(u+\xi) \end{pmatrix},
\end{align*}
	where we distinguish the cases $\xi_2 = \xi_3 = 0$, $\xi_2 = 0$ \& $\xi_3 \neq 0$, $\xi_2 \neq 0$ \& $\xi_3 = 0$ and $\xi_2 \xi_3 \neq 0$ and apply partial Poisson summation in passing from the the first line to the second. The desired equation then follows at once by definition
	$$ \Shoulder{N \Delta}{\id}{\id} = \Shoulder{N K}{\id}{\id} - \int_{\F \backslash \A} \Shoulder{N K}{\id}{n(u)} du. $$
\end{proof}

\begin{lemma}
	We have the decomposition of tempered distributions
	$$ \Shoulder{\Delta \Delta}{\id}{\id} = \sideset{}{_{i=1}^4} \sum \Shoulder{\Delta \Delta_i}{\id}{\id}, $$
	where each term is defined as
	$$ \Shoulder{\Delta \Delta_1}{\id}{\id} := \Vsum \OFour_2 \OFour_3 \Psi \begin{pmatrix} \xi_1 & \xi_2 \\ \xi_3 & \xi_4 \end{pmatrix}, $$
\begin{align*} 
	\Shoulder{\Delta \Delta_2}{\id}{\id} &:= \Vsum \OFour_2 \OFour_4 \Psi \begin{pmatrix} 0 & \xi_2 \\ \xi_3 & \xi_4 \end{pmatrix} + \Vsum \OFour_1 \OFour_2 \OFour_4 \Psi \begin{pmatrix} \xi_1 & 0 \\ \xi_3 & \xi_4 \end{pmatrix} \\
	&\quad + \Vsum \OFour_2 \OFour_3 \Psi \begin{pmatrix} \xi_1 & \xi_2 \\ 0 & \xi_4 \end{pmatrix} + \Vsum \OFour_2 \Psi \begin{pmatrix} \xi_1 & \xi_2 \\ \xi_3 & 0 \end{pmatrix},
\end{align*}
	$$ \Shoulder{\Delta \Delta_3}{\id}{\id} := - \Vsum_{\xi_1 \xi_2 + \xi_3 \xi_4 = 0} \OFour_2 \OFour_4 \Psi \begin{pmatrix} \xi_1 & \xi_2 \\ \xi_3 & \xi_4 \end{pmatrix}, $$
	$$ \Shoulder{\Delta \Delta_4}{\id}{\id} := - \Vsum \int_{\A} \OFour_1 \OFour_2 \Psi \begin{pmatrix} -\xi_2 u & \xi_2 \\ \xi_3 & \xi_3 u \end{pmatrix} \psi(-u \xi) du. $$
\label{DDDecomp}
\end{lemma}
\begin{proof}
	Combining Lemma \ref{KNF} and \ref{NDF}, and applying partial Poisson summation formula to $\Shoulder{K K}{\id}{\id}$, we rewrite $\Shoulder{\Delta \Delta}{\id}{\id}$ as
\begin{align*}
	&\Vsum \OFour_2 \OFour_4 \Psi \begin{pmatrix} \xi_1 & \xi_2 \\ \xi_3 & \xi_4 \end{pmatrix} + \sum_{\substack{\xi_1 \xi_2 \xi_3 \xi_4 = 0 \\ \xi_1 \xi_2 + \xi_3 \xi_4 \neq 0}} \OFour_2 \OFour_4 \Psi \begin{pmatrix} \xi_1 & \xi_2 \\ \xi_3 & \xi_4 \end{pmatrix} - \Vsum_{\xi_1 \xi_2 + \xi_3 \xi_4 = 0} \OFour_2 \OFour_4 \Psi \begin{pmatrix} \xi_1 & \xi_2 \\ \xi_3 & \xi_4 \end{pmatrix} \\
	&\quad - \Vsum \OFour_1 \OFour_2 \OFour_4 \Psi \begin{pmatrix} 0 & 0 \\ \xi_3 & \xi_4 \end{pmatrix} - \Vsum \OFour_2 \Psi \begin{pmatrix} \xi_1 & \xi_2 \\ 0 & 0 \end{pmatrix} - \Vsum \int_{\A} \OFour_1 \OFour_2 \Psi \begin{pmatrix} -\xi_2 v & \xi_2 \\ \xi_3 & \xi_3 v \end{pmatrix} \psi(-v \xi) dv.
\end{align*}
	The two terms in the last column are $\Shoulder{\Delta \Delta_3}{\id}{\id}$ resp. $\Shoulder{\Delta \Delta_4}{\id}{\id}$. We identify terms in the first two columns with $\Shoulder{\Delta \Delta_1}{\id}{\id} + \Shoulder{\Delta \Delta_2}{\id}{\id}$ by summing the following partial Poisson summation
	$$ \Vsum \OFour_2 \OFour_4 \Psi \begin{pmatrix} \xi_1 & \xi_2 \\ \xi_3 & \xi_4 \end{pmatrix} + \Vsum \OFour_2 \OFour_4 \Psi \begin{pmatrix} \xi_1 & \xi_2 \\ \xi_3 & 0 \end{pmatrix} = \Vsum \OFour_2 \Psi \begin{pmatrix} \xi_1 & \xi_2 \\ \xi_3 & \xi_4 \end{pmatrix} + \Vsum \OFour_2 \Psi \begin{pmatrix} \xi_1 & \xi_2 \\ \xi_3 & 0 \end{pmatrix}, $$
	$$ \Vsum \OFour_2 \Psi \begin{pmatrix} \xi_1 & \xi_2 \\ \xi_3 & \xi_4 \end{pmatrix} + \Vsum \OFour_2 \Psi \begin{pmatrix} \xi_1 & \xi_2 \\ 0 & \xi_4 \end{pmatrix} = \Vsum \OFour_2 \OFour_3 \Psi \begin{pmatrix} \xi_1 & \xi_2 \\ \xi_3 & \xi_4 \end{pmatrix} + \Vsum \OFour_2 \OFour_3 \Psi \begin{pmatrix} \xi_1 & \xi_2 \\ 0 & \xi_4 \end{pmatrix}, $$
	$$ \Vsum \OFour_2 \OFour_4 \Psi \begin{pmatrix} \xi_1 & \xi_2 \\ 0 & \xi_4 \end{pmatrix} + \Vsum \OFour_2 \OFour_4 \Psi \begin{pmatrix} \xi_1 & \xi_2 \\ 0 & 0 \end{pmatrix} = \Vsum \OFour_2 \Psi \begin{pmatrix} \xi_1 & \xi_2 \\ 0 & \xi_4 \end{pmatrix} + \Vsum \OFour_2 \Psi \begin{pmatrix} \xi_1 & \xi_2 \\ 0 & 0 \end{pmatrix}, $$
	$$ \Vsum \OFour_2 \OFour_4 \Psi \begin{pmatrix} \xi_1 & 0 \\ \xi_3 & \xi_4 \end{pmatrix} + \Vsum \OFour_2 \OFour_4 \Psi \begin{pmatrix} 0 & 0 \\ \xi_3 & \xi_4 \end{pmatrix} = \Vsum \OFour_1 \OFour_2 \OFour_4 \Psi \begin{pmatrix} \xi_1 & 0 \\ \xi_3 & \xi_4 \end{pmatrix} + \Vsum \OFour_1 \OFour_2 \OFour_4 \Psi \begin{pmatrix} 0 & 0 \\ \xi_3 & \xi_4 \end{pmatrix}. $$
\end{proof}

\noindent According to the above decomposition, we define for $1 \leq i \leq 4$
\begin{equation} 
	II_i(\lambda, \Psi) := \int_{ (\F^{\times} \backslash \A^{\times} )^3} \Shoulder{\Delta \Delta_i}{a(t_1)  }{d(t_2 z,z)} \omega(z) \norm[z]_{\A}^{s_0+2} \chi_1(t_1) \norm[t_1]_{\A}^{s_1-1} \chi_2(t_2) \norm[t_2]_{\A}^{s_2+1} d^{\times}z d^{\times}t_1 d^{\times}t_2 . 
\label{GMD4thDecomp}
\end{equation}
\begin{lemma}
	Each integral as (\ref{GMD4thDecomp}) is absolutely convergent in the following domain
	$$ D' := \left\{ (s_0, s_1, s_2) \in \C^3 \ \middle| \ \Re s_0 > 2, \Re s_1 > 1, \Re s_2 > \Re s_0 + \Re s_1 + 1 \right\}. $$
\label{GeomConv}
\end{lemma}
\begin{remark}
	Recall the new Hecke characters $\eta_i$ and complex numbers $s_i'$ are given by
	$$ \eta_1 = \omega^{-1}\chi_1^{-1}\chi_2, \quad \eta_2 = \omega^{-1}\chi_2, \quad \eta_3 = \omega^{-1}\chi_1^{-1}, \quad \eta_4 = \mathbbm{1}, $$
	$$ s_1' = s_2-s_0-s_1, \quad s_2' = s_2-s_0, \quad s_3' = -s_0-s_1, \quad s_4' = 0. $$
	It is convenient to record
	$$ D' = \left\{ (s_1', s_2', s_3') \in \C^3 \ \middle| \ -\Re s_3' - 2 > \Re s_2' - \Re s_1' > 1, \Re s_1' > 1 \right\}. $$
\end{remark}
\begin{proof}
	We give the treatment for $II_1$ and $II_2$ here. The remaining two can be directly treated in a similar way. But we prefer to give another proof in \S \ref{4MAC}.

\noindent (1) $II_1$. Assume $(s_0, s_1, s_2) \in D'$. Some basic properties of Fourier transform yield
	$$ \Shoulder{\Delta \Delta_1}{a(t_1)}{d(t_2z,z)} = \frac{\norm[t_1]_{\A}}{\norm[t_2 z^2]_{\A}} \Vsum \OFour_2 \OFour_3 \Psi \left( \begin{matrix} \xi_1 t_1^{-1} t_2 z & \xi_2 t_1z^{-1} \\ \xi_3 t_2^{-1} z^{-1} & \xi_4 z \end{matrix} \right), $$
	Consider the following function on $(\F^{\times} \backslash \A^{\times})^4$
	$$ f(z,t_1,t_2,t) := \frac{\norm[t_1]_{\A}}{\norm[t_2 z^2]_{\A}} \Vsum \OFour_2 \OFour_3 \Psi \left( \begin{matrix} \xi_1 t_1^{-1} t_2 z & \xi_2 t_1z^{-1} \\ \xi_3 t_2^{-1} z^{-1} & \xi_4 t z \end{matrix} \right). $$
	It is smooth in each variable. With the change of variables
	$$ \begin{pmatrix} t_1^{-1} t_2 z & t_1z^{-1} \\ t_2^{-1} z^{-1} & tz \end{pmatrix} = \begin{pmatrix} x_1 & x_2 \\ x_3 & x_4 \end{pmatrix} $$
	we have for any real numbers $M_j$ the relation
	$$ \norm[t_1]_{\A}^{-M_1} \norm[t_2]_{\A}^{-M_2} \norm[z]_{\A}^{-M_3} \norm[t]_{\A}^{-M_4} = \norm[x_1]_{\A}^{-(M_4+M_2-M_3-M_1)} \norm[x_2]_{\A}^{-(M_4+M_2-M_3)} \norm[x_3]_{\A}^{-(M_4-M_3-M_1)} \norm[x_4]_{\A}^{-M_4}. $$
	Let $\epsilon > 0$ be small and choose $M_j$ to be
	$$ M_1 \in \{ \Re s_1 \pm \epsilon \}, \quad M_2 \in \{ \Re s_2 \pm \epsilon \}, \quad M_3 \in \{ \Re s_0 \pm \epsilon \}, \quad M_4 + \Re s_i' \geq 1+3\epsilon, \forall 1 \leq i \leq 4. $$
	Proposition \ref{PoissonSEst} implies the following bound
	$$ \norm[f(t_1,t_2,z,t)] \ll \frac{\norm[t_1]_{\A}}{\norm[t_2 z^2]_{\A}} \norm[t_1]_{\A}^{-M_1} \norm[t_2]_{\A}^{-M_2} \norm[z]_{\A}^{-M_3} \norm[t]_{\A}^{-M_4}. $$
	Consequently, we can break the integral defining the following function on $\F^{\times} \backslash \A^{\times}$
	$$ g(t) := \int_{(\F^{\times} \backslash \A^{\times})^3} f(z,t_1,t_2,t) \omega(z) \norm[z]_{\A}^{s_0+2} \chi_1(t_1) \norm[t_1]_{\A}^{s_1-1} \chi_2(t_2) \norm[t_2]_{\A}^{s_2+1} d^{\times}z d^{\times} t_1 d^{\times}t_2 $$
	according as $\norm[t_1]_{\A}, \norm[t_2]_{\A}, \norm[z]_{\A} \geq 1$ or $< 1$ ($8$ cases in total), apply to each sub-integral suitable $M_j$ chosen above, and see $g(t)$ is well-defined with absolutely convergent integral, smooth, and satisfies the bound
	$$ \norm[g(t)] \ll \norm[t]_{\A}^{-M_4}, \quad M_4 + \Re s_i' > 1, \forall 1 \leq i \leq 4. $$
	The same bound holds if we replace $g(t)$ by $D^n.g(t)$ for the invariant differential $D$ on $s_{\F}(\R_+)$, where $s_{\F}$ is a/any section map of the adelic norm map on $\F^{\times} \backslash \A^{\times}$. Taking into account the decomposition as a direct product of
	$$ \F^{\times} \backslash \A^{\times} \simeq \F^{\times} \backslash \A^{(1)} \times \R_{>0}, $$
	where $\A^{(1)}$ is the subgroup of elements in $\A^{\times}$ with adelic norm $1$ (hence $\F^{\times} \backslash \A^{(1)}$ is compact), we see that Mellin inversion on $\F^{\times} \backslash \A^{\times}$ holds for $g(t)$. In particular, we have
	$$ g(1) = \sideset{}{_{\chi}} \sum \int_{\Re s = c} \int_{\F^{\times} \backslash \A^{\times}} g(t) \chi(t) \norm[t]_{\A}^s d^{\times}t \frac{ds}{2\pi i}, $$
	where $c+\Re s_i' > 1$ and $\chi$ traverses the unitary characters of $\F^{\times} \backslash \A^{(1)} \simeq \R_{>0}\F^{\times} \backslash \A^{\times}$. The innermost integral is identified as
\begin{align*}
	&\quad \int_{(\F^{\times} \backslash \A^{\times})^4} f(z,t_1,t_2,t) \omega(z) \norm[z]_{\A}^{s_0+2} \chi_1(t_1) \norm[t_1]_{\A}^{s_1-1} \chi_2(t_2) \norm[t_2]_{\A}^{s_2+1} \chi(t) \norm[t]_{\A}^s d^{\times}z d^{\times} t_1 d^{\times}t_2 d^{\times}t \\
	&= \int_{(\A^{\times})^4} \OFour_2 \OFour_3 \Psi \left( \begin{matrix} x_1 & x_2 \\ x_3 & x_4 \end{matrix} \right) \prod_{i=1}^4 \eta_i \chi(x_i) \norm[x_i]_{\A}^{s+s_i'} d^{\times}x_i.
\end{align*}
	Hence we get
\begin{equation}
	g(1) = II_1(\lambda, \Psi) = \sideset{}{_{\chi}} \sum \int_{\Re s = c} \int_{(\A^{\times})^4} \OFour_2 \OFour_3 \Psi \begin{pmatrix} x_1 & x_2 \\ x_3 & x_4 \end{pmatrix} \left( \prod_{i=1}^4 \eta_i \chi(x_i) \norm[x_i]_{\A}^{s+s_i'} d^{\times}x_i \right) \frac{ds}{2\pi i}. 
\label{II_1CE}
\end{equation}
	
\noindent (2) $II_2$. $\Delta \Delta_2$ has an obvious decomposition into the sum of four terms $\Delta \Delta_{2,i}$, yielding
	$$ II_2(\lambda, \Psi) = \sum_{i=5}^8 DG_i(\lambda, \Psi), $$
	where $DG_i$'s are given by the same formulas as (\ref{II_21CE}) to (\ref{II_24CE}), absolutely convergent in $D'$.
\end{proof}

	We summarized the above discussion in the following equality for $\lambda \in D'$ with absolute convergence
\begin{equation}
	\Theta(\lambda, \Psi) = \sideset{}{_{\substack{i=1 \\ i \neq 2}}^4} \sum II_i(\lambda, \Psi) + \sideset{}{_{j=5}^8} \sum DG_j(\lambda, \Psi),
\label{GeomSide}
\end{equation}
	where the terms on the right hand side are given by (\ref{GMD4thDecomp}), (\ref{II_1CE}) and (\ref{II_21CE}) - (\ref{II_24CE}).

\section{Analytic Continuation}
\label{AnalContSec}

	\subsection{Third Moment Side}
	
	Based on (\ref{SpecSide}), we shall give a meromorphic continuation of $\Theta(\lambda, \Psi)$ from $\lambda \in D$ to a small neighborhood of $\vec{0}$, precisely to
	$$ D_0 := \left\{ (s_1,s_2, s_0) \ \middle| \ \norm[s_0], \norm[s_1], \norm[s_2] < 1/6 \right\}. $$
	
\begin{lemma}
	Recall the continuous spectral Motohashi distributions (\ref{SpecMDCont}). There is a functional equation
	$$ \Sq{\widetilde{KK}}{\omega, s_0}{\chi, s}{x}{y} = \Sq{\widetilde{KK}}{\omega, s_0}{\omega\chi^{-1}, -s}{x}{y}, $$
	$$ M_3(\lambda, \Psi \mid \chi, s) = M_3(\lambda, \Psi \mid \omega\chi^{-1}, -s). $$
\label{CMDFE}
\end{lemma}
\begin{proof}
	The second equality obviously follows from the first one. The intertwining operator
	$$ \Intw: \pi(\chi \norm_{\A}^s, \omega\chi^{-1} \norm_{\A}^{-s}) \to \pi(\omega\chi^{-1} \norm_{\A}^{-s}, \chi \norm_{\A}^s) $$
	has the following two properties:
\begin{itemize}
	\item[(1)] (See \cite[(4.4) \& (4.17)]{GJ79}) For any $e \in \pi(\chi, \omega\chi^{-1}), e' \in \pi(\chi^{-1}, \omega^{-1}\chi)$,
	$$ \Pairing{\Intw e(s)}{\Intw e'(-s)} = \Pairing{e(s)}{e'(-s)}. $$
	\item[(2)] (See \cite[(5.15)]{GJ79}) For any $e \in \pi(\chi, \omega\chi^{-1})$,
	$$ \eis(s,e) = \eis(e(s)) = \eis(\Intw e(s)). $$
\end{itemize}
	The first property, together with the $\GL_2(\A)$-intertwining property, implies
	$$ \beta_s(e_2,e_1^{\vee}) =: \beta(e_2(s), e_1^{\vee}(-s)) = \beta(\Intw e_2(s), \Intw e_1^{\vee}(-s)). $$
	Moreover, if $e^{\vee}$ is the dual element of $e$, then $\Intw e^{\vee}(-s)$ is the dual element of $\Intw e(s)$. Hence we get
\begin{align*} 
	&\quad \sum_{e_1,e_2 \in \Bas(\chi, \omega \chi^{-1})} \Zeta \left( \frac{s_0+1}{2}, \Psi, \beta(e_2(s), e_1^{\vee}(-s)) \right) \cdot \Intw e_1(s)(x) \cdot \Intw e_2^{\vee}(-s) (y) \\
	&= \sum_{e_1,e_2 \in \Bas(\chi, \omega \chi^{-1})} \Zeta \left( \frac{s_0+1}{2}, \Psi, \beta(\Intw e_2(s), \Intw e_1^{\vee}(-s)) \right) \cdot \Intw e_1(s)(x) \cdot \Intw e_2^{\vee}(-s) (y) \\
	&= \sum_{\tilde{e}_1,\tilde{e}_2 \in \Bas(\omega \chi^{-1}, \chi)} \Zeta \left( \frac{s_0+1}{2}, \Psi, \beta(\tilde{e}_2(-s), \tilde{e}_1^{\vee}(s)) \right) \cdot \tilde{e}_1(-s)(x) \cdot \tilde{e}_2^{\vee}(s) (y).
\end{align*}
	Forming Eisenstein series in both variables yields the desired equality.
\end{proof}

\begin{lemma}
	Let ($\gp{K}$-finite) $f \in \pi(\chi_1,\chi_2)$ with $\chi_j \in \widehat{\R_+\F^{\times} \backslash \A^{\times}}$ and let $\chi \in \widehat{\R_+\F^{\times} \backslash \A^{\times}}$. Fix $s_0 \in \C$. The function $ s \mapsto \Zeta\left( s_0, \eis(s,f), \chi \right) $ has possible poles at
\begin{itemize}
	\item[(1)] $(\rho - 1)/2$ as $\rho$ runs over the zeros of the $\GL_1$ $L$-function $L^{(S)}(s, \chi_1 \chi_2^{-1})$ (partial $L$-function defined in (\ref{PartialL})), where $S$ is a finite set of finite places such that $\vp < \infty$ and $\vp \notin S$ implies $f$ is $\gp{K}_{\vp}$-invariant;
	\item[(2)] simple poles at $s=-s_0$ and $s=1-s_0$ if $\chi = \chi_1^{-1}$ with
	$$ \Res_{s=-s_0} \Zeta\left( s_0, \eis(s,f), \chi_1^{-1} \right) = - \zeta_{\F}^* f(\mathbbm{1}), \qquad \Res_{s=1-s_0} \Zeta\left( s_0, \eis(s,f), \chi_1^{-1} \right) = \zeta_{\F}^* \Intw f_{1-s_0}(w), $$
	where $\zeta_{\F}^*$ is the residue at $s=1$ of the Dedekind zeta function $\zeta_{\F}(s)$;
	\item[(3)] simple poles at $s=s_0$ and $s=s_0-1$ if $\chi = \chi_2^{-1}$ with
	$$ \Res_{s=s_0} \Zeta\left( s_0, \eis(s,f), \chi_2^{-1} \right) = \zeta_{\F}^* \Intw f_{s_0}(\mathbbm{1}), \qquad \Res_{s=s_0-1} \Zeta\left( s_0, \eis(s,f), \chi_2^{-1} \right) = -\zeta_{\F}^* f_{s_0-1}(w). $$
\end{itemize}
\label{EisZetaPoles}
\end{lemma}
\begin{proof}
	First consider the Godement sections
	$$ f_{\Phi}(s, g) = f_{\Phi}(s, \chi_1,\chi_2; g) := \chi_1(\det g) \norm[\det g]_{\A}^{1/2+s} \int_{\A^{\times}} \Phi((0,t)g) \chi_1\chi_2^{-1}(t) \norm[t]_{\A}^{1+2s} d^{\times}t, \quad \Phi \in \Sch(\A^2) $$
	with the related Eisenstein series $\eis(s,\Phi)$. It is easy to identify the relevant zeta integral
	$$ \Zeta\left( s_0, \eis(s,\Phi), \chi \right) = \int_{(\A^{\times})^2} \OFour_2 \Phi(t_1,t_2) \chi\chi_1(t_1) \norm[t_1]_{\A}^{s_0+s} \chi\chi_2(t_2) \norm[t_2]_{\A}^{s_0-s} d^{\times}t_1 d^{\times}t_2 $$
	as a two dimensional Tate's zeta integral. Thus we detect the type (2) and (3) poles. We have by (\ref{TateIntRes})
\begin{align*}
	\Res_{s=-s_0} \Zeta\left( s_0, \eis(s,\Phi), \chi_1^{-1} \right) &= -\zeta_{\F}^* \int_{\A^{\times}} \OFour_2 \Phi(0,t_2) \chi_1^{-1}\chi_2(t_2) \norm[t_2]_{\A}^{2s_0} d^{\times}t_2 \\
	&= -\zeta_{\F}^* \int_{\A^{\times}} \Phi(0,t_2) \chi_1\chi_2^{-1}(t_2) \norm[t_2]_{\A}^{1-2s_0} d^{\times}t_2 = - \zeta_{\F}^* f_{\Phi}(-s_0, \mathbbm{1})
\end{align*}
	which gives the first formula of residues. We leave the formulas for other residues as an exercise, providing the following formula for the intertwining operator
	$$ \Intw f_{\Phi}(s,\chi_1,\chi_2; g) =  f_{\widehat{\Phi}}(-s,\chi_2,\chi_1; g), \quad \widehat{\Phi}(x_1,x_2) := \OFour \Phi(-x_2,x_1). $$
	To see the type (1) poles, we claim that for any $\gp{K}$-isotypic $f$, we can find $\Phi \in \Sch(\A^2)$ such that
\begin{equation} 
	f_{\Phi}(s,g) = \left( \prod_{v \mid \infty} P_v(s) L_v(1+2s,\chi_{1,v}\chi_{2,v}^{-1}) \right) \cdot L^{(S)}(1+2s,\chi_1\chi_2^{-1}) \cdot f_s(g), 
\label{SecRel}
\end{equation}
	where
\begin{itemize}
	\item[(a)] $P_v(s)$ is a polynomial in $s$ with simple zeros, whose set of zeros is included in the set of poles of $L_v(1+2s,\chi_{1,v}\chi_{2,v}^{-1})$;
	\item[(b)] $S$ is a large finite set of finite places $\vp < \infty$ at which $f_{\vp}$ is not $\gp{K}_{\vp}$-invariant and
\begin{equation} 
	L^{(S)}(s,\chi_1\chi_2^{-1}) = \prod_{\vp < \infty, \vp \notin S} L_{\vp}(s,\chi_{1,\vp} \chi_{2,\vp}^{-1}). 
\label{PartialL}
\end{equation}
\end{itemize}
	In fact, such $\Phi = \otimes_v' \Phi_v$ can be chosen decomposable. The existence of $P_v(s)$ and $\Phi_v$ for $v \mid \infty$ such that
	$$ f_{\Phi_v}(s,g) = P_v(s) L_v(1+2s,\chi_{1,v}\chi_{2,v}^{-1}) f_{v,s}(g) $$
	 is proved in \cite[Lemma 3.5 (1) \& Lemma 3.8 (1)]{Wu5}. At $\vp < \infty$ and $\vp \notin S$, we choose $\Phi_{\vp} = \mathbbm{1}_{\vo_{\vp} \times \vo_{\vp}}$ so that
	 $$ f_{\Phi_{\vp}}(s,g) = L_{\vp}(1+2s,\chi_{1,\vp} \chi_{2,\vp}^{-1}) f_{\vp,s}(g). $$
	 At $\vp \in S$, by the isomorphism map $\iota$ defined in \cite[\S 3.4.2]{Wu5} we deduce the existence of $\Phi_{\vp}$ with support in 
	 $$ F_0 := \left\{ (x,y) \in \F_{\vp}^2 \ \middle| \ \max(\norm[x]_{\vp}, \norm[y]_{\vp}) = 1 \right\}, $$
	 such that $f_{\Phi_{\vp}}(0,\kappa) = f_{\vp,0}(\kappa)$ for all $\kappa \in \gp{K}_{\vp}$. Note that $F_0$ is stable by $\gp{K}_{\vp}$, therefore the integral
	 $$ f_{\Phi_{\vp}}(s,\kappa) = \chi_{1,\vp}(\det \kappa) \int_{\F_{\vp}^{\times}} \Phi_{\vp}((0,t)\kappa) \chi_{1,\vp}\chi_{2,\vp}^{-1}(t) \norm[t]_{\vp}^{1+2s} d^{\times}t $$
	 is non-vanishing only for $t \in \vo_{\vp}^{\times}$ thus independent of $s$, hence $f_{\Phi_{\vp}}(s,g)$ is already a flat section, implying
	 $$ f_{\Phi_{\vp}}(s,g) = f_{\vp,s}(g), $$
	 and conclude the proof of the claim (\ref{SecRel}). The type (1) poles are precisely the zeros of $L^{(S)}(1+2s,\chi_1\chi_2^{-1})$ for all possible $S$.
\end{proof}

\begin{corollary}
	The possible poles of $s \mapsto M_3(\lambda, \Psi \mid \chi, s)$ are classified as
\begin{itemize}
	\item[(0)] $(\rho - 1)/2$ resp. $(1-\rho)/2$ as $\rho$ runs over the zeros of $L^{(S)}(s, \omega^{-1}\chi^2)$ resp. $L^{(S)}(s, \omega\chi^{-2})$ (see (\ref{PartialL})), where $S$ is a finite set of finite places such that $\vp < \infty$ and $\vp \notin S$ implies $\Psi$ is bi-$\gp{K}_{\vp}$-invariant;
	\item[(1)] $(1-s_0)/2, -(1+s_0)/2$ if $\chi = \mathbbm{1}$; $(s_0-1)/2, (s_0+1)/2$ if $\chi = \omega$.
	\item[(2)] $-(s_1+(s_0+1)/2), -(s_1+(s_0-1)/2)$ if $\chi = \chi_1^{-1}$; $s_1+(s_0-1)/2, s_1+(s_0+1)/2$ if $\chi = \omega\chi_1$.
	\item[(3)] $-(s_2-(s_0-1)/2), -(s_2-(s_0+1)/2)$ if $\chi = \omega\chi_2^{-1}$; $s_2-(s_0+1)/2, s_2-(s_0-1)/2$ if $\chi = \chi_2$.
\end{itemize}
	The residues are meromorphic functions in $(s_1,s_2, s_0) \in \C^3$.
\label{ClassPoles}
\end{corollary}
\begin{proof}
	The types (0), (2) and (3) poles come from the $\GL_2 \times \GL_1$ zeta integrals. The type (1) poles come from the Godement-Jacquet integral, which is essentially a two dimensional Tate's integral due to (\ref{GJZetaPS}). Note that the formula for the residues of the simple poles of Tate's integral is reviewed in (\ref{TateIntRes}).
\end{proof}

	We note that every summand/integrand on the right hand side of (\ref{SpecSide}) is well-defined for $\lambda \in D_0$, and the sums/integrals are absolutely convergent. But the resulting function, denoted by
\begin{equation} 
	II_0(\lambda, \Psi) := \sum_{\substack{\pi \text{ cuspidal} \\ \omega_{\pi} = \omega }} M_3(\lambda, \Psi \mid \pi) + \sum_{\chi \in \widehat{\R_+ \F^{\times} \backslash \A^{\times}}} \int_{-\infty}^{\infty} M_3(\lambda, \Psi \mid \chi, i\tau) \frac{d\tau}{4\pi}, \quad \lambda \in D_0
\label{II_0AC}
\end{equation}
is not the meromorphic continuation of $\Theta(\lambda, \Psi) - DS_0(\lambda, \Psi)$ from $\lambda \in D$, because as $\lambda$ goes continuously from $D$ to $D_0$, it hits several ``walls''/hyperplanes of singularities, such as $\Re s_0 = 1$ (see Corollary \ref{ClassPoles}).

\begin{proposition}
	The meromorphic continuation of $\Theta(\lambda, \Psi) - DS_0(\lambda,\Psi)$ from $\lambda \in D$ to $D_0-H$ is given by
\begin{align*}
	\Theta(\lambda, \Psi) - DS_0(\lambda,\Psi) &= II_0(\lambda, \Psi) + \Res_{s=\frac{1-s_0}{2}} M_3(\lambda, \Psi \mid \mathbbm{1}, s) + \Res_{s=-\left( s_1 + \frac{s_0-1}{2} \right)} M_3(\lambda, \Psi \mid\chi_1^{-1}, s) \\
	&\quad + \Res_{s=-\left( s_2 - \frac{s_0+1}{2} \right)} M_3(\lambda, \Psi \mid \omega\chi_2^{-1}, s),
\end{align*}
	where $H=H(\omega,\chi_1,\chi_2)$ is the union of hyperplanes in which two of the above three residues may merge, which has measure $0$.
\label{SpecSideAC}
\end{proposition}
\begin{proof}
	Write $\Omega := \{ \mathbbm{1}, \omega, \chi_1^{-1}, \omega\chi_1, \chi_2, \omega\chi_2^{-1} \}$ and define for $\lambda \in D$
\begin{align*}
	II^{\Omega}(\lambda, \Psi) &:= \sum_{\substack{\pi \text{ cuspidal} \\ \omega_{\pi} = \omega }} M_3(\lambda, \Psi \mid \pi) + \sideset{}{_{\substack{ \chi \in \widehat{\R_+ \F^{\times} \backslash \A^{\times}} \\ \chi \notin \Omega  }}} \sum \int_{-\infty}^{\infty} M_3(\lambda, \Psi \mid \chi, i\tau) \frac{d\tau}{4\pi}.
\end{align*}
	Then every summand/integrand on the right hand side is entire in $\lambda = (s_0,s_1,s_2) \in \C^3$. Hence it defines an analytic function in $D_0$, which is its meromorphic continuation from $D$. Since
	$$ \Theta(\lambda, \Psi) - DS_0(\lambda, \Psi) = II^{\Omega}(\lambda, \Psi) + \sideset{}{_{\chi \in \Omega}} \sum \int_{\Re s = 0} M_3(\lambda, \Psi \mid \chi, s) \frac{ds}{4\pi i} $$
	for $(s_1,s_2,s_0) \in D$, we only need to give the analytic continuation for each $\chi \in \Omega$ of
	$$ \int_{\Re s = 0} M_3(\lambda, \Psi \mid \chi, s) \frac{ds}{4\pi i}. $$
	Let $B \gg 1$ and define a box region
	$$ X_B := \left\{ (s_1,s_2,s_0) \in \C^3 \ \middle| \ \norm[\Re s_j], \norm[\Im s_j] < B \right\}. $$ 
	Let $Q > 3B+1$ and take $L_Q$ to be the polygonal line joining $-i\infty, -iQ, Q-iQ, Q+iQ, iQ$ and $i\infty$. Let $S$ be the set of finite places $\vp < \infty$ at which $\Psi$ is not bi-$\gp{K}_{\vp}$-invariant. We choose $Q$ carefully so that $L_Q$ avoids any zero of $L^{(S)}(1-2s, \omega^{-1}\chi^2)$ for $\chi \in \Omega$. Then for $(s_1,s_2,s_0) \in D \cap X_B$
\begin{align}
	\int_{\Re s = 0} &M_3(\lambda, \Psi \mid \chi, s) \frac{ds}{4\pi i} = \int_{L_Q} M_3(\lambda, \Psi \mid \chi, s) \frac{ds}{4\pi i} \nonumber \\
	&- \frac{1}{2} \sum_{\rho: L^{(S)}(1-2s, \omega\chi^{-2}) = 0, (0<) \Re \rho, \norm[\Im \rho] < Q} \Res_{s=\rho} M_3(\lambda, \Psi \mid \chi, s) \nonumber \\
	&- \frac{1}{2} \sideset{}{_{\rho \in Z_{\chi}}} \sum \Res_{s=\rho} M_3(\lambda, \Psi \mid \chi, s), \label{CShift}
\end{align}
	where $Z_{\chi}$ is the set of non type (0) poles (see Corollary \ref{ClassPoles}) encountered in the contour shift. Precisely, we have (if multiple conditions are satisfied, take the union of the corresponding sets on the right hand side)
	$$ Z_{\chi} = \left\{ \begin{matrix} \{ \rho_{\chi}, (s_0+1)/2 \} & \text{if } \chi = \omega \\ \{ \rho_{\chi}, s_1 + (s_0+1)/2 \} & \text{if } \chi = \omega \chi_1 \\ \{ \rho_{\chi}, s_2 - (s_0-1)/2 \} & \text{if } \chi = \chi_2 \\ \emptyset & \text{otherwise} \end{matrix} \right. \quad \text{with} \quad \rho_{\chi} := \left\{ \begin{matrix} (s_0-1)/2 & \text{if } \chi = \omega \\ s_1 + (s_0-1)/2 & \text{if } \chi = \omega \chi_1 \\ s_2 - (s_0+1)/2 & \text{if } \chi = \chi_2 \end{matrix} \right.. $$
	The meromorphic continuation from $D \cap X_B$ to $D_0$ of the right hand side of (\ref{CShift}) has the same expression. In particular, the residues appearing on the right hand side of (\ref{CShift}) are naturally meromorphic functions in $(s_1,s_2,s_0) \in \C^3$. For example, we have for $L^{(S)}(1-2\rho, \omega\chi^{-2})$ and by (\ref{GJZetaPS})
\begin{align*} 
	&\quad \Res_{s=\rho} M_3(\lambda, \Psi \mid \chi, s) \\
	&= \sum_{e_1,e_2 \in \Bas(\chi,\omega\chi^{-1})} \int_{(\A^{\times})^2} \Four[2]{ {}_{e^{\vee}_1} \! \Psi_{e_2} } \begin{pmatrix} t_1 & 0 \\ 0 & t_2 \end{pmatrix} \chi(t_1) \norm[t_1]_{\A}^{\frac{s_0+1}{2}+\rho} \omega\chi^{-1}(t_2) \norm[t_2]_{\A}^{\frac{s_0+1}{2}-\rho} d^{\times}t_1 d^{\times}t_2 \cdot \\
	&\quad \Zeta \left( s_1+\frac{s_0+1}{2}, \eis(\rho,e_1), \chi_1 \right) \cdot \Res_{s=\rho} \Zeta \left( s_2-\frac{s_0-1}{2}, \eis(-s,e_2^{\vee}), \chi_2 \right).
\end{align*}
	Fix $(s_0,s_1,s_2) \in D_0$, we shift the contour back. The right hand side of (\ref{CShift}) becomes
\begin{align*}
	\int_{\Re s = 0} M_3(\lambda, \Psi \mid \chi, s) &\frac{ds}{4\pi i} - \frac{1}{2} \Res_{s=\rho_{\chi}} M_3(\lambda, \Psi \mid \chi, s) + \frac{1}{2} \Res_{s=\rho_{\chi}'} M_3(\lambda, \Psi \mid \chi, s),
\end{align*}
	where $\rho_{\chi}'$ is the pole we meet during the backwards contour shift whose residue does not cancel out with those residues already appeared. Precisely
	$$ \rho_{\chi}' = -\rho_{\omega\chi^{-1}} = \left\{ \begin{matrix} (1-s_0)/2 & \text{if } \chi = \mathbbm{1} \\ -(s_1+(s_0-1)/2) & \text{if } \chi = \chi_1^{-1} \\ -(s_2-(s_0+1)/2) & \text{if } \chi = \omega\chi_2^{-1} \\ \text{does not exist} & \text{otherwise} \end{matrix} \right. . $$
	But the functional equation in Lemma \ref{CMDFE} implies
	$$ \Res_{s=\rho_{\chi}} M_3(\lambda, \Psi \mid \chi, s) = -\Res_{s=\rho_{\omega\chi^{-1}}'} M_3(\lambda, \Psi \mid \omega\chi^{-1}, s). $$
	Hence we only need to sum over $\rho_{\chi}'$, giving the stated expression. We conclude since $B$ is arbitrary.
\end{proof}

	It remains to show the meromorphic continuation of $DS_0(\lambda,\Psi)$ as defined by (\ref{DS0Bis}), which will be given at the end of the next subsection \S \ref{GeomAspRes}.

	\subsection{Geometric Aspect of Residues}
	\label{GeomAspRes}
	
	The residues of $M_3(\lambda, \Psi \mid \chi,s) = M_3(\lambda, \Psi \mid \chi, s)$ appearing in Proposition \ref{SpecSideAC} are at $s=1/2$ for $\lambda = \vec{0}$. There are three more such residues which do not show up in Proposition \ref{SpecSideAC}:
\begin{align} 
	&\frac{1}{\zeta_{\F}^*} \Res_{s=\frac{s_0+1}{2}} M_3(\lambda, \Psi \mid \omega, s), \label{ResMDEx1} \\
	&\frac{1}{\zeta_{\F}^*} \Res_{s=s_1+\frac{1+s_0}{2}} M_3(\lambda, \Psi \mid \omega\chi_1, s), \label{ResMDEx2} \\
	&\frac{1}{\zeta_{\F}^*} \Res_{s=s_2-\frac{s_0-1}{2}} M_3(\lambda, \Psi \mid \chi_2, s). \label{ResMDEx3}
\end{align}
	We are going to give alternative expressions of them in some subregion of $D'$ of parameters. It will be convenient to use the new parameters in computation
	$$ \tilde{s}_0 := \frac{s_0}{2}, \quad \tilde{s}_1 := s_1 + \frac{s_0}{2}, \quad \tilde{s}_2 := s_2 - \frac{s_0}{2}. $$
	Hence the domains $D$ and $D'$ become
	$$ D = \left\{ (\tilde{s}_0, \tilde{s}_1, \tilde{s}_2) \in \C^3 \ \middle| \ \Re \tilde{s}_0 > 1, \Re \tilde{s}_1 - 2 \Re \tilde{s_0} - 2 > \frac{1}{2}, \Re \tilde{s}_2 > \frac{1}{2} \right\}, $$
	$$ D' = \left\{ (\tilde{s}_0, \tilde{s}_1, \tilde{s}_2) \in \C^3 \ \middle| \ \Re \tilde{s}_0 > 1, \Re \tilde{s}_1 - \Re \tilde{s}_0 > 1, \Re \tilde{s}_2 - \Re \tilde{s}_1 > 1 \right\}. $$
	
	We first record some elementary results concerning the zeta integrals.
\begin{lemma}
	Let $\chi_1,\chi_2$ and $\chi$ be quasi-characters of $\F^{\times} \backslash \A^{\times}$. Let $f \in \pi(\chi_1,\chi_2)$ be a smooth vector. 
\begin{itemize}
	\item[(1)] If $\Re s_0 > (1+ \Re (\chi_1^{-1}\chi_2))/2$, then for $1-\Re s_0 - \Re(\chi\chi_1) < \Re s < \Re s_0 - \Re(\chi\chi_2)$ we have
	$$ \Zeta(s, \eis(s_0,f), \chi) = \int_{\A^{\times}} f_{s_0}(wn(u)) \chi^{-1}\chi_2^{-1}(u) \norm[u]_{\A}^{1+s_0-s} d^{\times}u, $$
	where the integral is absolutely convergent. Here, $\Re (\chi) \in \R$ is defined by $\norm[\chi(t)] = \norm[t]_{\A}^{\Re(\chi)}$.
	\item[(2)] Let $W_{s_0}$ be the Whittaker function of $f_{s_0} \in \pi(\chi_1\norm_{\A}^{s_0}, \chi_2\norm_{\A}^{-s_0})$. If $\Re s + \Re(\chi) - 1 > \max(\Re(\chi_1)+\Re s_0, \Re(\chi_2)-\Re s_0)$, then we have
	$$ \Zeta(s, \eis(s_0,f), \chi) = \int_{\A^{\times}} W_{s_0}(a(y)) \chi(y) \norm[y]_{\A}^{s-\frac{1}{2}} d^{\times}y, $$
	where the integral is absolutely convergent.
\end{itemize}
\label{GlobZetaCalAlt}
\end{lemma}
\begin{proof}
	We only prove (1). Both sides being linear in $f_{s_0} \in \pi(\chi_1 \norm_{\A}^{s_0},\chi_2 \norm_{\A}^{-s_0})$, we can replace the flat section $f_{s_0}$ by a Godement section constructed from some $\Phi \in \Sch(\A^2)$
	$$ f_{\Phi}(s_0,g) := \chi_1(\det g) \norm[\det g]_{\A}^{1/2+s_0} \int_{\A^{\times}} \Phi((0,t)g) \chi_1\chi_2^{-1}(t) \norm[t]_{\A}^{1+2s_0} d^{\times}t. $$
	Routine computation gives the Whittaker function of $f_{\Phi}(s_0,g)$
	$$ W_{\Phi}(s_0, a(y)) = \chi_1(y) \norm[y]_{\A}^{1/2+s_0} \int_{\A^{\times}} \OFour_2 \Phi(ty,1/t) \chi_1\chi_2^{-1}(t) \norm[t]_{\A}^{2s_0} d^{\times}t. $$
	Hence for $\Re s \gg 1$, we get
\begin{align*} 
	\Zeta(s, \eis(f_{\Phi}(s_0,\cdot)), \chi) &= \int_{\A^{\times}} W_{\Phi}(s_0, a(y)) \chi(y) \norm[y]_{\A}^{s-1/2} d^{\times}y \\
	&= \int_{\A^{\times}} \! \int_{\A^{\times}} \OFour_2 \Phi(t_1,t_2) \chi\chi_1(t_1) \norm[t_1]_{\A}^{s+s_0} \chi\chi_2(t_2) \norm[t_2]_{\A}^{s-s_0} d^{\times}t_1 d^{\times}t_2.
\end{align*}
	Applying the global functional equation to $t_2$, we get for $1-\Re s_0 - \Re(\chi\chi_1) < \Re s < \Re s_0 - \Re(\chi\chi_2)$
\begin{align*} 
	\Zeta(s, \eis(f_{\Phi}(s_0,\cdot)), \chi) &= \int_{\A^{\times}} \! \int_{\A^{\times}} \Phi(t_1,t_2) \chi\chi_1(t_1) \norm[t_1]_{\A}^{s+s_0} \chi^{-1}\chi_2^{-1}(t_2) \norm[t_2]_{\A}^{1+s_0-s} d^{\times}t_1 d^{\times}t_2 \\
	&= \int_{\A^{\times}} f_{\Phi}(s_0, wn(u)) \chi^{-1}\chi_2^{-1}(u) \norm[u]_{\A}^{1+s_0-s} d^{\times}u,
\end{align*}
	since $f_{\Phi}(s_0, wn(u)) = \int_{\A^{\times}} \Phi(t,tu) \chi_1\chi_2^{-1}(t) \norm[t]_{\A}^{1+2s_0} d^{\times}t$.
\end{proof}

\begin{lemma}
	Let $f \in \pi(\chi_1, \chi_2)$. Then we have $W_f = W_{\Intw f}$, where $\Intw: \pi(\chi_1,\chi_2) \to \pi(\chi_2,\chi_1)$ is the intertwining operator.
\label{WhiInv}
\end{lemma}
\begin{proof}
	We have the functional equation of Eisenstein series $\eis(f) = \eis(\Intw f)$. The desired equality follows since by definition we have
	$$ W_f(g) = \int_{\F \backslash \A} \eis(f)(n(u)g) \psi(-u) du. $$
\end{proof}
	
	We then notice that for quasi-characters $\chi, \chi_1,\chi_2$ of $\F^{\times} \backslash \A^{\times}$ the following function on $\GL_2(\A) \times \GL_2(\A)$
	$$ f_{\Psi}(g_1,g_2; \chi_1,\chi_2) := \int_{(\A^{\times})^2} \OFour_2 \rpL(g_1) \rpR(g_2) \Psi \begin{pmatrix} t_1 & \\ & t_2 \end{pmatrix} \chi_1(t_1) \chi_2(t_2) \norm[t_1t_2]_{\A}^{-\frac{1}{2}} d^{\times}t_1 d^{\times}t_2, $$
	$$ \text{resp.} \quad f_{\Psi}(g_1,g_2; \chi,*) := \int_{\A^{\times}} \OFour_2 \rpL(g_1) \rpR(g_2) \Psi \begin{pmatrix} t & 0 \\ 0 & 0 \end{pmatrix} \chi(t) \norm[t]_{\A}^{-\frac{1}{2}} d^{\times}t $$
	defines a vector in the tensor product $\pi(\chi_1,\chi_2) \otimes \pi(\chi_1^{-1},\chi_2^{-1})$ resp. $\pi(\chi, \norm_{\A}^{\frac{1}{2}}) \otimes \pi(\chi^{-1}, \norm_{\A}^{-\frac{1}{2}})$ realized in the induced model, where the integrals in $t_1,t_2$ are Tate's integral.
	
\begin{definition}
	We also write $(f \times f)_{\Psi}(\cdot)$ for $f_{\Psi}(\cdot)$. We denote by $M$ resp. $W$ the image under the intertwining operator resp. Whittaker functional of $f$ with respect to one variable. For example, $(f \times W)(g_1,g_2;\chi_1,\chi_2) \in \pi(\chi_1,\chi_2) \otimes \Whi(\chi_1^{-1},\chi_2^{-1};\psi)$ is so defined that for any fixed $g_1$, the function $g_2 \mapsto (f \times W)(g_1,g_2;\chi_1,\chi_2)$ is the Whittaker function with respect to the additive character $\psi$ of $\F \backslash \A$ of the vector in the induced model given by $g_2 \mapsto (f \times f)_{\Psi}(g_1,g_2;\chi_1,\chi_2)$.
\end{definition}

\begin{lemma}
	We have the functional equation
	$$ (M \times f)_{\Psi}(g_1,g_2;\chi_1,\chi_2) = (f \times M)_{\Psi}(g_1,g_2;\chi_2,\chi_1). $$
\label{M2SecFE}
\end{lemma}
\begin{proof}
	It suffices to prove the equation for $g_1=g_2=\id$, since the general case follows from this special case by taking $\rpL(g_1)\rpR(g_2)\Psi$ instead of $\Psi$. By definition, we have in the absolutely convergent region
\begin{align*}
	&\quad (M \times f)_{\Psi}(\id,\id;\chi_1,\chi_2) \\
	&= \int_{\A} \int_{(\A^{\times})^2} \int_{\A} \Psi \left( n(-u) w^{-1} \begin{pmatrix} t_1 & x \\ 0 & t_2 \end{pmatrix} \right) dx \chi_1(t_1) \chi_2(t_2) \norm[t_1t_2]^{-\frac{1}{2}} d^{\times} t_1 d^{\times} t_2 du \\
	&= \int_{\A} \int_{(\A^{\times})^2} \int_{\A} \Psi \left( \begin{pmatrix} -t_2 & ut_1 \\ 0 & -t_1 \end{pmatrix} w n \left( \frac{x}{t_1} \right) \right) dx \chi_1(t_1) \chi_2(t_2) \norm[t_1t_2]^{-\frac{1}{2}} d^{\times} t_1 d^{\times} t_2 du \\
	&=  \int_{\A} \int_{(\A^{\times})^2} \int_{\A} \Psi \left( \begin{pmatrix} t_1 & u \\ 0 & t_2 \end{pmatrix} w n (x) \right) du \chi_2(t_1) \chi_1(t_2) \norm[t_1t_2]^{-\frac{1}{2}} d^{\times} t_1 d^{\times} t_2 dx \\
	&= (f \times M)_{\Psi}(g_1,g_2;\chi_2,\chi_1).
\end{align*}
	The equation is thus proved by the uniqueness of analytic continuation.
\end{proof}

\begin{lemma}
\begin{itemize}
	\item[(1)] If $\norm[\chi(t)] = \norm[t]_{\A}^{\sigma}$ for some $\sigma > 5/2$, then the Whittaker function $(W \times W)_{\Psi}(g_1,g_2; \chi, *)$ is given by the absolutely convergent integral
	$$ \int_{\A^{\times}} \int_{\A^2} \rpL(g_1) \rpR(g_2) \Psi \left( n(-u_1) \begin{pmatrix} 0 & 0 \\ z & 0 \end{pmatrix} n(u_2) \right) \psi(-u_1-u_2) du_1du_2 \chi(z) \norm[z]_{\A}^{\frac{1}{2}} d^{\times}z. $$
	\item[(2)] If $\norm[\chi_j(t)] = \norm[t]_{\A}^{\sigma_j}$ for some $\sigma_1 > 5/2,  \sigma_2 < 1/2$, then $(M \times W)_{\Psi}(g_1,g_2;\chi_1,\chi_2)$ is given by the absolutely convergent integral
	$$ \int_{(\A^{\times})^2} \int_{\A} \OFour_1 \OFour_2 \rpL(g_1) \rpR(g_2) \Psi \begin{pmatrix} ut_1 & -t_1 \\ t_2 & ut_2 \end{pmatrix} \psi(-u) \chi_2^{-1}(t_1) \norm[t_1]_{\A}^{\frac{3}{2}} \chi_1(t_2) \norm[t_2]_{\A}^{-\frac{1}{2}} du d^{\times}t_1 d^{\times}t_2. $$
\end{itemize}
\label{EssD}
\end{lemma}
\begin{proof}
	(1) Fix $g_1$. Then the function $g_2 \mapsto f_{\Psi}(g_1,g_2; \chi, *)$ is a meromorphic section in $\pi(\chi^{-1}, \norm_{\A}^{-\frac{1}{2}})$. Its Whittaker function in $\Whi(\chi^{-1}, \norm_{\A}^{-\frac{1}{2}}; \psi)$, i.e. $(f \times W)_{\Psi}(g_1,g_2; \chi, *)$ can be computed via the defining integral first assuming $\sigma < -1/2$, then by meromorphic continuation to $\sigma > 1/2$. For $\sigma < -1/2$, we have
\begin{align*}
	(f \times W)_{\Psi}(g_1,g_2; \chi, *) &= \int_{\A} f_{\Psi}(g_1,wn(u)g_2; \chi, *) \psi(-u) du \\
	&= \int_{\A} \int_{\A^{\times}} \OFour_1 \OFour_2 \rpL(g_1) \rpR(wn(u)g_2) \Psi \begin{pmatrix} t & 0 \\ 0 & 0 \end{pmatrix} \chi^{-1}(t) \norm[t]_{\A}^{\frac{3}{2}} \psi(-u) d^{\times}t du
\end{align*}
	by Tate's global functional equation. For any $\Psi \in \Sch(\Mat_2(\A))$, we have
\begin{align}
	\OFour_1 \OFour_2 \rpR(wn(u)) \Psi \begin{pmatrix} x_1 & x_2 \\ x_3 & x_4 \end{pmatrix} &= \int_{\A^2} \Psi \begin{pmatrix} -y_2 & y_1-uy_2 \\ -x_4 & x_3-ux_4 \end{pmatrix} \psi(-y_1x_1-y_2x_2) dy_1dy_2 \nonumber \\
	&= \OFour_1 \OFour_2 \Psi \begin{pmatrix} -x_2-ux_1 & x_1 \\ -x_4 & x_3-ux_4 \end{pmatrix}. \label{F1F2RF}
\end{align}
	Hence we obtain
\begin{align*} 
	(f \times W)_{\Psi}(g_1,g_2; \chi, *) &= \int_{\A} \int_{\A^{\times}} \OFour_1 \OFour_2 \rpL(g_1) \rpR(g_2) \Psi \begin{pmatrix} -ut & t \\ 0 & 0 \end{pmatrix} \chi^{-1}(t) \norm[t]_{\A}^{\frac{3}{2}} \psi(-u) d^{\times}t du \\
	&= \int_{\A^{\times}} \OFour_2 \rpL(g_1) \rpR(g_2) \Psi \begin{pmatrix} t^{-1} & t \\ 0 & 0 \end{pmatrix} \chi^{-1}(t) \norm[t]_{\A}^{\frac{1}{2}} d^{\times}t \\
	&= \int_{\A^{\times}} \OFour_2 \rpL(g_1) \rpR(g_2) \Psi \begin{pmatrix} t & t^{-1} \\ 0 & 0 \end{pmatrix} \chi(t) \norm[t]_{\A}^{-\frac{1}{2}} d^{\times}t.
\end{align*}
	Since the right hand side is absolutely convergent for all $\chi$, it gives the desired meromorphic continuation to $\sigma > 1/2$, where we have
\begin{multline*}
	(W \times W)_{\Psi}(g_1,g_2; \chi, *) = \int_{\A} (f \times W)_{\Psi}(wn(u_1)g_1,g_2; \chi, *) \psi(-u_1) du_1 \\
	= \int_{\A^{\times}} \int_{\A^2} \rpL(g_1) \rpR(g_2) \Psi \left( n(-u_1) w^{-1} \begin{pmatrix} t & u_2 \\ 0 & 0 \end{pmatrix} \right) \psi(-u_1-u_2t^{-1}) \chi(t) \norm[t]_{\A}^{-\frac{1}{2}} du_1du_2 d^{\times}t,
\end{multline*}
	which is precisely the desired formula by the change of variables $u_2 = u_2t$, $t=z$. The order change of integrations is justified by
	$$ \int_{\A^2} \max_{x \in \A} \extnorm{\rpL(g_1) \rpR(g_2) \Psi \begin{pmatrix} -u_1t & x \\ t & u_2 t  \end{pmatrix}} du_1 du_2 \ll_N \prod_v \min(1,\norm[t]_v^{-N}) \cdot \norm[t]_{\A}^{-2}, $$
	whose integral against $\norm[t]_{\A}^{\sigma+1/2}$ over $\A^{\times}$ is absolutely convergent if $\sigma + 1/2 - 2 > 1$, i.e. $\sigma > 5/2$.
	
\noindent (2) It suffices to prove the equation for $g_1=g_2=\id$. By Lemma \ref{WhiInv} and \ref{M2SecFE}, we have
	$$ (M \times W)_{\Psi}(\id,\id;\chi_1,\chi_2) = (f \times W)_{\Psi}(\id,\id;\chi_2,\chi_1). $$
	The right hand side can be computed by definition for $\sigma_1 - \sigma_2 > 1$ and (\ref{F1F2RF}) as
\begin{align*}
	&\quad (f \times W)_{\Psi}(\id,\id;\chi_2,\chi_1) = \int_{\A} (f \times f)_{\Psi}(\id,wn(u);\chi_2,\chi_1) \psi(-u) du \\
	&= \int_{\A} \int_{(\A^{\times})^2} \OFour_1 \OFour_2 \rpR(wn(u)) \Psi \begin{pmatrix} t_1 & \\ & t_2 \end{pmatrix} \chi_2^{-1}(t_1) \norm[t_1]_{\A}^{\frac{3}{2}} \chi_1(t_2) \norm[t_2]_{\A}^{-\frac{1}{2}} \psi(-u) d^{\times}t_1 d^{\times}t_2 du \\
	&= \int_{\A} \int_{(\A^{\times})^2} \OFour_1 \OFour_2 \Psi \begin{pmatrix} -ut_1 & t_1 \\ -t_2 & -ut_2 \end{pmatrix} \chi_2^{-1}(t_1) \norm[t_1]_{\A}^{\frac{3}{2}} \chi_1(t_2) \norm[t_2]_{\A}^{-\frac{1}{2}} \psi(-u) d^{\times}t_1 d^{\times}t_2 du.
\end{align*}
	The above integral is absolutely convergent by
	$$ \int_{\A} \max_{x \in \A} \extnorm{ \OFour_1 \OFour_2 \Psi \begin{pmatrix} x & t_1 \\ -t_2 & -ut_2 \end{pmatrix} } du \ll_N \prod_v \min(1,\norm[t_1]_v^{-N}) \min(1,\norm[t_2]_v^{-N}) \cdot \norm[t_2]_{\A}^{-1}, $$
	whose integral against $\norm[t_1]_{\A}^{3/2-\sigma_2} \norm[t_2]_{\A}^{\sigma_1-1/2}$ over $\A^{\times} \times \A^{\times}$ is absolutely convergent if $\sigma_1 > 5/2, \sigma_2 < 1/2$. The desired equation follows by order change of integrations and the change of variables $t_j \mapsto -t_j$.
\end{proof}

\begin{corollary}
	Assume $\lambda \in D'$. Recall the summation convention in Remark \ref{SumConv}.
\begin{itemize}
	\item[(1)] Let $\ell_0$ be the tempered distribution given by
	$$ \ell_0(\Psi) := \Vsum \int_{\A^2} \rpL \begin{pmatrix} \xi_1 & \\ & 1 \end{pmatrix} \rpR \begin{pmatrix} \xi_2 \xi & \\ & \xi \end{pmatrix} \Psi \left( n(-u_1) \begin{pmatrix} 0 & 0 \\ 1 & 0 \end{pmatrix} n(u_2) \right) \psi(-u_1-u_2) du_1du_2. $$
	Then (\ref{ResMDEx1}) is equal to the absolutely convergent integral
	$$ \int_{(\F^{\times} \backslash \A^{\times})^3} \ell_0 \left( \rpL \begin{pmatrix} t_1 & \\ & 1 \end{pmatrix} \rpR \begin{pmatrix} t_2 t & \\ & t \end{pmatrix} \Psi \right) \omega(t) \norm[t]_{\A}^{s_0+2} \chi_1(t_1)\norm[t_1]_{\A}^{s_1-1} \chi_2(t_2)\norm[t_2]_{\A}^{s_2+1} d^{\times}t d^{\times}t_1 d^{\times}t_2. $$
	\item[(2)] Let $\ell_1$ be the tempered distribution given by
	$$ \ell_1(\Psi) := \Vsum \int_{\A} \OFour_1 \OFour_2 \rpL \begin{pmatrix} \xi_1 & \\ & 1 \end{pmatrix} \rpR \begin{pmatrix} \xi_2 \xi & \\ & \xi \end{pmatrix} \Psi \begin{pmatrix} u & -1 \\ 1 & u \end{pmatrix} \psi(-u) du. $$
	Then (\ref{ResMDEx2}) is equal to the absolutely convergent integral
	$$ \int_{(\F^{\times} \backslash \A^{\times})^3} \ell_1 \left( \rpL \begin{pmatrix} t_1 & \\ & 1 \end{pmatrix} \rpR \begin{pmatrix} t_2 t & \\ & t \end{pmatrix} \Psi \right) \omega(t) \norm[t]_{\A}^{s_0+2} \chi_1(t_1)\norm[t_1]_{\A}^{s_1-1} \chi_2(t_2)\norm[t_2]_{\A}^{s_2+1} d^{\times}t d^{\times}t_1 d^{\times}t_2. $$
	\item[(3)] Let $\ell_2$ be the tempered distribution given by
	$$ \ell_2(\Psi) := \Vsum \OFour_2 \OFour_4 \rpL \begin{pmatrix} \xi_1 & \\ & 1 \end{pmatrix} \rpR \begin{pmatrix} \xi_2 \xi & \\ & \xi \end{pmatrix} \Psi \begin{pmatrix} -1 & -1 \\ 1 & -1 \end{pmatrix}. $$
	Then (\ref{ResMDEx3}) is equal to the absolutely convergent integral
	$$ \int_{(\F^{\times} \backslash \A^{\times})^3} \ell_2 \left( \rpL \begin{pmatrix} t_1 & \\ & 1 \end{pmatrix} \rpR \begin{pmatrix} t_2 t & \\ & t \end{pmatrix} \Psi \right) \omega(t) \norm[t]_{\A}^{s_0+2} \chi_1(t_1)\norm[t_1]_{\A}^{s_1-1} \chi_2(t_2)\norm[t_2]_{\A}^{s_2+1} d^{\times}t d^{\times}t_1 d^{\times}t_2. $$
\end{itemize}
\label{RResEss}
\end{corollary}
\begin{proof}
	(1) Applying (\ref{TateIntRes}), (\ref{GJZetaPS}) and Lemma \ref{GlobZetaCalAlt}, we identify (\ref{ResMDEx1}) with
\begin{multline*}
	\sum_{e_1,e_2 \in \Bas(\omega,\id)} \int_{\A^{\times}} \OFour_2 \left( {}_{e_1^{\vee}} \Psi_{e_2} \right) \begin{pmatrix} t & 0 \\ 0 & 0 \end{pmatrix} \omega(t)\norm[t]_{\A}^{1+2\tilde{s}_0} d^{\times}t \cdot \\
	\qquad \int_{\A^{\times}} W_{e_1} \left( \frac{1}{2}+\tilde{s}_0, a(y) \right) \chi_1(y) \norm[y]_{\A}^{\tilde{s}_1} d^{\times}y \cdot \int_{\A^{\times}} W_{e_2^{\vee}} \left( -\frac{1}{2}-\tilde{s}_0, a(y) \right) \chi_2(y) \norm[y]_{\A}^{\tilde{s}_2} d^{\times}y \\
	= \sum_{e_1,e_2 \in \Bas(\omega,\id)} \int_{\gp{K}^2} (f \times f)_{\Psi}(\kappa_1,\kappa_2; \omega \norm_{\A}^{\frac{3}{2}+2\tilde{s_0}}, *) e_1^{\vee}(\kappa_1) e_2(\kappa_2) d\kappa_1 d\kappa_2 \cdot \\
	\int_{\A^{\times}} W_{e_1 \otimes \norm_{\A}^{1+\tilde{s}_0}} \left( \frac{1}{2}+\tilde{s}_0, a(y) \right) \chi_1(y) \norm[y]_{\A}^{\tilde{s}_1-\tilde{s_0}-1} d^{\times}y \cdot \int_{\A^{\times}} W_{e_2^{\vee} \otimes \norm_{\A}^{-1-\tilde{s}_0}} \left( -\frac{1}{2}-\tilde{s}_0, a(y) \right) \chi_2(y) \norm[y]_{\A}^{\tilde{s}_2+\tilde{s}_0+1} d^{\times}y \\
	= \int_{(\A^{\times})^2} (W \times W)_{\Psi}(a(t_1),a(t_2);\omega \norm_{\A}^{\frac{3}{2}+2\tilde{s}_0},*) \chi_1(t_1) \norm[t_1]_{\A}^{\tilde{s}_1-\tilde{s}_0-1} \chi_2(t_2) \norm[t_2]_{\A}^{\tilde{s}_2+\tilde{s}_0+1} d^{\times}y_1 d^{\times}y_2.
\end{multline*}
	In the above we need the tricky twists by $\norm_{\A}^{\pm(1+\tilde{s}_0)}$ because $(f \times f)_{\Psi}(\cdot,\cdot; \omega \norm_{\A}^{3/2+2\tilde{s_0}}, *) \in \pi(\omega \norm_{\A}^{3/2+2\tilde{s}_0}, \norm_{\A}^{1/2}) \otimes \pi(\omega^{-1} \norm_{\A}^{-3/2-2\tilde{s}_0}, \norm_{\A}^{-1/2})$ has central characters $\omega \norm_{\A}^{2+2\tilde{s}_0}$ resp. $\omega^{-1} \norm_{\A}^{-2-2\tilde{s}_0}$ in the first resp. second variable. While the flat section $e_1(1/2+\tilde{s}_0) \in \pi(\omega \norm_{\A}^{1/2+\tilde{s}_0}, \norm_{\A}^{-1/2-\tilde{s}_0})$ resp. $e_2^{\vee}(-1/2-\tilde{s}_0) \in \pi(\omega^{-1} \norm_{\A}^{-1/2-\tilde{s}_0}, \norm_{\A}^{1/2+\tilde{s}_0})$, we must apply the indicated twists to make the relevant representations dual to each other. This subtlety re-occurs in the proof of (2) and (3) below, and we will not mention it again. Inserting Lemma \ref{EssD} (1), this is equal to
\begin{align*}
	&\int_{(\A^{\times})^2} \left( \int_{\A^{\times}} \int_{\A^2} \rpL(a(t_1)) \rpR(a(t_2)) \Psi \left( n(-u_1) \begin{pmatrix} 0 & 0 \\ t & 0 \end{pmatrix} n(u_2) \right) \psi(-u_1-u_2) du_1du_2 \omega(t) \norm[t]_{\A}^{2+2\tilde{s}_0} d^{\times}t \right) \cdot \\
	&\qquad \chi_1(t_1) \norm[t_1]_{\A}^{\tilde{s}_1-\tilde{s}_0-1} \chi_2(t_2) \norm[t_2]_{\A}^{\tilde{s}_2+\tilde{s}_0+1} d^{\times}t_1 d^{\times}t_2.
\end{align*}
	It is equal to the stated formula by the absolute convergence. 
	
\noindent (2) Similarly we identify (\ref{ResMDEx2}) with
\begin{align*}
	&\quad \sum_{e_1,e_2 \in \Bas(\omega\chi_1,\chi_1^{-1})} \int_{(\A^{\times})^2} \OFour_2 \left( {}_{e_1^{\vee}} \Psi_{e_2} \right) \begin{pmatrix} t_1 & 0 \\ 0 & t_2 \end{pmatrix} \omega\chi_1(t_1) \norm[t_1]_{\A}^{1+\tilde{s}_0+\tilde{s}_1} \chi_1^{-1}(t_2) \norm[t_2]_{\A}^{\tilde{s}_0-\tilde{s}_1} d^{\times}t_1 d^{\times}t_2 \cdot \\
	&\qquad \qquad \Intw e_1 \left( \frac{1}{2}+\tilde{s}_1, \id \right) \cdot \int_{\A^{\times}} W_{e_2^{\vee}} \left( -\frac{1}{2}-\tilde{s}_1, a(y) \right) \chi_2(y) \norm[y]_{\A}^{\tilde{s}_2} d^{\times}y \\
	&= \int_{\A^{\times}} (M \times W)_{\Psi}(a(t_1),a(t_2);\omega\chi_1 \norm_{\A}^{\frac{3}{2}+\tilde{s}_0+\tilde{s}_1}, \chi_1^{-1} \norm_{\A}^{\frac{1}{2}+\tilde{s}_0-\tilde{s}_1}) \chi_2(t_2) \norm[t_2]_{\A}^{\tilde{s}_2+\tilde{s}_0+1} d^{\times}t_2.
\end{align*}
	We get the desired formula by inserting Lemma \ref{EssD} (2).
	
\noindent (3) Similarly we identify (\ref{ResMDEx3}) with
\begin{multline*}
	\sum_{e_1,e_2 \in \Bas(\chi_2,\omega\chi_2^{-1})} \int_{(\A^{\times})^2} \OFour_2 \left( {}_{e_1^{\vee}} \Psi_{e_2} \right) \begin{pmatrix} t_1 & 0 \\ 0 & t_2 \end{pmatrix} \chi_2(t_1) \norm[t_1]_{\A}^{1+\tilde{s}_0+\tilde{s}_2} \omega\chi_2^{-1}(t_2) \norm[t_2]_{\A}^{\tilde{s}_0-\tilde{s}_2} d^{\times}t_1 d^{\times}t_2 \cdot \\
	\int_{\A^{\times}} e_1 \left( \frac{1}{2}+\tilde{s}_2, wn(t) \right) \omega^{-1}\chi_1^{-1}\chi_2(t) \norm[t]_{\A}^{\frac{1}{2}+\tilde{s}_2-\tilde{s}_0} d^{\times}t \cdot e_2^{\vee}(\id) \\
	= \int_{\A^{\times}} \int_{(\A^{\times})^2} \OFour_2 \OFour_4 \rpL(wn(t)) \Psi \begin{pmatrix} t_1 & 0 \\ 0 & t_2 \end{pmatrix} \chi_2(t_1) \norm[t_1]_{\A}^{1+\tilde{s}_0+\tilde{s}_2} \omega^{-1}\chi_2(t_2) \norm[t_2]_{\A}^{1+\tilde{s}_2-\tilde{s}_0} d^{\times}t_1 d^{\times}t_2 \cdot \\
	\omega^{-1}\chi_1^{-1}\chi_2(t) \norm[t]_{\A}^{1+\tilde{s}_2-\tilde{s}_1} d^{\times}t \\
	= \int_{(\A^{\times})^3} \OFour_2 \OFour_4 \Psi \begin{pmatrix} -tt_1 & -t_2 \\ t_1 & -tt_2 \end{pmatrix} \cdot \\
	\chi_2(t_1) \norm[t_1]_{\A}^{1+\tilde{s}_0+\tilde{s}_2} \omega^{-1}\chi_2(t_2) \norm[t_2]_{\A}^{1+\tilde{s}_2-\tilde{s}_0} \omega^{-1}\chi_1^{-1}\chi_2(t) \norm[t]_{\A}^{1+\tilde{s}_2-\tilde{s}_1} d^{\times}t_1 d^{\times}t_2 d^{\times}t,
\end{multline*}
	where we have applied the following elementary relation, whose proof is similar to (\ref{F1F2RF})
\begin{equation} 
	\OFour_2 \OFour_4 \rpL(wn(t)) \Psi \begin{pmatrix} x_1 & x_2 \\ x_3 & x_4 \end{pmatrix} = \OFour_2 \OFour_4 \Psi \begin{pmatrix} -tx_1-x_3 & -x_4 \\ x_1 & x_2-tx_4 \end{pmatrix}. 
\label{F2F4LF}
\end{equation}
	The last expression is easily identified with the desired formula.
\end{proof}

	To end this subsection, we give the meromorphic continuation of $DS_0(\lambda,\Psi)$ defined in (\ref{DS0Bis}). By definition, we have for $\lambda \in D$
	$$ DS_0(\lambda, \Psi) = \int_{(\A^{\times})^2} \Tree{\widetilde{WW}^{(1)}}{\omega,s_0}{a(t_1)}{a(t_2)} \chi_1(t_1) \norm[t_1]_{\A}^{s_1-1} \chi_2(t_2) \norm[t_2]_{\A}^{s_2+1} d^{\times}t_1 d^{\times}t_2 $$
	via the Whittaker-Fourier expansions in each variable, where
\begin{align*} 
	\Tree{\widetilde{WW}^{(1)}}{\omega,s_0}{x}{y} &:= \int_{(\F \backslash \A)^2} \Tree{\widetilde{K K}^{(1)}}{\omega,s_0}{n(u_1)x}{n(u_2)y} \psi(-u_1-u_2) du_1 du_2 \\
	&= \int_{\A^2} \int_{\A^{\times}} \rpL_x \rpR_y \Psi \left( n(-u_1) w^{-1} \begin{pmatrix} 0 & z \\ 0 & 0 \end{pmatrix} w n(u_2) \right) \omega(z) \norm[z]_{\A}^{s_0+2} d^{\times}z \psi(-u_1-u_2) du_1 du_2. 
\end{align*}
	The equality of matrices
	$$ w^{-1} \begin{pmatrix} 0 & -1 \\ 0 & 0 \end{pmatrix} w = \begin{pmatrix} 0 & 0 \\ 1 & 0 \end{pmatrix} $$
	identifies $DS_0$ via (\ref{ResMDEx1}) and Corollary \ref{RResEss} (1) with
	$$ DS_0(\lambda,\Psi) = \frac{1}{\zeta_{\F}^*} \Res_{s=\frac{1+s_0}{2}} M_3(\lambda, \Psi \mid \mathbbm{1},s), \quad \lambda \in D'. $$
	The right hand side gives the meromorphic continuation of $DS_0(\lambda,\Psi)$.

	\subsection{Fourth Moment Side}
	\label{4MAC}
	
	Based on (\ref{GeomSide}), we shall give another expression of the meromorphic continuation of $\Theta(\lambda, \Psi)$ from $\lambda \in D$ to $D_0$. Along the way, we complete the proof of Lemma \ref{GeomConv}.
	
	$DG_j, 5 \leq j \leq 8$ are given as higher dimensional global Tate integrals in (\ref{II_21CE}) to (\ref{II_24CE}), hence have obvious meromorphic continuation to $\C^3$. We are left with $II_i$ for $i \neq 2$.
	
	(\ref{II_1CE}) is still valid in $D_0$, giving the analytic continuation of $II_1(\lambda, \Psi)$ therein. Fix $\lambda = (s_0,s_1,s_2) \in D_0$. We shift the contour back to $\Re s = 1/2$, crossing four poles of the Tate integrals (because $\norm[\Re s_i'] < 1/2$ for $(s_0,s_1,s_2) \in D_0$) and get
\begin{equation}
	II_1(\lambda, \Psi) = M_4(\lambda, \Psi) + \sum_{j=1}^4 DG_j(\lambda, \Psi), \label{II_1AC}
\end{equation}
	where $M_4$ and $DG_j$ are given by (\ref{II_10AC}) and (\ref{II_11AC}) to (\ref{II_14AC}).
	
	We can easily identify the distributions $\ell_1(\Psi)$ resp. $\ell_2(\Psi)$ in Corollary \ref{RResEss} as
	$$ \ell_1(\Psi) = \Vsum \int_{\A} \OFour_1 \OFour_2 \Psi \begin{pmatrix} -\xi_2 u & \xi_2 \\ \xi_3 & \xi_3 u \end{pmatrix} \psi(-u \xi) du = - \Shoulder{\Delta \Delta_4}{\id}{\id}, $$
	$$ \ell_2(\Psi) = \Vsum_{\xi_1 \xi_2 + \xi_3 \xi_4 = 0} \OFour_2 \OFour_4 \Psi \begin{pmatrix} \xi_1 & \xi_2 \\ \xi_3 & \xi_4 \end{pmatrix} = - \Shoulder{\Delta \Delta_3}{\id}{\id}. $$
	Hence $II_3$ and $II_4$ are absolutely convergent in $D'$ by Corollary \ref{RResEss}, and we have
\begin{equation} 
	II_3(\lambda, \Psi) = - DS_4(\lambda, \Psi) = - \frac{1}{\zeta_{\F}^*} \Res_{s=s_2+\frac{1-s_0}{2}} M_3(\lambda, \Psi \mid \chi_2, s), 
\label{II_3AC}
\end{equation}
\begin{equation} 
	II_4(\lambda, \Psi) = - DS_5(\lambda, \Psi) = - \frac{1}{\zeta_{\F}^*} \Res_{s=s_1+\frac{1+s_0}{2}} M_3(\lambda, \Psi \mid \omega\chi_1, s), 
\label{II_4AC}
\end{equation}
	justifying their meromorphic continuation.
	
	We summarize the above discussion in the following proposition.
\begin{proposition}
	The meromorphic continuation of $\Theta(\lambda, \Psi)$ from $\lambda \in D'$ to $D_0-H'$ is given by
\begin{align*}
	\Theta(\lambda, \Psi) = M_4(\lambda, \Psi) + \sum_{j=1}^8 DG_j(\lambda, \Psi) - DS_4(\lambda, \Psi) - DS_5(\lambda, \Psi),
\end{align*}
	where $M_4$ and $DG_j$ are given by (\ref{II_10AC}) and (\ref{II_11AC}) to (\ref{II_24CE}), and $DS_4$ \& $DS_5$ are given by (\ref{II_3AC}) \& (\ref{II_4AC}); and $H'=H'(\omega,\chi_1,\chi_2)$ is the union of hyperplanes where two of the above residues in the degenerate terms may merge, which has measure $0$.
\label{GeomSideAC}
\end{proposition}

\section{Analysis of Degenerate Terms}

	We prove Proposition \ref{DegMainRel} in this section. Henceforth, we assume $\omega = \chi_1 = \chi_2 = \mathbbm{1}$. In this case, we can work out
	$$ H \cup H' = \left\{ \lambda = (s_0,s_1,s_2) \in \C^3 \ \middle| \ s_1(s_1+s_0-s_2)(s_2-s_0) \cdot s_2(s_2+s_1)(s_0+s_1) = 0 \right\}. $$ 
	
	Suppose $\delta > 0$ is small enough so that $\zeta_{\F}(s)\zeta_{\F}(1+s) \neq 0$ for $\norm[s] < 2\delta$. Assume $\Norm[\lambda] < \delta/4$ with $\lambda \notin H \cup H'$. The terms $DG_j(\lambda), 1 \leq j \leq 4$ are the residues of all possible poles of $s \mapsto M_4(\lambda, \Psi \mid \id,s)$ in the region $\norm[s-1] < \delta$. On the circle $\norm[s-1] = \delta$, $M_4(\lambda, \Psi \mid \id, s)$ is regular and continuous in $\lambda$. Thus
	$$ \sum_{j=1}^4 DG_j(\lambda, \Psi) = \frac{1}{\zeta_{\F}^*} \int_{\norm[s-1]=\delta} M_4(\lambda, \Psi \mid \id, s) \frac{ds}{2\pi i}, $$
	which holds as $\lambda \to \vec{0}$, proving the regularity with
	$$ \sum_{j=1}^4 DG_j(\vec{0}, \Psi) = \frac{1}{\zeta_{\F}^*} \int_{\norm[s-1]=\delta} M_4(\vec{0}, \Psi \mid \id, s) \frac{ds}{2\pi i} = \frac{1}{\zeta_{\F}^*} \Res_{s=1} M_4(\vec{0}, \Psi \mid \id, s). $$
	Similarly, we have
	$$ \sum_{j=5}^8 DG_j(\lambda, \Psi) = -\frac{1}{\zeta_{\F}^*} \int_{\norm[s]=\delta} M_4(\lambda, \Psi \mid \id, s) \frac{ds}{2\pi i}, $$
	which holds as $\lambda \to \vec{0}$, proving the regularity with
	$$ \sum_{j=5}^8 DG_j(\vec{0}, \Psi) = -\frac{1}{\zeta_{\F}^*} \int_{\norm[s]=\delta} M_4(\vec{0}, \Psi \mid \id, s) \frac{ds}{2\pi i} = -\frac{1}{\zeta_{\F}^*} \Res_{s=0} M_4(\vec{0}, \Psi \mid \id, s). $$
	The first formula in Proposition \ref{DegMainRel} readily follows. The second formula in Proposition \ref{DegMainRel} can be proved in the same way. We leave the details to the reader.
	
\begin{remark}
	The regularity at $\lambda=(0,0,0)$ for other $\omega,\chi_1,\chi_2$ can be proven in a similar way. We leave the case-by-case details to the reader.
\end{remark}

\section{Proof of Compact Variation}

	\subsection{Mixed Moment Side}
	
\begin{theorem}[Godement-Jacquet pre-trace formula: compact variation]
	The integral
	$$ \Tree{\widetilde{KK}_{\D}}{\omega,s_0}{x}{y} := \int_{\F^{\times} \backslash \A^{\times}} KK_{\D}(x,yz) \omega(z) \norm[z]_{\A}^{2s_0+2} d^{\times}z \cdot \norm[\nu_{\D}(x^{-1}y)]_{\A}^{s_0} $$
	is absolutely convergent for $\Re s_0 > 1$, and defines a smooth function in $\D^{\times}(\F) \backslash \D^{\times}(\A)$ with central character $\omega$ for $x$ resp. $\omega^{-1}$ for $y$. Its Fourier inversion with respect to $y$ in $\intL^2(\D^{\times}(\F) \backslash \D^{\times}(\A), \omega^{-1})$ converges normally for $(x,y) \in (\D^{\times}(\F) \backslash \D^{\times}(\A))^2$, and takes the form
\begin{align*}
	\Tree{\widetilde{KK}_{\D}}{\omega, s_0}{x}{y} &= \sideset{}{_{\substack{\pi \text{\ cuspidal} \\ \omega_{\pi} = \omega}}} \sum \sideset{}{_{\varphi_1, \varphi_2 \in \Bas(\pi)}} \sum \Zeta \left( s_0+\frac{1}{2}, \Psi, \beta(\varphi_2, \varphi_1^{\vee}) \right) \varphi_1(x) \varphi_2^{\vee}(y) \\
	&+ \frac{1}{\Vol([\D^{\times}])} \sideset{}{_{\substack{ \eta \in \widehat{\F^{\times} \backslash \A^{\times}} \\ \eta^2 = \omega }}} \sum \left( \int_{\D^{\times}(\A)} \Psi(g) \eta(\nu_{\D}(g)) \norm[\nu_{\D}(g)]_{\A}^{s_0+1} dg \right) \eta(\nu_{\D}(x)) \overline{\eta(\nu_{\D}(y))}.
\end{align*}
\label{GJPreTraceVar}
\end{theorem}
\begin{proof}
	There exist $i,j,k=ij \in \D^{\times}(\F)$ so that $\D = \F \oplus \F i \oplus \F j \oplus \F k$ as $\F$-vector spaces. Hence $\Sch(\D(\A)) \simeq \Sch(\A^4)$. We have by definition
\begin{multline*}
	\Tree{\widetilde{KK}_{\D}}{\omega, s_0}{x}{y} \cdot \norm[\nu_{\D}(x^{-1}y)]_{\A}^{-s_0} = \int_{\F^{\times} \backslash \A^{\times}} \left( \sum_{\xi \in \D^{\times}(\F)} \rpL_x \rpR_y \Psi(\xi z) \right) \norm[z]_{\A}^{2s_0+2} d^{\times}z \\
	= \int_{\F^{\times} \backslash \A^{\times}} \left( \sum_{\alpha_0, \alpha_1,\alpha_2,\alpha_3 \in \F^{\times}} \rpL_x \rpR_y \Psi(z(\alpha_0+\alpha_1 i + \alpha_2 j + \alpha_3 k)) \right) \norm[z]_{\A}^{2s_0+2} d^{\times}z + \cdots
\end{multline*}
	where we have written out the most problematic term (for convergence) and omitted the others. Proposition \ref{PoissonSEst} implies for any $N > 1$
	$$ \sum_{\alpha_0, \alpha_1,\alpha_2,\alpha_3 \in \F^{\times}} \rpL_x \rpR_y \Psi(z(\alpha_0+\alpha_1 i + \alpha_2 j + \alpha_3 k))  \ll_N \left( \min(\norm[z]_{\A}^{-1}, \norm[z]_{\A}^{-N}) \right)^4. $$
	Hence the above integral converges absolutely for $\Re s_0 > 1$. The proofs of the other assertions are straightforward analogues of those of Theorem \ref{GJPreTrace}. We leave the details to the reader.
\end{proof}

	Theorem \ref{GJPreTraceVar} implies readily the decomposition in the region $\Re s_0 > 1$ of
\begin{align*}
	&\quad \Theta_{\D}(s_0, \Psi) \\
	&= \sideset{}{_{\substack{\pi \text{\ cuspidal} \\ \omega_{\pi} = \omega}}} \sum \sideset{}{_{\varphi_1, \varphi_2 \in \Bas(\pi)}} \sum \Zeta \left( s_0+\frac{1}{2}, \Psi, \beta(\varphi_2, \varphi_1^{\vee}) \right) \int_{[\E^{\times}]}\varphi_1(t_1) \Omega(t_1)^{-1} d^{\times}t_1 \int_{[\E^{\times}]} \varphi_2^{\vee}(t_2) \Omega(t_2) d^{\times}t_2 \\
	&+ \frac{\Vol([\E^{\times}])^2}{\Vol([\D^{\times}])} \sideset{}{_{\substack{ \eta \in \widehat{\F^{\times} \backslash \A^{\times}} \\ \eta^2 = \omega }}} \sum \left( \int_{\D^{\times}(\A)} \Psi(g) \eta(\nu_{\D}(g)) \norm[\nu_{\D}(g)]_{\A}^{s_0+1} dg \right) \cdot \mathbbm{1}_{\Omega = \eta \circ \nu_{\E}} \\
	&= M(s_0, \Psi) + DS(s_0, \Psi).
\end{align*}
	The terms on the right hand side have meromorphic continuation to $s_0 \in \C$ as we explained in \S \ref{SpTVar}.

	\subsection{Second Moment Side}
	
	The group action of $H = (\E^{\times} \times \E^{\times})/\F^{\times}$ on $\E^2 \simeq \D, (x,y) \mapsto x+yj$ is identified as
	$$ H \times \E^2 \to \E^2, \quad (t_1,t_2) \times (x_1,x_2) \mapsto (t_1^{-1}t_2 x_1, t_1^{-1}\bar{t}_2 x_2). $$
	The above action of $H$ is identified as the action of
	$$ \widetilde{H} := \left\{ (t_1,t_2) \in \E^{\times} \times \E^{\times} \ \middle| \ t_1 t_2^{-1} \in \E^1 \right\} < \E^{\times} \times \E^{\times} $$
by Hilbert's Theorem 90, where $\E^{\times} \times \E^{\times}$ acts component by component on $\E^2$ as
	$$ \left( \E^{\times} \times \E^{\times} \right) \times \E^2 \to \E^2, \quad (t_1,t_2) \times (x_1, x_2) \mapsto (t_1x_1,t_2x_2). $$
	Our chosen character of $H$ defined in (\ref{CharH}) is identified with the character on $\widetilde{H}$
\begin{equation}
	\Omega_{\widetilde{H}}: \widetilde{H}(\A) \to \C^{\times}, \quad (t_1,t_2) \mapsto \Omega(t_1).
\label{CharHVar}
\end{equation}
	Since $\D^{\times}(\F) = \left( \E + \E j \right) - \{ 0 \}$, we have the orbital decomposition for the action of $\E^{\times} \times \E^{\times}$
	$$ \D^{\times}(\F) = (\E^{\times} + \E^{\times}j) \bigsqcup \E^{\times} \bigsqcup \E^{\times}j. $$
	Consequently, we have another decomposition in the region $\Re s_0 > 1$ of
\begin{align*}
	\Theta_{\D}(s_0, \Psi) &= \int_{[\E^{\times}] \times \E^{\times} \backslash \A_{\E}^{\times}} \left( \sum_{\alpha,\beta \in \E^{\times}} \Psi(t_1^{-1}t_2\alpha+t_1^{-1}\bar{t}_2\beta j) \right) \Omega(t_1^{-1} t_2) \norm[t_1^{-1}t_2]_{\A_{\E}}^{s_0+1} d^{\times}t_1 d^{\times}t_2 + \\
	&\quad \int_{[\E^{\times}] \times \E^{\times} \backslash \A_{\E}^{\times}} \left( \sum_{\alpha \in \E^{\times}} \Psi(t_1^{-1}t_2 \alpha) \right) \Omega(t_1^{-1}t_2) \norm[t_1^{-1}t_2]_{\A_{\E}}^{s_0+1} d^{\times}t_1 d^{\times}t_2 + \\
	&\quad \int_{[\E^{\times}] \times \E^{\times} \backslash \A_{\E}^{\times}} \left( \sum_{\beta \in \E^{\times}} \Psi(t_1^{-1}\bar{t}_2 \beta) \right) \Omega(t_1^{-1}t_2) \norm[t_1^{-1}t_2]_{\A_{\E}}^{s_0+1} d^{\times}t_1 d^{\times}t_2 \\
	&= \int_{\E^{\times} \backslash \A_{\E}^{\times} \times \E^1(\F) \backslash \E^1(\A)} \left( \sum_{\alpha,\beta \in \E^{\times}} \Psi(t_1\alpha+t_1t_2\beta j) \right) \Omega(t_1) \norm[t_1]_{\A_{\E}}^{s_0+1} d^{\times}t_1 d^{\times}t_2 + \\
	&\quad \Vol([\E^{\times}]) \cdot \int_{\A_{\E}^{\times}} \Psi(t) \Omega(t) \norm[t]_{\A_{\E}}^{s_0+1} d^{\times}t_1 + \Vol([\E^{\times}]) \cdot \int_{\A_{\E}^{\times}} \Psi(tj) \Omega(t) \norm[t]_{\A_{\E}}^{s_0+1} d^{\times}t_1 \cdot \id_{\Omega \mid_{\E^1(\A)}=\id}.
\end{align*}
	The second resp. third term is simply $DG_1(s_0, \Psi)$ resp. $DG_2(s_0, \Psi)$, which has obvious meromorphic continuation to $s_0 \in \C$ regular at $s_0=0$. For the first term, we have by Proposition \ref{PoissonSEst}
	$$ \sum_{\alpha,\beta \in \E^{\times}} \extnorm{ \Psi(t_1\alpha+t_1t_2\beta j) } \ll_{p,N} \min(\norm[t_1]_{\A_{\E}}^{-1}, \norm[t_1]_{\A_{\E}}^{-N}) \norm[t_1t_2]_{\A_{\E}}^{-p}, \quad \forall p,N \geq 1. $$
	Thus for any $p > 1$ and $\Re s_0 > 1+p$, the function
	$$ g: \E^{\times} \backslash \A_{\E}^{\times} \to \C, \quad g(t) = \int_{\E^{\times} \backslash \A_{\E}^{\times} \times \E^1(\F) \backslash \E^1(\A)} \left( \sum_{\alpha,\beta \in \E^{\times}} \Psi(t_1\alpha+t_1t_2 t \beta j) \right) \Omega(t_1) \norm[t_1]_{\A_{\E}}^{s_0+1} d^{\times}t_1 d^{\times}t_2 $$
	is smooth, invariant by $\E^1(\F) \backslash \E^1(\A) < \E^{\times} \backslash \A_{\E}^{\times}$ and satisfies the bound
	$$ \norm[g(t)] \ll_p \min(\norm[t]_{\A}^{-1}, \norm[t]_{\A}^{-p}). $$
	The same bound holds if we replace $g(t)$ by $D^n.g(t)$ for the invariant differential $D$ on $s_{\E}(\R_+)$, where $s_{\E}$ is a/any section map of the adelic norm map on $\E^{\times} \backslash \A_{\E}^{\times}$. Thus $g(t)$ is Mellin invertible with
	$$ g(1) = \frac{1}{\Vol(\E^{\times} \R_+ \E^1(\A) \backslash \A_{\E}^{\times})} \sum_{\Xi \in \widehat{\E^{\times} \R_+ \E^1(\A) \backslash \A_{\E}^{\times}}} \int_{(c)} \int_{\E^{\times} \E^1(\A) \backslash \A_{\E}^{\times}} g(t) \Xi(t) \norm[t]_{\A_{\E}}^s d^{\times}t \frac{ds}{2\pi i}, $$
	where $1 < c < \Re s_0 - 1$ is arbitrary. Write $\Omega = \Omega_0 \norm[\cdot]_{\A_{\E}}^{i\tau_0}$ for a unique $\tau_0 \in \R$ such that $\Omega_0 \mid_{\R_+} = 1$. Moreover, we can rewrite the above double integral with a contour shift to $c' > \Re s_0 + 1$, obtaining
\begin{align*}
	&\quad \int_{(c)} \int_{\E^{\times} \E^1(\A) \backslash \A_{\E}^{\times}} g(t) \Xi(t) \norm[t]_{\A_{\E}}^s d^{\times}t \frac{ds}{2\pi i} \\
	&= \int_{(c)} \int_{\A_{\E}^{\times} \times \A_{\E}^{\times}} \Psi(t_1+t_2j) \Omega \Xi^{-1}(t_1)\norm[t_1]_{\A_{\E}}^{s_0+1-s} \Xi(t_2) \norm[t_2]_{\A}^s d^{\times}t_1 d^{\times}t_2 \frac{ds}{2\pi i} \\
	&= \int_{(c')} \int_{\A_{\E}^{\times} \times \A_{\E}^{\times}} \Psi(t_1+t_2j) \Omega \Xi^{-1}(t_1)\norm[t_1]_{\A_{\E}}^{s_0+1-s} \Xi(t_2) \norm[t_2]_{\A}^s d^{\times}t_1 d^{\times}t_2 \frac{ds}{2\pi i} + \\
	&\quad \zeta_{\E}^* \id_{\Xi=\Omega_0} \cdot \int_{\A_{\E}^{\times}} \left( \int_{\A_{\E}} \Psi(x_1+t_2j) dx_1 \right) \Omega(t_2) \norm[t_2]_{\A_{\E}}^{s_0} d^{\times}t_2 - \zeta_{\E}^* \id_{\Xi=\Omega_0} \cdot \int_{\A_{\E}^{\times}} \Psi(t_2j) \Omega(t_2) \norm[t_2]_{\A_{\E}}^{s_0+1} d^{\times}t_2.
\end{align*}
	All three terms have meromorphic continuation to $s_0 \in \C$. We have proved the meromorphic continuation of $\Theta_{\D}(s_0,\Psi)$ to $s_0 \in \C$.
	
	Assuming $\norm[s_0] < 1/4$, we shift the contour to get
\begin{align*}
	&\quad \int_{(c')} \int_{\A_{\E}^{\times} \times \A_{\E}^{\times}} \Psi(t_1+t_2j) \Omega \Xi^{-1}(t_1)\norm[t_1]_{\A_{\E}}^{s_0+1-s} \Xi(t_2) \norm[t_2]_{\A}^s d^{\times}t_1 d^{\times}t_2 \frac{ds}{2\pi i} \\
	&= \int_{(\frac{1}{2})} \int_{\A_{\E}^{\times} \times \A_{\E}^{\times}} \Psi(t_1+t_2j) \Omega \Xi^{-1}(t_1)\norm[t_1]_{\A_{\E}}^{s_0+1-s} \Xi(t_2) \norm[t_2]_{\A}^s d^{\times}t_1 d^{\times}t_2 \frac{ds}{2\pi i} + \\
	&\quad \zeta_{\E}^* \id_{\Xi=\Omega_0} \cdot \int_{\A_{\E}^{\times}} \Psi(t_2j) \Omega(t_2) \norm[t_2]_{\A_{\E}}^{s_0+1} d^{\times}t_2 + \zeta_{\E}^* \id_{\Xi = \id} \cdot \int_{\A_{\E}^{\times}} \left( \int_{\A_{\E}} \Psi(t_1+x_2j) dx_2 \right) \Omega(t_1) \norm[t_1]_{\A_{\E}}^{s_0} d^{\times}t_1.
\end{align*}
	We thus get
\begin{align*}
	\Theta_{\D}(s_0,\Psi) &= \frac{1}{\Vol([\E^1 \backslash \E^{\times}])} \sum_{\Xi \in \widehat{[\E^1 \backslash \E^{\times}]}} \int_{(\frac{1}{2})} \int_{\A_{\E}^{\times} \times \A_{\E}^{\times}} \Psi(t_1+t_2j) \Omega \Xi^{-1}(t_1)\norm[t_1]_{\A_{\E}}^{s_0+1-s} \Xi(t_2) \norm[t_2]_{\A}^s d^{\times}t_1 d^{\times}t_2 \frac{ds}{2\pi i} \\
	&+ \frac{\zeta_{\E}^*}{\Vol([\E^1 \backslash \E^{\times}])} \int_{\A_{\E}^{\times}} \left( \int_{\A_{\E}} \Psi(t_1+x_2j) dx_2 \right) \Omega(t_1) \norm[t_1]_{\A_{\E}}^{s_0} d^{\times}t_1 \\
	&+ \frac{\zeta_{\E}^*}{\Vol([\E^1 \backslash \E^{\times}])} \id_{\Omega \mid_{\E^1(\A)} = \id} \cdot \int_{\A_{\E}^{\times}} \left( \int_{\A_{\E}} \Psi(x_1+t_2j) dx_1 \right) \Omega(t_2) \norm[t_2]_{\A_{\E}}^{s_0} d^{\times}t_2 \\
	&+ \Vol([\E^{\times}]) \cdot \int_{\A_{\E}^{\times}} \Psi(t) \Omega(t) \norm[t]_{\A_{\E}}^{s_0+1} d^{\times}t_1 + \Vol([\E^{\times}]) \cdot \int_{\A_{\E}^{\times}} \Psi(t) \Omega(t) \norm[t]_{\A_{\E}}^{s_0+1} d^{\times}t_1 \cdot \id_{\Omega \mid_{\E^1(\A)}=\id} \\
	&= M_2(s_0,\Psi) + \sum_{j=1}^4 DG_j(s_0,\Psi).
\end{align*}
	Note that the terms containing $\id_{\Omega \mid_{\E^1(\A)} = \id}$ occur only if we can take $\Xi \norm[\cdot]_{\A_{\E}}^s = \Omega \norm[\cdot]_{\A_{\E}}^{s_0}$, in which case the restriction of $\Omega$ to $\mathbf{E}^1(\A)$ necessarily has to be trivial, since $\Xi$ is so.

\section{Appendix: Comparison with Period Approach}

	\subsection{Recall of Period Approach}
	
	For simplicity, we consider the case where $\omega=\chi_1=\chi_2=\mathbbm{1}$ are trivial character. The period approach, proposed by Michel-Venkatesh, considers the regularized integral 
	$$ \int_{\F^{\times} \backslash \A^{\times}}^{\reg} \eis(s_1,f_1) \cdot \eis(s_2,f_2) (a(t)) d^{\times}t $$
	along the diagonal torus of the product of two Eisenstein series constructed from smooth vectors $f_1,f_2 \in \Bas(\mathbbm{1}, \mathbbm{1})$. One expects a suitable Fourier inversion formula for this product, so that the projection to the space of a cuspidal representation $\pi$ gives the contribution
\begin{equation} 
	\sideset{}{_{\varphi \in \Bas(\pi)}} \sum \Pairing{\eis(s_1,f_1) \cdot \eis(s_2,f_2)}{\varphi} \int_{\F^{\times} \backslash \A^{\times}} \varphi(a(t)) d^{\times}t. 
\label{PerProj}
\end{equation}
	By Hecke-Jacuqet-Langlands' theory, the above integral represents $L(1/2, \pi)$. By the Rankin-Selberg theory, the automorphic Fourier coefficient represents $L(1/2+s_1,\pi^{\vee}) L(1/2+s_2, \pi^{\vee})$. Hence (\ref{PerProj}) represents a certain third moment. This approach has been made rigorous by Nelson \cite{Ne20}. Precisely, Nelson
\begin{itemize} 
	\item defined the above regularized integral as the special value at $s=0$ of the following integral (equivalent to $\ell_s$ defined by \cite[(10.3)]{Ne20})
	$$ I(s, s_1,s_2) := \int_{\F^{\times} \backslash \A^{\times}} \left( \eis(s_1,f_1) \cdot \eis(s_2,f_2) - \eisCst(s_1,f_1) \cdot \eisCst(s_2,f_2) \right) (a(t)) \norm[t]_{\A}^s d^{\times}t, $$
which converges for $\Re s \gg 1$ (depending on $s_1,s_2$) and has meromorphic continuation\footnote{In fact, this follows also from the theory of zeta function for finitely regularizable functions in \cite[\S 2.3]{Wu2}.} to $s \in \C$; 
	\item restricted to $\norm[s_1], \norm[s_2]$ small; 
	\item expanded $I(s,s_1,s_2)$ in two different ways and analytically continue the obtained equation to $s=s_1=s_2=0$ to deduce the relevant Motohashi-type formula.
\end{itemize} 
	
	We compare the period approach with our approach in this section. The comparison seems to be more convenient in the followng region
	$$ \left\{ (s_1,s_2) \in \C^2 \ \middle| \ \Re s_1, \Re s_2 - \Re s_1 > 1/2 \right\} $$
	instead of the original region\footnote{We do not know how to work directly with the original region. Nor can we unfortunately match the degenerate terms with Nelson's version.} considered by Nelson \cite[Theorem 10.2]{Ne20}, where $\norm[s_1]$ and $\norm[s_2]$ are small. The proofs are easy computation hence left to the reader.

	\subsection{Comparison of Geometric Sides}
	
	Recall the Godement section constructed from $\Phi \in \Sch(\A^2)$, absolutely convergent for $\Re s > 0$
	$$ f_{\Phi}(s,g) := \norm[\det g]_{\A}^{\frac{1}{2}+s} \int_{\A^{\times}} \Phi((0,t)g) \norm[t]_{\A}^{1+2s} d^{\times}t. $$
\begin{lemma}
	Let $\eis(s,\Phi)$ be the Eisenstein series formed from $f_{\Phi}(s, \cdot)$. Then its constant term is 
	$$ \eisCst(s,\Phi)(g) = f_{\Phi}(s,g) + f_{\widehat{\Phi}}(-s,g), $$
	$$ \text{with} \quad f_{\widehat{\Phi}}(-s,g) = \norm[\det g]_{\A}^{\frac{1}{2}+s} \int_{\A^{\times}} \OFour_2 \rpL_{t^{-1}} \rpR_g \Phi(1,0) \norm[t]_{\A}^{1+2s} d^{\times}t, $$
	where by abuse of notations we have denoted by $\rpL$ (and $\rpR$) the action
	$$ \rpL_t \rpR_g \Phi(x,y) := \Phi((t^{-1}x, t^{-1}y)g). $$
	We also have the difference, absolutely convergent for all $s \in \C$
	$$ \left( \eis(s,\Phi) - \eisCst(s,\Phi) \right)(g) = \norm[\det g]_{\A}^{\frac{1}{2}+s} \int_{\F^{\times} \backslash \A^{\times}} \left( \sum_{\alpha_1, \alpha_2 \in \F^{\times}} \OFour_2 \rpL_{t^{-1}} \rpR_g \Phi (\alpha_1, \alpha_2) \right) \norm[t]_{\A}^{1+2s} d^{\times}t. $$
\label{GSEisFourInv}
\end{lemma}
\begin{corollary}
	Let $\Phi_1, \Phi_2 \in \Sch(\A^2)$ and define $\Psi \in \Sch(\Mat_2(\A))$ by
	$$ \Psi \begin{pmatrix} x_1 & x_2 \\ x_3 & x_4 \end{pmatrix} := \Phi_1(x_1,x_2) \Phi_2(x_3, x_4). $$
	Then for $\Re s \gg 1$, we have a decomposition of 
	$$ \tilde{I}(s,s_1,s_2) := \int_{\F^{\times} \backslash \A^{\times}} \left( \eis(s_1,\Phi_1) \cdot \eis(s_2,\Phi_2) - \eisCst(s_1,\Phi_1) \cdot \eisCst(s_2,\Phi_2) \right) (a(t)) \norm[t]_{\A}^s d^{\times}t $$
	as $\tilde{I}(s,s_1,s_2) = \sideset{}{_{j=0}^4} \sum \tilde{I}_j(s,s_1,s_2)$, where (recall the convention Remark \ref{SumConv}),
	$$ \tilde{I}_0(\cdot) = \int_{(\F^{\times} \backslash \A^{\times})^3} \Vsum \OFour_2 \OFour_4 \rpL_{d(t_1,t_2)^{-1}} \rpR_{a(t)} \Psi\begin{pmatrix} \xi_1 & \xi_2 \\ \xi_3 & \xi_4 \end{pmatrix} \norm[t_1]_{\A}^{1+2s_1} \norm[t_2]_{\A}^{1+2s_2} \norm[t]_{\A}^{1+s_1+s_2+s} d^{\times}t_1 d^{\times}t_2 d^{\times}t, $$
	$$ \tilde{I}_1(\cdot) = \int_{(\F^{\times} \backslash \A^{\times})^3} \Vsum \OFour_2 \rpL_{d(t_1,t_2)^{-1}} \rpR_{a(t)} \Psi\begin{pmatrix} \xi_1 & \xi_2 \\ 0 & \xi_4 \end{pmatrix} \norm[t_1]_{\A}^{1+2s_1} \norm[t_2]_{\A}^{1+2s_2} \norm[t]_{\A}^{1+s_1+s_2+s} d^{\times}t_1 d^{\times}t_2 d^{\times}t, $$
	$$ \tilde{I}_2(\cdot) = \int_{(\F^{\times} \backslash \A^{\times})^3} \Vsum \OFour_2 \OFour_4 \rpL_{d(t_1,t_2)^{-1}} \rpR_{a(t)} \Psi\begin{pmatrix} \xi_1 & \xi_2 \\ \xi_3 & 0 \end{pmatrix} \norm[t_1]_{\A}^{1+2s_1} \norm[t_2]_{\A}^{1+2s_2} \norm[t]_{\A}^{1+s_1+s_2+s} d^{\times}t_1 d^{\times}t_2 d^{\times}t, $$
	$$ \tilde{I}_3(\cdot) = \int_{(\F^{\times} \backslash \A^{\times})^3} \Vsum \OFour_4 \rpL_{d(t_1,t_2)^{-1}} \rpR_{a(t)} \Psi\begin{pmatrix} 0 & \xi_2 \\ \xi_3 & \xi_4 \end{pmatrix} \norm[t_1]_{\A}^{1+2s_1} \norm[t_2]_{\A}^{1+2s_2} \norm[t]_{\A}^{1+s_1+s_2+s} d^{\times}t_1 d^{\times}t_2 d^{\times}t, $$
	$$ \tilde{I}_4(\cdot) = \int_{(\F^{\times} \backslash \A^{\times})^3} \Vsum \OFour_2 \OFour_4 \rpL_{d(t_1,t_2)^{-1}} \rpR_{a(t)} \Psi\begin{pmatrix} \xi_1 & 0 \\ \xi_3 & \xi_4 \end{pmatrix} \norm[t_1]_{\A}^{1+2s_1} \norm[t_2]_{\A}^{1+2s_2} \norm[t]_{\A}^{1+s_1+s_2+s} d^{\times}t_1 d^{\times}t_2 d^{\times}t. $$
\end{corollary}

	We apply to $\tilde{I}_0(s,s_1,s_2)$ the method leading to the decomposition (\ref{II_1AC}), and will get four other degenerate terms. Together with $\tilde{I}_j(s,s_1,s_2)$ for $1 \leq j \leq 4$, these eight degenerate terms correspond precisely to our $DG_j$ for $1 \leq j \leq 8$ given in (\ref{II_11AC}) to (\ref{II_24CE}). The geometric side of the period approach does not contain terms corresponding to our $II_3$ \& $II_4$, which are moved to the spectral side.

	\subsection{Comparison of Spectral Sides}
	
	The relation of $I(s,s_1,s_2)$ or $\tilde{I}(s,s_1,s_2)$ to the third moment of $L$-functions for $\GL_2$ is given by the automorphic Fourier inversion of
	$$ \eis(s_1,\Phi_1) \cdot \eis(s_2,\Phi_2) - \widetilde{\Reis}, $$
where $\widetilde{\Reis} = \widetilde{\Reis}_1 + \widetilde{\Reis}_2$ and $\widetilde{\Reis}_j$ are Eisenstein series constructed from
	$$ f_{\Phi_1}(s_1, \cdot ) f_{\Phi_2}(s_2, \cdot), \quad \text{resp.} \quad f_{\widehat{\Phi}_1}(-s_1, \cdot) f_{\Phi_2}(s_2, \cdot). $$
	First, we notice that for $\Re s >  1/2$
	$$ \eis(s,\Phi)(g) = \sum_{\gamma \in \gp{B}(\F) \backslash \GL_2(\F)} f_{\Phi}(s, \gamma g) = \norm[\det g]_{\A}^{\frac{1}{2}+s} \int_{\F^{\times} \backslash \A^{\times}} \sum_{\vec{\alpha} \in \F^2 - \{ (0,0) \}} \Phi(t \vec{\alpha} g) \norm[t]_{\A}^{1+2s} d^{\times}t. $$
	Hence we get
	$$ \left( \eis(s_1,\Phi_1) \eis(s_2,\Phi_2) \right)(g) = \norm[\det g]_{\A}^{1+s_1+s_2} \int_{(\F^{\times} \backslash \A^{\times})^2} \sum_{\xi \in \Mat_2^*(\F)} \rpL_{d(t_1,t_2)^{-1}} \rpR_g \Psi(\xi) \norm[t_1]_{\A}^{1+2s_1} \norm[t_2]_{\A}^{1+2s_2} d^{\times}t_1 d^{\times}t_2, $$
	where $\Mat_2^*(\F) \subset \Mat_2(\F)$ is defined by
	$$ \Mat_2^*(\F) = \left\{ \xi = \begin{pmatrix} \vec{\xi}_1 \\ \vec{\xi}_2 \end{pmatrix} \in \Mat_2(\F) \ \middle| \ \vec{\xi}_j \neq (0,0), j=1,2 \right\}. $$
	By definition and the formulas in Lemma \ref{GSEisFourInv}, it is easy to deduce that $\widetilde{\Reis}_1$ and $\widetilde{\Reis}_2$ are given by
	$$ \norm[\det g]_{\A}^{1+s_1+s_2} \int_{(\F^{\times} \backslash \A^{\times})^2} \sum_{\gamma \in \gp{B}(\F) \backslash \GL_2(\F)} \Vsum \rpL_{d(t_1,t_2)^{-1}} \rpR_{\gamma g} \Psi \begin{pmatrix} 0 & \xi_2 \\ 0 & \xi_4 \end{pmatrix} \norm[t_1]_{\A}^{1+2s_1} \norm[t_2]_{\A}^{1+2s_2} d^{\times}t_1 d^{\times}t_2, $$
	$$ \norm[\det g]_{\A}^{1+s_1+s_2} \sum_{\gamma \in \gp{B}(\F) \backslash \GL_2(\F)}  \int_{(\F^{\times} \backslash \A^{\times})^2} \Vsum \OFour_2 \rpL_{d(t_1,t_2)^{-1}} \rpR_{\gamma g} \Psi \begin{pmatrix} \xi_1 & 0 \\ 0 & \xi_4 \end{pmatrix} \norm[t_1]_{\A}^{1+2s_1} \norm[t_2]_{\A}^{1+2s_2} d^{\times}t_1 d^{\times}t_2. $$
	It is easy to verify
	$$ \bigcup_{\xi_2, \xi_4 \in \F^{\times}} \begin{pmatrix} 0 & \xi_2 \\ 0 & \xi_4 \end{pmatrix} \GL_2(\F) = \Mat_2^*(\F) - \GL_2(\F). $$
	Hence we identify the following difference as an absolutely convergent integral
\begin{align*}
	&\quad \left( \eis(s_1,\Phi_1) \cdot \eis(s_2,\Phi_2) - \widetilde{\Reis}_1 \right)(g) \\
	&= \norm[\det g]_{\A}^{1+s_1+s_2} \int_{(\F^{\times} \backslash \A^{\times})^2} \sum_{\xi \in \GL_2(\F)} \rpL_{d(t_1,t_2)^{-1}} \rpR_g \Psi(\xi) \norm[t_1]_{\A}^{1+2s_1} \norm[t_2]_{\A}^{1+2s_2} d^{\times}t_1 d^{\times}t_2 \\
	&= \norm[\det g]_{\A}^{1+s_1+s_2} \int_{(\F^{\times} \backslash \A^{\times})^2} \sum_{\xi \in \GL_2(\F)} \rpL_{a(t_1)^{-1}} \rpR_{d(z,z)} \rpR_g \Psi(\xi) \norm[t_1]_{\A}^{1+2s_1} \norm[z]_{\A}^{2(1+s_1+s_2)} d^{\times}z d^{\times}t_1 \\
	&= \int_{\F^{\times} \backslash \A^{\times}} \Tree{\widetilde{KK}^{(2)}}{\mathbbm{1},2(s_1+s_2)}{a(t_1)}{g} \norm[t_1]_{\A}^{s_2-s_1} d^{\times}t_1
\end{align*}
	where $\Shoulder{KK^{(2)}}{x}{y}$ is defined in (\ref{KKDecomp}). In order to rewrite $\widetilde{\Reis}_2$, we notice that
	$$ \Vsum \OFour_2 \Psi \begin{pmatrix} \xi_1 & 0 \\ 0 & \xi_4 \end{pmatrix} = \Vsum \int_{\A} \Psi \begin{pmatrix} \xi_1 & u \\ 0 & \xi_4 \end{pmatrix} du = \int_{\F \backslash \A} \Vsum_{\xi_2 \in \F} \Psi \left( n(u) \begin{pmatrix} \xi_1 & \xi_2 \\ 0 & \xi_4 \end{pmatrix} \right) du, $$
	which implies readily
	$$ \sum_{\gamma \in \gp{B}(\F) \backslash \GL_2(\F)} \Vsum \OFour_2 \rpL_{d(t_1,t_2)^{-1}} \rpR_{\gamma g} \Psi \begin{pmatrix} \xi_1 & 0 \\ 0 & \xi_4 \end{pmatrix} = \int_{\F \backslash \A} \sum_{\xi \in \GL_2(\F)} \rpL_{n(u)} \rpL_{d(t_1,t_2)^{-1}} \rpR_g \Psi \left( \xi \right) du. $$
	It follows that
	$$ \left( \eis(s_1,\Phi_1) \cdot \eis(s_2,\Phi_2) - \widetilde{\Reis} \right)(g) = \int_{\F^{\times} \backslash \A^{\times}} \Tree{\widetilde{\Delta K}^{(2)}}{\mathbbm{1},2(s_1+s_2)}{a(t_1)}{g} \norm[t_1]_{\A}^{s_2-s_1} d^{\times}t_1, $$
	whose contribution to the regularized integral \cite[(10.6)]{Ne20}
\begin{align*}
	&\quad \int_{\F^{\times} \backslash \A^{\times}} \left( \left( \eis(s_1,\Phi_1) \cdot \eis(s_2,\Phi_2) - \widetilde{\Reis} \right) - \left( \eis(s_1,\Phi_1) \cdot \eis(s_2,\Phi_2) - \widetilde{\Reis} \right)_{\gp{N}} \right)(a(y)) \norm[y]_{\A}^s d^{\times}y \\
	&= \int_{(\F^{\times} \backslash \A^{\times})^2} \Tree{\widetilde{\Delta \Delta}^{(2)}}{\mathbbm{1},2(s_1+s_2)}{a(t_1)}{a(t_2)} \norm[t_1]_{\A}^{s_2-s_1} \norm[t_2]_{\A}^s d^{\times}t_1 d^{\times}t_2
\end{align*}
	is the essential part of the main Motohashi distribution we considered in Proposition \ref{GoalMotoDis}.

\section*{Acknowledgement}
	
	The author would like to thank Olga Balkanova, Dmitry Frolenkov, Peter Humphries, Paul Nelson, Zhi Qi for discussions, Valentin Blomer for some complement on the history of (inverting) Motohashi's formula. The author would also like to give special thanks to Prof. Binyong Sun and Lei Zhang for useful comments on an earlier version. The preparation of this paper scatters during the author's stays at Alfr\'ed R\'enyi Institute in Hungary supported by the MTA R\'enyi Int\'ezet Lend\"ulet Automorphic Research Group and in the School of Mathematical Sciences at Queen Mary University of London. The author would like to thank both institutes for their hospitality and the support of the Leverhulme Trust Research Project Grant RPG-2018-401. Last but not least, the author thank two anonymous referees for their suggestions, which improve the readability and rigor of the paper.

\bibliographystyle{acm}

\bibliography{mathbib}

\address{\quad \\ Han WU \\ School of Mathematical Sciences \\ Queen Mary University of London \\ Mile End Road \\ E1 4NS, London \\ United Kingdom \\ wuhan1121@yahoo.com}

\end{document}